\documentclass[12pt,reqno]{amsart}
\pdfoutput=1
\usepackage{amsfonts}
\usepackage{amssymb}
\usepackage{amsmath}
\usepackage{tikz-cd}
\usepackage{centernot}
\usepackage{enumerate}
\usepackage{relsize}
\usepackage{scalerel}
\usepackage{bm}
\usepackage{amsrefs}

\usepackage{fullpage}

\newtheorem{theorem}{Theorem}[section]
\newtheorem{lemma}[theorem]{Lemma}
\newtheorem{proposition}[theorem]{Proposition}

\theoremstyle{definition}
\newtheorem{definition}[theorem]{Definition}

\theoremstyle{remark}
\newtheorem{remark}[theorem]{}

\theoremstyle{definition}
\newtheorem*{convention}{Convention}

\newcommand{\C}{\mathcal C}

\newcommand{\F}{\mathcal F}

\renewcommand{\H}{\mathcal H}

\renewcommand{\P}{\mathcal P}
\newcommand{\Q}{\mathcal Q}
\newcommand{\R}{\mathcal R}

\newcommand{\U}{\mathcal U}

\newcommand{\W}{\mathcal W}
\newcommand{\X}{\mathcal X}
\newcommand{\Y}{\mathcal Y}
\newcommand{\Z}{\mathcal Z}

\newcommand{\BB}{\mathbb B}
\newcommand{\CC}{\mathbb C}

\newcommand{\NN}{\mathbb N}

\newcommand{\QQ}{\mathbb Q}
\newcommand{\RR}{\mathbb R}

\newcommand{\CCC}{\mathbf C}

\newcommand{\<}{\langle}
\renewcommand{\>}{\rangle}
\renewcommand{\:}{\colon}
\newcommand{\To}{\rightarrow}
\newcommand{\inv}{^{-1}}
\newcommand{\subsetof}{\subseteq}

\newcommand{\suchthat}{\,|\,}
\newcommand{\oast}{\circledast}
\newcommand{\tensor}{\otimes}
\newcommand{\iso}{\cong}

\newcommand{\Yields}{\;\Rightarrow\;}
\newcommand{\Equivalent}{\;\Leftrightarrow\;}
\newcommand{\Union}{\bigcup}
\newcommand{\union}{\cup}

\newcommand{\comp}{\text{$\sim$}}
\newcommand{\atomof}{\text{\raisebox{1.2pt}{\smaller\,$\propto$\,}}}
\DeclareMathOperator*{\bigast}{\scalerel*{\ast}{\sum}}
\DeclareMathOperator*{\bigoast}{\scalerel*{\circledast}{\bigoplus}}

\newcommand{\Hom}{\mathrm{Hom}}
\newcommand{\Tr}{\mathrm{Tr}}
\newcommand{\At}{\mathrm{At}}

\newcommand{\Set}{\mathbf{Set}}
\newcommand{\qSet}{\mathbf{qSet}}

\newcommand{\qPar}{\mathbf{qPar}}
\newcommand{\Rel}{\mathbf{Rel}}
\newcommand{\qRel}{\mathbf{qRel}}
\newcommand{\FdHilb}{\mathbf{FdHilb}}

\newcommand{\Mstar}[1]{\mathbf{M}^*_{#1}}

\allowdisplaybreaks

\begin{document}

\title{Quantum sets}
\author{Andre Kornell}
\address{Department of Mathematics, University of California, Davis, Davis, California 95616}
\email{kornell@math.ucdavis.edu}

\begin{abstract}
A quantum set is defined to be simply a set of nonzero finite-dimensional Hilbert spaces. Together with binary relations, essentially the quantum relations of Weaver, quantum sets form a dagger compact category. Functions between quantum sets are certain binary relations that can be characterized in terms of this dagger compact structure, and the resulting category of quantum sets and functions generalizes the category of ordinary sets and functions in the manner of noncommutative mathematics. In particular, this category is dual to a subcategory of von Neumann algebras. The basic properties of quantum sets are presented thoroughly, with the noncommutative dictionary in mind, and with an eye to convenient application. As a motivating example, a notion of quantum graph coloring is derived within this framework, and it is shown to be equivalent to the notion that appears in the quantum information theory literature.
\end{abstract}

\maketitle

\section{Introduction}\label{section 1}

This paper concerns the quantum generalization of sets, in the sense of noncommutative geometry. It does not concern the quantization of sets as a primitive mathematical notion, nor the quantization of the familiar cumulative hierarchy described by Zermelo-Fraenkel set theory, as in \cite{Takeuti81} \cite{Ozawa07}. Rather, it concerns the quantization of sets as trivial structures, that is, as structures without structure. The category of quantum sets and functions described here generalizes the category of ordinary sets and functions in the same sense that the opposite of the category of unital C*-algebras and unital $*$-homomorphisms generalizes the category of ordinary compact Hausdorff spaces and continuous maps \cite{GarciaBondiaVarillyFigueroa}. 

Ordinary sets can be identified with discrete topological spaces, with atomic measure spaces, and indeed, with many elementary examples from familiar classes of structures. The identification of sets with discrete topological spaces is particularly significant to noncommutative mathematics, because the quantum generalization of locally compact Hausdorff topological spaces is well established; it is the class of C*-algebras. Likewise, the identification of sets with atomic measure spaces is also particularly significant, because von Neumann algebras are a quantum generalization of suitably well-behaved measure spaces \cite{Fremlin}*{343B}. To simplify the discussion, we consider measure spaces up to equivalence of measure, and up to subsets of measure zero, so that indeed, atomic measure spaces correspond exactly to sets. Thus, up to isomorphism, we expect quantum sets to be in canonical bijective correspondence with a class of von Neumann algebras, whose commutative members are exactly the von Neumann algebras of ordinary atomic measure spaces. Likewise for C*-algebras and ordinary discrete topological spaces.

The class of atomic, or synonymously, fully atomic von Neumann algebras is the obvious candidate for a canonical quantum generalization of atomic measure spaces, but it is not the quantum generalization used in this paper. Recall that a von Neumann algebra is said to be atomic if and only if every nonzero projection is above a minimal nonzero projection. It is well known that a von Neumann algebra is atomic if and only if it is an $\ell^\infty$-direct sum of type I factors \cite{Blackadar17}*{IV.2.2.1}. As a quantum generalization of atomic measure spaces, this class of von Neumann algebras suffers from a number of related problems, each of which takes some time to expound. One such problem is that an ultraweakly closed unital $*$-subalgebra of an atomic von Neumann algebra need not itself be atomic, whereas the quotient of an atomic measure space must be atomic. In the appropriate variant of Gelfand duality, ultraweakly closed unital $*$-subalgebras correspond to quotient measure spaces. The existence of nondiagonalizable self-adjoint operators is another such problem (Theorem \ref{Q}). The quantum generalization of atomic measure spaces that is used in this paper definitionally requires that every ultraweakly closed $*$-subalgebra be atomic. Such a von Neumann algebra is an $\ell^\infty$-direct sum of \emph{finite} type I factors, i.e., of simple matrix algebras.

Essentially the same class of operator algebras arises as the quantum generalization of discrete topological spaces in the theory of locally compact quantum groups. The C*-algebras are the established quantum generalization of locally compact Hausdorff spaces, and locally compact quantum groups are defined to be C*-algebras with additional structure, which specifies a group operation on the Gelfand spectrum of the C*-algebra whenever the C*-algebra is commutative \cite{KustermansVaes00}. Pontryagin duality is a fundamental duality among \emph{abelian} locally compact groups, and research into compact quantum groups has been substantially guided by the goal of generalizing Pontryagin duality to include nonabelian groups. In the context of this generalized duality, the dual of a \emph{nonabelian} locally compact group is a locally compact \emph{quantum} group that is abelian in a suitable sense. It is a basic fact that the Pontryagin dual of an abelian compact group is discrete, and vice versa, and this fact generalizes to the noncommutative setting, naturally yielding a notion of discrete quantum group, and therefore a natural notion of discrete quantum space \cite{DeCommerKasprzakSkalskiSoltan16}. Consequently, a discrete quantum space is defined to be a $c_0$-direct sum of simple matrix algebras.

The same picture of quantum sets emerges in both settings. The {$\ell^\infty$-direct} sum generalizes the disjoint union of well-behaved measure spaces, and the $c_0$-direct sum generalizes the disjoint union of locally compact Hausdorff spaces. Thus, a quantum set apparently consists of atoms, each possessing a positive integer internal dimension. We define a quantum set to be a structure that is fully determined by an ordinary set $\At(\X)$ of nonzero finite-dimensional complex Hilbert spaces. The atomic quantum measure space constructed from $\X$ is the $\ell^\infty$-direct sum 
$\bigoplus_{X \in \At(\X)} L(X)$, and the discrete quantum topological space constructed from $\X$ is the $c_0$-direct sum 
$\bigoplus_{X \in \At(\X)} L(X)$, where $L(X)$ denotes the algebra of all linear operators on a finite-dimensional Hilbert space $X$, which are automatically bounded. We refrain from identifying $\X$ with $\At(\X)$ to avoid ambiguity when we generalize familiar notions from ordinary sets to quantum sets, and to protect the intuition that ordinary sets are a subclass of quantum sets, rather than vice versa. Each ordinary set $S$ can be regarded as a quantum set $`S$ by replacing each of its elements with a one-dimensional atom. 

One of the broader ambitions of this article is to provide a cohesive foundation for noncommutative discrete mathematics, and establishing a uniform system of terminology is a significant part of that goal. Noncommutative mathematics considers many notions that generalize familiar notions to the noncommutative setting, and they all need names. There is a sound and firmly established convention that these names should refer to the notions being generalized, as an aid to intuition and to memory. The modifier ``quantum'' also appears naturally in this context as a way of distinguishing the generalization from the original notion, but often its appearance does more harm than good.

As an example, consider the generalization of the notion ``is equinumerous to''. We briefly suppose that the appropriate generalization of this notion to quantum sets is to equate $\X$ and $\Y$ iff $\sum_{X \in \At(\X)} (\mathrm{dim}(X))^2 = \sum_{Y \in \At(\Y)} (\mathrm{dim}(Y))^2$. In fact, this supposition is well-founded, but because this article is focused on the categorical properties of quantum sets, we cannot give it a worthy defense here, beyond observing that this invariant behaves appropriately for injective and surjective morphisms in the category $\qSet$. However, this proposed generalization illustrates the desirability of two conventions in notation and terminology.

If the equation in the above paragraph holds, shall we say that $\X$ is equinumerous to $\Y$, or that $\X$ is quantum equinumerous to $\Y$? All else being equal, it is preferable to limit the proliferation of the word ``quantum'' in text and speech already saturated with the word. In fact, there is another reason to prefer the unqualified term ``equinumerous'': it implicitly carries the assurance that two ordinary sets are equinumerous in this generalized sense if and only if they are equinumerous in the familiar sense. It is not immediately clear from the terminology that $`\NN$ is not ``quantum equinumerous'' to $`\RR$. After all, there are nonisomorphic simple graphs that are quantum isomorphic. Once we have chosen to use ``equinumerous'' instead of ``quantum equinumerous'', it becomes absolutely necessary to distinguish each quantum set $\X$ from the ordinary set $\At(\X)$ of its atoms. Indeed, quantum sets $\X$ and $\Y$ can fail to be equinumerous even if $\At(\X)$ and $\At(\Y)$ are equinumerous, and vice versa.

\begin{convention}
Use the modifier ``quantum'' when generalizing a class of structures, e.g., quantum sets, quantum groups, quantum graphs, etc. Do not use the modifier ``quantum'' when generalizing a notion that agrees with the ordinary notion on ordinary structures, e.g., equinumerous quantum sets, a homomorphism between quantum groups, a connected quantum graph, etc.
\end{convention}

We will also use the modifier ``quantum'' for variants of ordinary notions that allow some variable to range over a class of quantum structures. For example, the quantum chromatic number of an ordinary graph is the least number of colors for which there exists a quantum family of proper graph colorings (Definition \ref{pseudotelepathy}). 

\subsection{connection to quantum information theory}\label{game} Together with \cite{MustoReutterVerdon}, our category of quantum sets occupies a point of contact between noncommutative mathematics and quantum information theory. One branch of research in quantum information theory considers cooperative games where the two players must coordinate their responses to a referee without communicating with one another. If the two players are allowed to make measurements on quantum systems that are entangled with each other, then they may have a winning strategy even if no winning strategy exists without this aid. This phenomenon is termed ``quantum pseudotelepathy''; its investigation has been surveyed by Brassard, Broadbent, and Tapp \cite{BrassardBroadbentTapp}. They essentially credit Heywood and Redhead with discovering the first such game \cite{HeywoodRedhead}, obtained from the proof of the Kochen-Specker theorem \cite{KochenSpecker}.

We highlight one line of research leading from quantum pseudotelepathy to our quantum sets. It was discovered that one example in quantum communication complexity can be understood as quantum pseudotelepathy for a graph coloring game \cite{BrassardCleveTapp}*{theorem 4} \cite{GalliardWolf} \cite{GalliardTappWolf} \cite{CleveHoyerTonerWatrous}*{4.1}. Such a game is played with a fixed finite simple graph and a fixed finite set of colors, which may or may not suffice to properly color the graph. There are two cooperating players, traditionally named Alice and Bob, and a referee. The referee separates the players, and then queries each player about a vertex on the graph. Both players must respond with a color, and neither player hears the question for, or the answer from, the other player. If the vertices are equal, the players win iff they respond with the same color; if the vertices are adjacent, the players win iff they respond with different colors; if the vertices are neither equal nor adjacent, the players win automatically. These victory conditions are chosen so that a deterministic classical winning strategy for the two players is the same thing as a proper graph coloring. However, if the two players possess quantum systems entangled with each other, then a winning strategy may be available even when no proper graph coloring exists. To implement such a strategy, each player performs a measurement on their quantum system before responding to the referee.

An ordinary graph coloring is a function from vertices to colors, and when a quantum strategy is available, then one such strategy can be described by a family of functions indexed by a quantum set (Proposition \ref{pseudotelepathy}). The first step toward this connection is the result of Cameron, Montanaro, Newman, Severini, and Winter \cite{CameronMontanaroNewmanSeveriniWinter}*{proposition 1} that when a quantum winning strategy is available, there must be a quantum winning strategy that uses \emph{projective} measurement. In our framework, each projective measurement corresponds to a function from a quantum set describing the quantum system, to the classical set of possible outcomes, so in retrospect, this proposition might be viewed as the first appearance of quantum sets in this context. From the perspective of noncommutative mathematics, this result is the quantum analog of the proposition that when a probabilistic classical winning strategy exists, so does a deterministic such strategy.

The next step toward this connection between pseudotelepathy and our quantum sets is taken in the paper of Man\v cinska and Roberson on quantum homomorphisms between ordinary graphs \cite{MancinskaRoberson}. A proper graph coloring can be characterized as a graph homomorphism to the complete simple graph of colors, and the graph coloring game can be generalized to a game for which the deterministic classical strategies are exactly the graph homomorphisms from one given simple graph to another. Man\v cinska and Roberson begin their paper by generalizing the proposition of Cameron et al. to quantum homomorphism games. In doing so, they all but define a quantum homomorphism to be a quantum winning strategy that uses projective measurement \cite{MancinskaRoberson}*{corollary 2.2}. Such a definition is made explicitly by Abramsky, Barbosa, de Silva, and Zapata \cite{AbramskyBarbosaDeSilvaZapata}*{theorem 7}. They define a quantum homomorphism of finite relational structures in terms of projective measurements; their definition is nontrivially compatible \cite{AbramskyBarbosaDeSilvaZapata}*{theorem 16} with the definition implicit in \cite{MancinskaRoberson}*{corollary 2.2}.

The work of Musto, Reutter, and Verdon \cite{MustoReutterVerdon} \cite{MustoReutterVerdon18} cements the analogy. They introduce a category of quantum sets and quantum functions, and define quantum homomorphisms to be functions preserving adjacency \cite{MustoReutterVerdon}*{definition 5.4}. Their quantum sets and their quantum functions are not strictly speaking the same as our quantum sets and functions, but the notions are very closely related. Their quantum sets are essentially our \textit{finite} quantum sets, and our functions are essentially their \textit{one-dimensional} quantum functions. Formally, our category of finite quantum sets and functions is equivalent to their category of quantum sets and one-dimensional quantum functions. 

In fact, even the quantum functions of dimension other than one appear naturally in our framework. A quantum function is essentially a quantum family of functions indexed by an atomic quantum set, which determines the dimension of that quantum function. We describe this bijective correspondence in detail toward the end of this introductory section, after demonstrating the relevance of our framework to the graph coloring game in \mbox{subsection \ref{subsection}}.

\subsection{summary of contents}\label{section 1.2} We begin the development by defining a category $\qRel$ of quantum sets and their binary relations, which is a dagger compact category with biproducts, or synonymously, a strongly compact closed category with biproducts \cite{AbramskyCoecke08}. The category $\mathbf{Rel}$ of ordinary sets and ordinary binary relations can be canonically included into $\qRel$ as a full subcategory. We notate this inclusion $S \mapsto `S$. The usual partial order on binary relations generalizes to binary relations between quantum sets. 

We define a subcategory $\qSet$ of $\qRel$, by directly applying the usual characterization of functions as binary relations satisfying a pair of conditions expressed in terms of the dagger compact structure on $\mathbf{Rel}$. There is a contravariant equivalence between the category of quantum sets and functions, and the category of hereditarily atomic von Neumann algebras and unital normal $*$-homomorphisms; we define a von Neumann algebra to be hereditarily atomic just in case every von Neumann subalgebra is atomic. This contravariant equivalence generalizes the familiar equivalence between ordinary sets and commutative atomic von Neumann algebras, which is defined by $S \mapsto \ell^\infty(S)$.

Our quantum sets can be characterized as sets of nonzero finite-dimensional Hilbert spaces, or equivalently, as hereditarily atomic von Neumann algebras. Analogies between sets and Hilbert spaces, or equivalently, type I factors, are well established in the literature \cite{GilesKummer71} \cite{Schlesinger99} \cite{Weaver01}. The distinguishing feature of our definition is essentially that we allow infinite $\ell^\infty$-direct sums of factors, but restrict to finite type I factors. Our immediate purpose is to obtain a category of quantum sets and binary relations that is dagger compact, and a category of quantum sets and functions that is complete, cocomplete, and closed.

Researchers working with quantum groups have settled on the same notion of discreteness. Specifically, the discrete quantum groups, those locally compact quantum groups that are dual to compact quantum groups, can also be characterized as locally compact quantum groups whose C*-algebras are $c_0$-direct sums of matrix C*-algebras \cite{PodlesWoronowicz90} \cite{EffrosRuan94} \cite{VanDaele96}. A recent paper investigates the action of compact quantum groups on ``discrete quantum spaces'' \cite{DeCommerKasprzakSkalskiSoltan16}, essentially the quantum sets of the present paper. I thank Piotr So\l tan for bringing this exciting connection to my attention.

Our definition of binary relations is essentially that of Weaver \cite{Weaver10}; it arose from his work with Kuperberg on quantum metrics \cite{KuperbergWeaver10}, itself motivated in part by quantum information theory. It is straightforward that the category $\qRel$ of quantum sets and binary relations is equivalent to the category of hereditarily atomic von Neumann algebras and quantum relations in Weaver's sense. The category of all von Neumann algebras and quantum relations is sadly not dagger compact.

We denote the monoidal structure on $\qRel$ and $\qSet$ by $- \times -$, and its monoidal unit by $\mathbf 1$. A monoidal product $\X_1 \times \X_2$ of two quantum sets is generally not their product in the category theory sense, but this monoidal structure does retain a number of important properties of the usual Cartesian product of ordinary sets. First, $`S_1 \times `S_2$ is naturally isomorphic to $`(S_1 \times S_2)$, as functors in ordinary sets $S_1$ and $S_2$. Second, this monoidal structure has a terminal monoidal unit in $\qSet$, which allows us to define canonical projection morphisms $P_1\:\X_1 \times \X_2 \To \X_1$ and $P_2 \: \X_1 \times \X_2 \To \X_2$. Third, the monoidal product $\X_1 \times \X_2$ satisfies the uniqueness property of the category-theoretic product.

Explicitly, the symmetric monoidal category $(\qSet, \times, \mathbf 1)$
\begin{enumerate}
\item is finitely complete,
\item is finitely cocomplete,
\item is closed,
\item has a terminal monoidal unit $\mathbf 1$,
\item has, for each pair of morphisms  $F_1\: \Y \To \X_1$ and $F_2\: \Y \To \X_2$, \emph{at most} one morphism making the following diagram commute,
$$
\begin{tikzcd}
&
\Y \arrow{ld}[swap]{F_1} \arrow{rd}{F_2}
\arrow{d}[swap]{!}
&
\\
\X_1
&
\X_1 \times \X_2 \arrow{l}{P_1} \arrow{r}[swap]{P_2}
&
\X_2
\end{tikzcd}
$$
\item and has, for every monic $\Z \rightarrowtail \X$, a unique ``classical'' function from $\X$ to the coproduct $\mathbf 1 \uplus \mathbf 1$ making the following diagram into a pullback square:
$$
\begin{tikzcd}
\Z \arrow{r} \arrow[tail]{d}
&
\mathbf 1 \arrow[hook]{d}{J_2}
\\
\X \arrow[dotted]{r}{!}
&
\mathbf 1 \uplus \mathbf 1
\end{tikzcd}$$
\end{enumerate}
The function $J_2\: \mathbf 1 \hookrightarrow  \mathbf 1 \uplus \mathbf 1$ is the inclusion of $\mathbf 1$ into the second summand of the coproduct. In any symmetric monoidal category with terminal unit, we define a morphism $F\: \X \To \Y$ to be ``classical'' just in case there is a morphism $\X \To \Y \times \X$ whose first component is $F$ and whose second component is the identity on $\X$.

The monoidal structure $\times$ may be viewed as formalizing the compatibility of observables. If a quantum system is described by some quantum set $\X$, then its observables correspond to functions $\X \To `\RR$. Two observables $F_1$ and $F_2$ are then compatible in the familiar physical sense if and only if there is a function $\X \To `\RR \times `\RR$ whose first and second components are $F_1$ and $F_2$ respectively. A classical observable is then an observable that is compatible with every observable. More generally, a classical function between quantum sets is essentially a deterministic quantum operation that is compatible with every observable. Such functions are characterized by Proposition \ref{12.8}.

Each topos may be regarded as a symmetric monoidal category by equipping it with its category-theoretic product. Adjusting the definition of topoi in this way, we find that a symmetric monoidal category satisfying properties (1) -- (6) above is a topos if and only if its monoidal product coincides with its category-theoretic product, in other words, if and only if it is Cartesian monoidal. This observation suggests two tantalizing notions: first, that there is a direct quantum generalization of topos theory, and second, that monoidal categories generalize Cartesian categories like quantum objects generalize their classical counterparts in noncommutative mathematics.

Thus, the category of quantum sets and functions retains a number of familiar properties of the category of ordinary sets and functions. However, the incompatibility of morphisms that is formalized by the discrepancy between the monoidal product $\times$, and the category-theoretic product $\ast$, is both essential to the quantum setting, and a significant departure from the behavior of ordinary sets. Furthermore, the terminal object is not a generating object, and this fact complicates the direct comparison of our category to standard category-theoretic accounts of set theory, such as Lawvere's elementary theory of the category of sets \cite{Lawvere64}. However, we mention that the axiom of choice fails in its standard categorical formulations, because state preparation is a nondeterministic quantum operation, i.e., it is not described by a function. Thus, the measurement of a qubit is a surjective function $\H_2 \To `\{0,1\}$ that has no right inverse $`\{0,1\} \To \H_2$. Here, $\H_2$ is a quantum set which consists of a single two-dimensional atom, formalizing the qubit.

The opposite of the category of von Neumann algebras and unital normal $*$-homomorphisms with the spatial tensor product is also a symmetric monoidal category satisfying (1) -- (6) \cite{Kornell17}, and it likewise arises from the category of von Neumann algebras and quantum relations in Weaver's sense \cite{Kornell11}. The arguments are essentially the same; their adaptibility to quantum sets rests on the fact that $\ell^\infty$-direct sums of finite type I factors form a hereditary class of von Neumann algebras, in other words, that the von Neumann subalgebras of any von Neumann algebra in this class are themselves in this class. So\l tan's quantum families of maps are an even earlier instance of the same construction of exponential \mbox{operator algebras \cite{Soltan09}}.

Effectus theory \cite{ChoJacobsWesterbaanWesterbaan15} offers another generalization of topoi to the quantum setting. It is a much broader generalization than the one suggested by properties (1) -- (6), and in particular, it is directly applicable to categories formalizing probabilistic quantum operations by completely positive maps. The morphisms of $\qSet$ are deterministic in the sense of \cite{Kornell17}, but it is possible to reason indirectly about probabilistic processes by constructing a quantum set of subprobability distributions $\F(\X)$ for each quantum set $\X$ \cite{WesterbaanThesis}*{4.3.4}, following the same approach as for the categories of von Neumann algebras in \cite{Westerbaan14}. Furthermore, I expect that quantum sets form an adequate model of the quantum lambda calculus, as in \cite{ChoWesterbaan16}.

In noncommutative geometry, operator algebras are viewed as algebras of complex-valued functions on a quantum object. Our definition of quantum sets incarnates this intuition. A quantum set $\X$ is a concrete object distinct from and simpler than the corresponding operator algebra $\ell(\X)$, and furthermore, up to equivalence of categories, this operator algebra literally consists of functions from the quantum set $\X$ to a canonical quantum set $\C$, which has the structure of a ring, internal to the category $\qSet$. Formally, we show that the functors $\ell(\,-\,)$ and $\mathrm{Fun}(\,-\,;\C)$ are naturally isomorphic, as contravariant functors to the category of $*$-algebras over $\CC$.

In the appendix, we also show that the functors $\mathrm{Proj}(\ell(\,-\,))$, $\mathrm{Fun}(\,-\,; `\{0,1\})$, and $\mathrm{Rel}(\,-\, ; \mathbf 1)$ are naturally isomorphic. The binary relations from a quantum set $\X$ to the terminal quantum set $\mathbf 1$ correspond to the predicates on $\X$, which we naively expect to be quantum sets in the their own right. We exhibit a technical modification to the definition of quantum sets that incorporates this intuition, without undermining preceding arguments.

In some places, the presentation is more explicit and verbose than strictly necessary. I have included calculation and discussion where a citation or a terse remark might do. A minimalist approach would be at odds with the secondary, expository goals of this article. The purpose of this article is to provide a convenient foundation for noncommutative discrete mathematics, a growing area that includes both researchers in noncommutative mathematics, and researchers in other academic disciplines. For researchers in the latter category, this article might serve as an introduction to noncommutative mathematics.

\subsection{quantum families of graph colorings}\label{subsection}
We now demonstrate the connection to quantum pseudotelepathy with an example. The purpose of this example is not only to show that the notion of quantum graph coloring implicit in \cite{CameronMontanaroNewmanSeveriniWinter} can be formalized in this framework, but furthermore to show that it can be motivated within noncommutative mathematics, entirely apart from the graph coloring game. This example applies several defined notions, so I suggest that the interested reader look over the example to get a general sense of what is done, and then come back to it once they are more familiar with the notions being applied.

Let $G$ be a finite simple graph, and $T$ a finite set, intuitively of colors. A proper graph coloring of $G$ is a function $f\: G \To T$ satisfying the familiar condition that the values of $f$ on any two adjacent vertices must be distinct. Similarly, we define a family of proper graph colorings indexed by a set $X$ to be a function $f\: G \times X \To T$ such that $f(g_1, x) \neq f(g_2, x)$ for all pairs of adjacent vertices $g_1$ and $g_2$, and all indices $x$. Equivalently, we can ask that the inverse image of each color for the function $f(\,\cdot\,, x)$ be an independent subset of $G$, for each index $x$, or that the inverse images of each color for the functions $f(g_1, \, \cdot \,)$ and $f(g_2, \, \cdot \,)$ be disjoint subsets of $X$, for each adjacent pair of vertices $g_1$ and $g_2$.

Thus, we want to define a quantum family of proper graph colorings, indexed by a quantum set $\X$, to be a function $F\: `G \times \X \To `T$ such that the inverse images of each color under the functions $F(g_1, \,\cdot \,)$ and $F(g_2, \, \cdot\,)$ are disjoint, for each adjacent pair of vertices $g_1$ and $g_2$. We explain what this should mean. The quantum set $`G$ is the quantum set canonically identified with the ordinary set $G$; it has a one-dimensional atom for each element of $G$ (section \ref{section 2}). The notation $`G \times \X$ refers to the Cartesian product of quantum sets (Definition \ref{B}). The morphism $F$ should be a function (Definition \ref{J}, section \ref{section 3}). The notation $F(g_1, \, \cdot \,)$ refers to the function $F_{g_1} = F \circ ( `g_1 \times I_\X) \: \X \To `T$ (Definition \ref{E}), where the symbol $g_1$ is also being used to denote the ordinary function $ \{\ast\} \To G$ with value $g_1$; likewise for the notation $F(g_2, \, \cdot \,)$. The quantum set $`\{\ast\}$ is a unit for the monoidal product $\times$.

For each color $t \in T$, let us write $\{t\}_T$ for that singleton considered as a predicate on $T$, that is, as a subset of $T$. It is canonically identified with a predicate $`\{t\}_T$ on the quantum set $`T$ (Definition \ref{qur}). Taking inverse images, we obtain two predicates $F_{g_1}^\star(`\{t\}_T)$ and $F_{g_2}^\star(`\{t\}_T)$ on $\X$ (Definition \ref{image}). We ask that these two predicates be disjoint (Definition \ref{qur}).

\begin{definition}\label{pseudotelepathy}
A \emph{quantum family of proper graph colorings} indexed by a quantum set $\X$ is a function $F\: `G \times \X \To `T$ such that the predicates $F_{g_1}^\star(`\{t\}_T)$ and $F_{g_2}^\star(`\{t\}_T)$ are disjoint for each color $t \in T$, whenever $g_1, g_2 \in G$ are adjacent vertices.
\end{definition}

\begin{proposition}
The following are equivalent:
\begin{enumerate}[(1)]
\item There is a winning strategy for the graph coloring game using quantum entanglement.
\item There exists a quantum family of proper graph colorings indexed by a quantum set $\X \neq `\emptyset$.
\item There exists a quantum family of proper graph colorings indexed by a quantum set $\H$, with $\At(\H) = \{H\}$ for some nonzero finite-dimensional Hilbert space $H$.
\end{enumerate}
\end{proposition}

This proposition refers to the graph coloring game investigated in \cite{CameronMontanaroNewmanSeveriniWinter}, also described in subsection \ref{game}. The quantum family of proper graph colorings in (3) is essentially a function of dimension $\mathrm{dim}(H)$ from $G$ to $T$ in the sense of \cite{MustoReutterVerdon}. The empty quantum set $`\emptyset$ has no atoms at all (Definition \ref{B}).

\begin{proof} Let $H$ be a nonzero finite-dimensional Hilbert space, and let $\H$ be the quantum set whose only atom is $H$. The quantum set $`G \times \H$ is a coproduct of copies of $\H$, one for each element of $G$ (section \ref{section 8}). The functions $`g \times I_\H \: \H \To `G \times \H$ are the injections for this coproduct. Thus, a function $F\: `G \times \H \To `T$ is uniquely determined by its restrictions $F_g\: \H \To `T$, for $g \in G$. 

Under the duality between quantum sets and hereditarily atomic von Neumann algebras, the functions $F_g\: \H \To `T$ correspond to unital normal $*$-homomorphisms $F^\star_g\: \ell^\infty(T) \To L(H)$ (Definition \ref{P}, Theorem \ref{Q10}). The expression $L(H)$ denotes the von Neumann algebra of all bounded linear operators on $H$, which are essentially matrices. Each unital normal $*$-homomorphism $F_g^\star\: \ell^\infty(T) \To L(H)$ is uniquely determined by the projections $p_{gt} =F_g^\star(\delta_t)$, for $t \in T$, where $\delta_t \in \ell^\infty(T)$ is the function that takes value $1$ at $t$, and vanishes otherwise. Thus, we have a bijection between functions $F\: `G \times \H \To `T$, and indexed families of projections $({p_{gt} \in L(H)}\suchthat {g \in G,} \, {t \in T})$ satisfying $\sum_{t\in T} p_{gt} = 1_H$ for each $g \in G$.

Let $g_1$ and $g_2$ be adjacent vertices. The condition that the predicate $F_{g_1}^\star(`\{t\}_T)$ be disjoint from the predicate $F_{g_2}^\star(`\{t\}_T)$ is equivalent to the condition that the projection $F_{g_1}^\star(\delta_t)$ be orthogonal to the projection $F_{g_2}^\star(\delta_t)$ (Theorem \ref{4functors}). Therefore, we have a bijection between quantum families of proper graph colorings indexed by $\H$, and indexed families of projections $(p_{gt} \in L(H)\suchthat g \in G, \, t \in T)$ satisfying $\sum_{t\in T} p_{gt} = 1_H$ for each vertex $g$, and satisfying $p_{g_1 t} \cdot p_{g_2 t} = 0$ for each adjacent pair of vertices $g_1$ and $g_2$, and each color $t$. Appealing to \cite{CameronMontanaroNewmanSeveriniWinter}*{section II}, we establish $(1) \Leftrightarrow (3)$.

The implication $(3) \Rightarrow (2)$ is trivial. We assume (2), that there exists a quantum family of proper graph colorings $F\: `G \times \X \To `T$ for some nonempty quantum set $\X$. Let $H$ be any atom of $\X$, and let $\H$ be the quantum set whose only atom is $H$. Let $J\: \H \hookrightarrow \X$ be the inclusion function (Definition \ref {Y1}). We claim that the function $\tilde F = F \circ (` \mathrm{id}_G \times J)\: `G \times \H \To `T$ is a quantum family of proper graph colorings. We compute that for all vertices $g$,
\begin{equation*}
\tilde F_g = F \circ (` \mathrm{id}_G \times J) \circ( `g  \times I_{\H}) = 
F \circ ( `g  \times I_\X) \circ J = F_g \circ J.
\end{equation*}
Let $g_1$ and $g_2$ be adjacent vertices, and let $t$ be a color. By assumption, $F_{g_1}^\star(`\{t\}_T)$ and $F_{g_2}^\star(`\{t\}_T)$ are disjoint. The inverse images of disjoint predicates are disjoint, so $\tilde F_{g_1}^\star(`\{t\}_T) = J^\star(F_{g_1}^\star(`\{t\}_T))$ and $\tilde F_{g_2}^\star(`\{t\}_T) = J^\star(F_{g_2}^\star(`\{t\}_T))$ are disjoint (Theorem \ref{4functors}). Therefore, $\tilde F$ is a quantum family of proper graph colorings, indexed by $\H$.
\end{proof}

We have described a class of quantum strategies as being equivalently quantum families of functions. The  quantum functions of Musto, Reutter, and Verdon \cite{MustoReutterVerdon}*{definition 3.11} are likewise equivalently quantum families of functions. Their category $\mathrm{QSet}$ \cite{MustoReutterVerdon}*{definition 3.18} is weakly equivalent as a 1-category to the category $\CCC$ whose objects are finite quantum sets, and whose nonzero morphisms are families of functions indexed by atomic quantum sets. More precisely, an object of $\CCC$ is a finite quantum set in our sense, and a morphism of $\CCC$ from $\X$ to $\Y$ is a function $\X \times \H \To \Y$, with $\H$ empty or atomic. Composition is defined in the expected way. The weak equivalence from $\CCC$ to $\mathrm{QSet}$ is closely related to the equivalence in Theorem \ref{Q10}, but it is covariant. Each object $\X$ in $\CCC$ is taken to the finite-dimensional von Neumann algebra $\ell^\infty(\X)$, which is also uniquely a special symmetric dagger Frobenius algebra \cite{Vicary}*{theorem 4.6}, i.e., a quantum set in the sense of Musto, Reutter, and Verdon. Each morphism $F\: \X \times \H \To \Y$ from $\X$ to $\Y$ in $\CCC$ is taken to the linear map obtained from the coalgebra homomorphism $(F^\star)^\dagger$ by putting it through the following natural homomorphisms:
\begin{equation*}
L(\ell^\infty(\X) \tensor L(H), \ell^\infty(\Y) ) \;\iso \; L(\ell^\infty(\X) \tensor H \tensor H^*, \ell^\infty(\Y)) \; \iso \; L(\ell^\infty(\X) \tensor H, H \tensor \ell^\infty(\X))
\end{equation*}
Here, $H$ is the unique atom of $\H$, if $\H$ is nonempty, and otherwise $H = 0$.
The result is a function in the sense of Musto, Reutter, and Verdon \cite{MustoReutterVerdon}*{definition 3.11}; the dimension of $H$ is the dimension of the function. We do not prove the claimed equivalence of categories, which is barely outside the scope of this article.

\subsection{notation and terminology} For the benefit of researchers working in physics and computer science, the development is initially framed in terms of Hilbert spaces, rather than operator algebras. Our binary relations correspond to the quantum relations of Weaver \cite{Weaver10}, but this Hilbert space framing avoids operator topologies. The Hermitian adjoint is rendered $\dagger$, and the symbol $\ast$ is reserved for the Banach space adjoint. Nevertheless, we retain the stock term ``$*$-homomorphism'', to mean an algebra homomorphism that preserves the Hermitian adjoint $\dagger$. We write $H \leq K$ when $H$ is a subspace of $K$. We write $L(H,K)$ for the space of all bounded linear operators from $H$ to $K$. If $V \leq L(H_1, H_2)$ and $W \leq L(H_2, H_3)$, we write $W \cdot V$ for the span of the set $\{ wv \suchthat v \in V,\, w \in W \}$. We also use the dot as a visual separator in products. Homomorphisms are not assumed to be unital, but representations are assumed to be nondegenerate. Von Neumann algebras are assumed to be concrete and to contain the identity operator. A von Neumann subalgebra of $A$ is an ultraweakly closed $*$-subalgebra that need not contain the identity operator, but the Hilbert space shrinks correspondingly. An ortholattice is a bounded lattice, not necessarily distributive, equipped with an orthocomplementation operation. For quantum sets $\X$ and $\Y$, we write $\mathrm{Fun}(\X;\Y)$ for the set of functions from $\X$ to $\Y$, we write $\mathrm{Par}(\X;\Y)$ for the set of partial functions from $\X$ to $\Y$, and we write $\mathrm{Rel}(\X;\Y)$ for the set of binary relations from $\X$ to $\Y$. An ordinary set is just a set in the ordinary sense, as opposed to a quantum set, and an ordinary function is just a function in the ordinary sense, as opposed to a function between quantum sets.

For reference, I suggest \textit{Operator Algebras} \cite{Blackadar17} and\textit{Categorical Quantum Mechanics} \cite{AbramskyCoecke08}; the term ``strongly compact closed'' is used synonymously with ``dagger compact''.

\section{Quantum sets}\label{section 2}

\begin{definition}\label{A}
A \emph{quantum set} $\X$ is completely determined by a set $\At(\X)$ of nonzero finite-dimensional Hilbert spaces, called the \emph{atoms} of $\X$.
\end{definition}

Formally $\X = \At(\X)$, so each quantum set is also an ordinary set, i.e., a set in the ordinary sense. Of course, in the standard formalization of mathematics in set theory, the same is true of each real number. Formalized as a Dedekind cut, the real number $\pi$ is equal to the set $\QQ \cap (-\infty, \pi)$, but it is good mathematical practice to draw a distinction between the two objects. We similarly draw a distinction between $\X$ and $\At(\X)$. We write $X \atomof \X$ to mean that $X$ is an atom of $\X$, that is, as an abbreviation for $X \in \At(\X)$. The symbol $\atomof$ is intended to suggest the word ``atom''.

Thus, the atoms of $\X$ are not the elements of $\X$, except on a purely formal level. The intuition is that the genuine elements of $\X$ may be mathematically fictional, like the points of a quantum compact Hausdorff space. They may be inextricably mixed together, having no individual identity. The atoms of $\X$ merely correspond to those subsets of $\X$ that are indecomposable with respect to the union operation, which we now define.

\begin{definition}\label{B}\label{C} \quad 
\begin{enumerate}
\item A quantum set $\X$ is \emph{empty} iff $\At(\X) = \emptyset$.
\item A quantum set $\X$ is \emph{finite} iff $\At(\X)$ is finite.
\item A quantum set $\X$ is a \emph{subset} of a quantum set $\Y$, written $\X \subsetof \Y$, iff $\At(\X) \subsetof \At(\Y)$.
\item The \emph{union} $\X \union \Y$ of quantum sets $\X$ and $\Y$ is defined by
\begin{equation*}
\At(\X \union \Y) = \At(\X) \union \At(\Y).
\end{equation*}
\item The \emph{Cartesian product} $\X \times \Y$ of quantum sets $\X$ and $\Y$ is defined by
\begin{equation*}
\At(\X \times \Y) = \{X \tensor Y \suchthat \mathrm X \atomof \X \text{ and } Y \atomof \Y\}.
\end{equation*}
\end{enumerate}
\end{definition}

Each of the five notions just defined are generalizations of the corresponding notions for ordinary sets, in the following precise sense. We identify each ordinary set $S$ with a quantum set $`S$, which has a one-dimensional atom for each element of $S$. Formally, $`S$ is defined by $\At(`S) = \{\CC_s \suchthat s \in S\}$, where $\CC_s = \ell^2(\{s\})$, with the understanding that $\CC_s \neq \CC_{s'}$ whenever $s$ and $s'$ are distinct elements of $S$.

For all ordinary sets $S$ and $T$, we now have that $`S$ is empty if and only if $S$ is empty in the ordinary sense, that $`S$ is finite if and only if $S$ is finite in the ordinary sense, and that $`S$ is a subset of $`T$ if and only if $S$ is a subset of $T$ in the ordinary sense. Furthermore, for all ordinary sets $S$ and $T$, we have that $`S \union `T = `(S \union T)$. However, the equation $`S \times `T = `(S \times T)$ is only true up to isomorphism. The equality becomes strict if the dagger compact category $\mathbf{FdHilb}$ of finite-dimensional Hilbert spaces and linear operators is tailored appropriately (appendix \ref{section 14}).

We say that quantum sets $\X$ and $\Y$ are isomorphic iff there exists an ordinary bijection $\At(\X) \To \At(\Y)$ that preserves the dimensions of atoms. This condition is equivalent to the existence of an invertible morphism from $\X$ to $\Y$ in the category $\qSet$ of quantum sets and functions (Proposition \ref{L2}).

This pattern of generalization does strongly suggest a notion of membership for quantum sets. It is certainly natural to say that an object $x$ is an element of a quantum set $\X$ iff $`\{x\} \subsetof \X$, or equivalently, iff $\CC_x \in \At(\X)$. One-dimensional atoms not of the form $\CC_x$ also intuitively correspond to individual elements, albeit not to specific objects in the mathematical universe. As with the generalized Cartesian product, this flaw can be repaired (appendix \ref{section 14}), but it is inconsequential to the structural approach taken in this article. 

Thus, we define a quantum set $\X$ to be a \emph{singleton} iff it has exactly one atom, and that atom is one-dimensional. More inclusively, we define a quantum set $\X$ to be \emph{atomic} iff it simply has exactly one atom. The two notions coincide for ordinary sets, considered as quantum sets in the manner described above. We use the established term for the former notion, because it formalizes the intuition that $\X$ has exactly one element, and the latter notion does not. For example, any atomic quantum set that is not a singleton admits a surjection onto $`\{1,2\}$ (Proposition \ref{Y}).

Our overall approach is to use established terms in an unqualified way only when that term retains its familiar meaning whenever all the relevant quantum sets are of the form $`S$. The notation is chosen to blur the distinction between $S$ and $`S$, and to emphasize the distinction between $\X$ and $\At(\X)$. Nevertheless, since we will occasionally wish to define a quantum set by specifying its atoms, we introduce the following notation.

\begin{definition}\label{C2}
Let $M$ be an ordinary set of finite-dimensional Hilbert spaces. We write $\Q M$ for the unique quantum set such that $\At(\Q M) = \{ H \in M \suchthat \dim(H) \neq 0 \}$. For example, we write $\mathbf 1 = \Q\{\CC\}$ for the quantum set whose only atom is $\CC$.
\end{definition}

The definition of quantum sets (Definition \ref{A}) might more naturally be given to include zero-dimensional Hilbert spaces. Chris Heunen observed that this yields equivalent categories $\qRel$ and $\qSet$; adding zero-dimensional atoms to a quantum set does not affect its isomorphism class. We exclude zero-dimensional atoms so that the isomorphism of quantum sets in $\qSet$ is equivalent to their isomorphism in the naive sense. Nevertheless, sometimes the most natural definition of a given quantum set includes Hilbert spaces that may be zero-dimensional, and in those cases, the convention in Definition \ref{C2} is useful and appropriate.

\section{Binary relations between quantum sets}\label{section 3}

We write $L(H, K)$ for the set of all bounded linear operators from $H$ to $K$.

\begin{definition}\label{D}
A \emph{binary relation} $R$ from a quantum set $\X$ to a quantum set $\Y$ is a function that assigns to each pair of atoms, $X$ of $\X$ and $Y$ of $\Y$, a subspace ${R(X,Y) \leq L(X, Y)}$. We write $\mathrm{Rel}(\X;\Y)$ for the set of all binary relations from $\X$ to $\Y$.
\end{definition}

If $\X$ and $\Y$ are finite, then we can visualize each binary relation $R$ from $\X$ to $\Y$ as a matrix. Indexing the atoms of both quantum sets by initial segments of the natural numbers, we picture $R$ as a matrix whose $(i,j)$ entry is the subspace $R(X_j, Y_i) \leq L(X_j, Y_i)$. This visual intuition is useful for understanding the operations on these binary relations, and for constructing examples, even when $\X$ or $\Y$ is infinite. In the countably infinite case, we picture $R$ as an infinite matrix.

For ordinary sets $S$ and $T$, the binary relations from $`S$ to $`T$ are in canonical bijective correspondence with the ordinary binary relations from $S$ to $T$, because for all one-dimensional Hilbert spaces $X$ and $Y$, the vector space $L(X,Y)$ is itself one-dimensional. If $r$ is an ordinary binary relation from $S$ to $T$, we write $`r$ for the corresponding binary relation from $`S$ to $`T$.

\begin{definition}\label{E}
Let $\X$, $\Y$, and $\Z$ be quantum sets. If $R$ is a binary relation from $\X$ to $\Y$, and $S$ is a binary relation from $\Y$ to $\Z$, then their \emph{composition} is the relation from $\X$ to $\Z$ defined by
\begin{equation*}
(S \circ R)(X, Z) = \mathrm{span} \{ s r \suchthat \exists Y \atomof \Y\: s \in S(Y,Z) \mathbin{\&} r \in R(X, Y)\}.
\end{equation*}
\end{definition}

Visualizing $R$ and $S$ as matrices, their composition $S \circ R$ is essentially obtained by matrix multiplication. The entries of the matrices are subspaces of linear operators; we add and multiply these subspaces in the naive way. Thus, quantum sets and their binary relations form a category $\qRel$. The identity relation $I_\X$ on a quantum set $\X$ is defined as follows: for all $X,X'\atomof \X$, the subspace $I_\X (X, X)$ is spanned by the identity operator on $X$, and $I_\X(X, X') = 0$ if $X \neq X'$. Visually, it is a diagonal matrix with spans of identity operators on its diagonal. 

Identifying each ordinary set $S$ with its quantum set counterpart $`S$, the category $\mathbf{Rel}$ of ordinary sets and ordinary binary relations is a full subcategory of $\qRel$. More precisely, the functor $S \mapsto `S$ is an equivalence of categories from $\mathbf{Rel}$ to the full subcategory of quantum sets whose atoms are all one-dimensional. Furthermore, this is an equivalence of dagger monoidal categories for the canonical dagger monoidal structure on $\qRel$, which we proceed to describe.

\begin{definition}\label{E2}
The \emph{Cartesian product} of binary relations $R_1$ and $R_2$ is the binary relation $R_1 \times R_2$ from $\X_1 \times \X_2$ to $\Y_1 \times \Y_2$ defined by
\begin{equation*}
(R_1 \times R_2)(X_1 \tensor X_2, Y_1 \tensor Y_2) = \mathrm{span}\{r_1 \tensor r_2 \suchthat r_1 \in R_1(X_1, Y_1),\, r_2 \in R_2(X_2, Y_2)\},
\end{equation*}
where $R_1$ is a binary relation from a quantum set $\X_1$ to a quantum set $\Y_1$, and $R_2$ is a binary relation from a quantum set $\X_2$ to a quantum set $\Y_2$.
\end{definition}

In simple examples, it is often possible to visualize the Cartesian product $R_1 \times R_2$ as a matrix of matrices, just as we do for the tensor product of ordinary matrices.

\begin{definition}\label{F}
For each finite-dimensional Hilbert space $H$, write $H^* = L(H, \CC)$ for the dual Hilbert space. For each linear operator $v \in L(H, K)$, write $v^* \in L(K^*, H^*)$ for the transpose of $v$, defined by $v^*(\varphi) = \varphi \circ v$. For each subspace $V \leq L(H, K)$, write $V^* = \{v^* \suchthat v \in V\} \leq L(K^*, H^*)$. The \emph{dual} of a quantum set $\X$ is the quantum set $\X^* = \Q\{X^*\suchthat X \atomof \X\}$. The \emph{transpose} of a binary relation $R$ from $\X$ to $\Y$ is the binary relation $R^*$ from $\Y^*$ to $\X^*$ defined by $R^*(Y^*,X^*) = R(X, Y)^*$, for $X \atomof \X$ and $Y \atomof \Y$.
\end{definition}

\begin{definition}\label{G}
Let $H$ and $K$ be finite-dimensional Hilbert spaces. For each linear operator $v \in L(H, K)$, write $v^\dagger \in L(K, H)$ for the Hermitian adjoint of $v$, defined by $\<v^\dagger k |h \> = \< k|vh\>$. For each subspace $V \leq L(H, K)$, write $V^\dagger = \{v^\dagger\suchthat v \in V\} \leq L(K,H)$. The \emph{adjoint} of a binary relation $R$ from $\X$ to $\Y$ is the binary relation $R^\dagger$ from $\Y$ to $\X$ defined by $R^\dagger(Y, X) = R(X, Y)^\dagger$.
\end{definition}

\begin{theorem}\label{dagger}
The structure $(\qRel,(\,-\,\times \,-\,), \mathbf 1, (\,-\,)^*, (\,-\,)^\dagger)$ is a dagger compact category.
\end{theorem}

The complete proof of this theorem is straightforward, but very tedious, so it is omitted. It is well known that the category $\FdHilb$ of finite-dimensional Hilbert spaces and linear operators, equipped with the tensor product and the Hermitian adjoint, is a dagger compact category, and the proof of Theorem \ref{dagger} proceeds by lifting these properties of $\FdHilb$ to $\qRel$, one commutative diagram at a time. Therefore, in lieu of providing a proof of Theorem \ref{dagger}, we exhibit the natural transformations that form the implicit structure of the dagger compact category $\qRel$, and we detail the properties of this structure that Theorem \ref{dagger} so tersely claims.

First, Theorem \ref{dagger} claims that $(\qRel, (\,-\,\times \,-\,), \mathbf 1)$ is a symmetric monoidal category. The monoidal product is the functor $(\,-\,\times \,-\,) \:\qRel \times \qRel \To \qRel$ given by Definition \ref{B}(5) and Definition \ref{E2}. The monoidal unit is the quantum set $\mathbf 1= \Q\{\CC\}$. The category is further equipped with natural isomorphisms called the associator, the left unitor, the right unitor, and the braiding. Each component of the associator in $\qRel$ is a binary relation $(\X \times \Y) \times \Z \To \X \times (\Y \times \Z)$ that takes pairs of the form $((X \tensor Y) \tensor Z, X \tensor (Y \tensor Z))$ to the span of the corresponding component of the associator in $\FdHilb$, and vanishes otherwise. Each component of the left unitor in $\qRel$ is a binary relation $\mathbf 1 \times \X \To \X$ that takes pairs of the form $(\CC \tensor X, X)$ to the span of the corresponding component of the left unitor in $\FdHilb$, and vanishes otherwise. The right unitor in $\qRel$ is defined similarly. Each component of the braiding in $\qRel$ is a binary relation $\X \times \Y \to \Y \times \X$ that takes pairs of the form $(X \tensor Y, Y \tensor X)$ to the span of the corresponding component of the braiding in $\FdHilb$, and vanishes otherwise. These natural isomorphisms satisfy three identities, typically expressed by commutative diagrams. The first identity disambiguates the structural natural isomorphism $((\W \times \X) \times \Y) \times \Z \To \W \times (\X \times (\Y \times \Z))$. The second identity disambiguates the structural natural isomorphism $\X \times \Y \To \Y \times \X$. The third identity disambiguates the structural natural isomorphism $\X \times \mathbf 1  \To \X$. 

Second, Theorem \ref{dagger} claims that $(\qRel, (\,-\,\times \,-\,), \mathbf 1)$ is a compact closed category, whose duals are given by the functor $(\,-\,)^*$. The contravariant functor $(\,-\,)^*\: \qRel \To \qRel$ is given by Definition \ref{F}. The category is further equipped with natural transformations called the unit and the counit. Each component of the unit in $\qRel$ is a binary relation $H_\X \:\mathbf 1 \to \X^* \times \X$ that takes pairs of the form $(\CC, X^* \tensor X)$ to the span of the corresponding component of the unit in $\FdHilb$, and vanishes otherwise. The $X$-component of the unit in $\FdHilb$ is the linear operator from $\CC$ to $X \tensor X^*$ that takes the complex number $1$ to the ``identity matrix". Each component of the counit in $\qRel$ is a binary relation $E_\X \:\X \times \X^* \to \mathbf 1$ that takes pairs of the form $(X \tensor X^*, \CC)$ to the corresponding component of the counit in $\FdHilb$, and vanishes otherwise. The $X$-component of the counit in $\mathbf{FdHilb}$ is the evaluation operator $X \tensor X^* \To \CC$. These natural transformations satisfy two identities. The first identity expresses that $(E_\X \times I_\Y) \circ (I _\X \times R^* \times I_\Y)\circ (I_\X \times H_\Y) = R$ for each binary relation $R$ from a quantum set $\X$ to a quantum set $\Y$. Dually, the second identity expresses that $ (I_{\X^*} \times E_\Y)  \circ (I_{\X^*} \times R \times I_{\Y^*})   \circ (H_\X \times I_{\Y^*}) = R^*$ for each binary relation $R$ from a quantum set $\X$ to a quantum set $\Y$.

Third, Theorem \ref{dagger} claims that $\qRel$ is a dagger category, compatibly with its compact closed structure. The contravariant functor $(\,-\,)^\dagger \: \qRel \To \qRel$ is given by Definition \ref{G}. It is identity on quantum sets, and for each binary relation $R$, it satisfies $(R^\dagger)^\dagger = R$. The functor $\dagger$ is compatible with the symmetric monoidal structure on $\qRel$ in the sense that
$(R \times S)^\dagger = R^\dagger \times S^\dagger$, for all binary relations $R$ and $S$, and each component of the associator, the left unitor, the right unitor and the braiding is unitary. Recall that a morphism $R$ in a dagger category is said to be unitary if it is invertible and satisfies $R\inv = R^\dagger$. Finally, the functor $\dagger$ is compatible with the compact closed structure on $\qRel$ in the sense that $E_\X^\dagger =  S_{\X^*,\X} \circ H_\X$ for all quantum sets $\X$, where $H_\X$, $E_\X$, and $S_{\X^*, \X}$ are components of the unit, the counit, and the braiding, respectively.

\smallskip

The significance of the dagger compact structure on $\qRel$ is two-fold. First, any dagger compact category supports a robust graphical calculus, in which morphisms are depicted as boxes, and objects are depicted as strings \cite{AbramskyCoecke08}. This graphical calculus renders the structural morphisms intuitive and unobtrusive, sometimes dramatically easing calculation.

Second, the category $\Rel$ of ordinary sets and binary relations is well known to be dagger compact, and it is straightforward to check that the inclusion $S \mapsto `S$ is a full and faithful functor that respects this dagger compact structure. Using the alternative definitions given in appendix \ref{section 14}, it does so exactly (Proposition \ref{Z5}). Many classes of ordinary structures can be axiomatized in terms of the dagger compact structure on $\Rel$, so the dagger compact structure on $\qRel$ yields a method of quantization. For example, in this article we use this method to quantize the structures that consist of two sets with a function between them (Definition \ref{J}). One of our main results is that in this instance, this quantization method agrees with the standard quantization method of noncommutative geometry (Theorem \ref{Q10}).

This method of quantization is essentially due to Weaver \cite{Weaver10}, who distilled it from his work with Kuperberg on the quantization of metric spaces \cite{KuperbergWeaver10}. Weaver's original approach is more general in the sense that he defines his quantum relations on arbitrary von Neumann algebras, and not just on the hereditarily atomic von Neumann algebras implicitly considered here (Proposition \ref{Q}). Thus, it generalizes discrete structures simultaneously to the quantum-mathematical and the measure-theoretic settings. However, the category of all von Neumann algebras and all quantum relations is not dagger compact, so some structures can be can be directly generalized to the quantum-mathematical, but not the measure-theoretic setting.

For example, consider the quantum generalization of simple graphs. An ordinary simple graph is just a set $S$ equipped with a symmetric, irreflexive binary relation $r$, so we can define a quantum graph to be a quantum set $\X$ equipped with a binary relation $R$ that is symmetric in the sense that $R^\dagger = R$, and irreflexive in the sense that $\mathrm{Tr}(R) = 0$. In any dagger compact category, we have a trace on the endomorphisms of each object $\X$, which is defined by $\Tr(R) = E_\X \circ (R \times I_{\X^*}) \circ E_\X^\dagger$. It is easy to check that a binary relation $r$ in $\Rel$ is irreflexive if and only if $\Tr(r) = 0$.

This definition of a quantum graph does not generalize directly to arbitrary von Neumann algebras. In that context, a quantum graph is defined to be equipped with a symmetric, \emph{reflexive} quantum relation. This resolution of the difficulty is well motivated, and it is adequate for many purposes, but not for all purposes. For example, we can define a matching in a quantum graph $(\X,R)$ to be a symmetric, irreflexive binary relation $S$ on $\X$ such that $S \leq R$ and $S\circ S = I$. For another example, we can define a graph homomorphism from a quantum graph $(\X_1,R_1)$ to a quantum graph $(\X_2, R_2)$ to be a function $F$ from $\X_1$ to $\X_2$ such that $F \circ R_1 \leq R_2 \circ F$. No easy fix is apparent in either example if instead of working with quantum sets, we work with arbitrary von Neumann algebras.

\begin{remark}
The category $\qRel$ also has coproducts of arbitrary families of quantum sets, even infinite families.
The coproduct of two quantum sets $\X$ and $\Y$ in the category $\qRel$ is their disjoint union $\X \uplus \Y$, which is defined up to isomorphism by $\At(\X \uplus \Y) = \At(\X) \uplus \At(\Y)$. Assuming that $\At(\X)$ and $\At(\Y)$ are disjoint, the inclusions $J\: \X \hookrightarrow \X \uplus \Y$ and $K\: \Y \hookrightarrow \X \uplus \Y$ are then defined by $J(X,X) = \CC\cdot 1_X$ for $X \atomof \X$, and $K(Y,Y) = \CC \cdot 1_Y$ for $Y \atomof \Y$, with the other components vanishing. If $R$ is a binary relation from $\X$ to a quantum set $\Z$, and $S$ is a binary relation from $\Y$ to $\Z$, then we can define a binary relation $[R,S]$ from $\X \uplus \Y$ to $\Z$,  by $[R,S](X,Z) = R(X,Z)$ and $[R,S](Y,Z) = S(Y,Z)$, for $X \atomof \X$, $Y \atomof \Y$, and $Z \atomof \Z$. It is straightforward to check that $[R,S] \circ J = R$, that $[R,S]\circ K = S$, and that together these two equations uniquely determine $[R,S]$.
\end{remark}

It follows by general arguments that $\qRel$ has biproducts and is enriched over commutative monoids  \cite{AbramskyCoecke08}*{5.2}. One may check that the monoidal product of this enrichment is the disjunction $\vee$, defined below. In fact, it is easy to see directly from Definition \ref{D} that composition respects infinitary disjunction. Altogether, the set of binary relations from a quantum set $\X$ to a quantum set $\Y$ carries the structure of a complete orthomodular lattice, but composition does not generally respect the rest of this structure.

\begin{definition}\label{I}
Let $\X$ and $\Y$ be quantum sets. The set $\mathrm{Rel}(\X; \Y)$ of binary relations from $\X$ to $\Y$ is canonically a complete orthomodular lattice, with operations defined atomwise. Explicitly, for $R,S\in \mathrm{Rel}(\X; \Y)$, we define:
\begin{enumerate}
\item $R \vee S = (R(X, Y) \vee S(X,Y) \suchthat X \atomof \X,\, Y \atomof \Y)$
\item $R \wedge S =( R(X, Y) \wedge S(X,Y)\suchthat X \atomof \X,\, Y \atomof \Y)$
\item $\neg R = (R(X,Y)^\perp\suchthat X \atomof \X,\, Y \atomof \Y)$
\end{enumerate}
We extend (1) and (2) to arbitrary families in $\mathrm{Rel}(\X;\Y)$, in the obvious way. We also write:
\begin{enumerate}
\setcounter{enumi}{3}
\item $R \leq S \;\Leftrightarrow \; \forall X \atomof \X\: \forall Y \atomof\Y\: R(X,Y) \leq S(X,Y)$
\item $R \perp S \;\Leftrightarrow \; \forall X \atomof \X\: \forall Y \atomof\Y\: R(X,Y) \perp S(X,Y)$
\end{enumerate}

\end{definition}

\section{Functions between quantum sets}\label{section 4}

\begin{definition}\label{J}
A \emph{function} from a quantum set $\X$ to a quantum set $\Y$ is a binary relation $F$ from $\X$ to $\Y$ such that $F^\dagger \circ F \geq I_\X$ and $F \circ F^\dagger \leq I_\Y$. We write $\mathrm{Fun}(\X;\Y)$ for the set of all functions from $\X$ to $\Y$.
\end{definition}

The identity binary relation $I_\X$ on a quantum set $\X$ satisfies $I_\X^\dagger = I_\X$ and $I_\X \circ I_\X = I_\X$, so it is a function. Furthermore, if $F$ is a function from $\X$ to a quantum set $\Y$, and $G$ is a function from $\Y$ to a quantum set $\Z$, then $G \circ F$ is also a function; we display one of the two relevant computations:
\begin{equation*}
( G \circ F)^\dagger \circ (G \circ F) = F^\dagger \circ G^\dagger \circ G \circ F \geq F^\dagger \circ I_\Y \circ F = F^\dagger \circ F \geq I_\X
\end{equation*}
Thus, quantum sets and functions form a subcategory $\qSet$ of $\qRel$.

\begin{proposition}\label{K}
Let $F$ be a function from a quantum set $\X$ to a quantum set $\Y$. If $F$ is invertible in $\qRel$, then $F\inv = F^\dagger$.
\end{proposition}

\begin{proof}

\begin{equation*}
F^\dagger = F^\dagger \circ F \circ F\inv \geq I_\X \circ F \inv = F\inv
\end{equation*}

\begin{equation*}
F^\dagger = F\inv \circ F \circ F^\dagger \leq F\inv \circ I_\Y = F \inv 
\end{equation*}
\end{proof}

\begin{definition}\label{L}
Let $R$ be a binary relation from a quantum set $\X$ to a quantum set $\Y$. We say that $R$ is
\begin{enumerate}
\item \emph{injective} iff $R^\dagger \circ R \leq I_\X$, and
\item \emph{surjective} iff $R \circ R^\dagger \geq I_\Y$.
\end{enumerate}
\end{definition}

As a consequence of Proposition \ref{K}, an invertible function must be both injective and surjective. Conversely, a function $F\: \X \To \Y$ that is both injective and surjective satisfies $F^\dagger \circ F = I_\X$ and $F \circ F^\dagger = I_\Y$, so it is invertible. Evidently, if $F$ is invertible in $\qRel$, then its inverse $F\inv = F^\dagger$ is also a function.

\begin{proposition}\label{L2}
Let $F$ be an invertible function from a quantum set $\X$ to a quantum set $\Y$. Then there exists an ordinary invertible function $f$ from $\At(\X)$ to $\At(\Y)$, and a family of unitaries $(u_X\: X \To f(X) \suchthat X \atomof \X)$ such that $F(X,f(X)) = \CC \cdot u_X$ for all $X \atomof \X$, with $F(X,Y) = 0$ whenever $Y \neq f(X)$. Conversely, every relation from $\X$ to $\Y$ that is of this form is an invertible function.
\end{proposition}

\begin{proof}
As we have seen, the invertibility of a function is equivalent to the system of equations $F^\dagger \circ F = I_\X$ and $F \circ F^\dagger = I_\Y$. Equivalently,
\begin{equation*}
\bigvee_{Y \atomof \Y} F(X,Y)^\dagger \cdot F(X',Y) = I_\X(X',X)
 \qquad \text{and} \qquad 
\bigvee_{X \atomof \X} F(X,Y) \cdot F(X,Y')^\dagger = I_\Y(Y',Y),
\end{equation*}
for all $X , X' \atomof \X$, and for all $Y, Y' \atomof \Y$, respectively.
These equations imply that for all $X \atomof \X$ and $Y \atomof \Y$, we have $F(X,Y)^\dagger \cdot F(X,Y) \leq \CC \cdot 1_X$ and $F(X,Y) \cdot F(X,Y)^\dagger \leq \CC \cdot 1_Y$, so if $F(X,Y)$ is nonzero, then it is spanned by a single unitary operator.

Fix $X \atomof \X$. At least one of the spaces $F(X,Y)$ must be nonzero. Furthermore, if $F(X,Y)$ and $F(X,Y')$ are both nonzero, then both are spanned by a unitary, and we have $I_\Y(Y', Y) \geq F (X,Y) \cdot F(X, Y')^\dagger \neq 0$, so $Y' = Y$. Thus, there is a unique atom $f(X) \atomof \Y$ such that $F(X,f(X)) \neq 0$. Choose a unitary $u_X$ that spans $F(X, f(X))$. As we vary $X$, we obtain a function $f\: \At(\X) \To \At(\Y)$, with each space $F(X, f(X))$ spanned by a unitary $u_X$, and all other spaces $F(X,Y)$ vanishing.

The function $f$ is an injection by similar reasoning. To show that $f$ is a surjection, assume that there exists an atom $Y_0 \atomof \Y$ that is not in the range of $f$. For this atom, we have $F(X,Y_0) = 0$ for all $X \atomof \X$, so we obtain a contradiction: $1_Y \in I_\Y(Y_0,Y_0) =0$. Therefore, $f$ is a bijection.

It is easy to see that any function from $\X$ to $\Y$ of the given form is invertible, simply by verifying the two equations given at the beginning of this proof. In each join, at most one of the terms is nonzero, and if there is such a term, then it is spanned by the product of a unitary operator with its adjoint.
\end{proof}

Both injectivity and surjectivity have natural dual notions.

\begin{definition}\label{L3}
Let $R$ be a binary relation from a quantum set $\X$ to a quantum set $\Y$. We say that $R$ is
\begin{enumerate}
\item \emph{coinjective} iff $R \circ R^\dagger \leq I_\Y$, and
\item \emph{cosurjective} iff $ R^\dagger \circ R \geq I_\X$.
\end{enumerate}
\end{definition}

Thus, a function is a binary relation that is both coinjective and cosurjective.

\begin{definition}\label{M}
A \emph{partial function} from a quantum set $\X$ to a quantum set $\Y$ is a coinjective binary relation from $\X$ to $\Y$.
\end{definition}

Quantum sets and partial functions form a subcategory $\qPar$ of $\qRel$, and $\qSet$ is a subcategory of $\qPar$.

\begin{proposition}\label{N}
Let $F$ be a partial function from a quantum set $\X$ to a quantum set $\Y$. If $F$ is invertible in $\qRel$, then $F\inv = F^\dagger$, so $F$ is a function.
\end{proposition}

\begin{proof}
The relation $F \circ F^\dagger$ is an invertible subrelation of $I_\Y$; therefore $F \circ F^\dagger = I_\Y$. Similarly, we can prove that $F^\dagger\circ F =I_\X$, if we can show that $F^\dagger \circ F \leq I_\X$. We can:
\begin{equation*}
F^\dagger \circ F = F\inv \circ F \circ F^\dagger \circ F \leq F \inv \circ I_\Y \circ F = I_\X 
\end{equation*}
\end{proof}

Not every invertible relation is a function. Let $a \in L(\CC^2, \CC^2)$ be an invertible matrix that isn't a scalar multiple of a unitary matrix. The relation $R$ from $\Q\{\CC^2\}$ to $\Q\{\CC^2\}$ defined by $R(\CC^2, \CC^2) = \CC \cdot a $ is evidently invertible, but not coinjective.

\section{Hereditarily atomic von Neumann algebras}\label{section 5}

\begin{definition}\label{O}
Let $\X$ be a quantum set. Define:
\begin{equation*}
\ell(\X) = \prod_{X \atomof \X} L(X)
\end{equation*}
This has the structure of a $\ast$-algebra over $\CC$, equipped with the product topology.
\end{definition}

For any ordinary set $S$, the $\ast$-algebra $\ell(`S)$ is canonically isomorphic to $\ell(S) = \CC^S$. We will later show that the self-adjoint elements of $\ell(\X)$ are in canonical bijective correspondence with functions from $\X$ to $`\RR$ (Proposition \ref{FH}). Similarly, the normal elements of $\ell(\X)$ are in canonical bijective correspondence with functions from $\X$ to $`\CC$. Arbitrary elements of $\ell(\X)$ are in bijective correspondence with functions from $\X$ to a canonical quantum set $\C$ (Definition \ref{quantum c}), which intuitively consists of complex numbers whose real and imaginary parts need not be simultaneously observable.

\begin{definition}\label{P}
\begin{equation*}
\ell^\infty(\X) = \left\{ a \in \ell(\X) \, \middle|\,\sup_{X \atomof \X} \|a(X)\|_\infty <  \infty \right\}
\end{equation*}
\begin{equation*}
c_0(\X) =\left \{ a \in \ell(\X) \, \middle|\, \lim_{X \To \infty} \|a(X)\|_\infty =0 \right\}
\end{equation*}
\end{definition}

The limit above is in the sense of one-point compactification; in other words, for each $\epsilon>0$, we should have $\|a(X)\|_\infty < \epsilon$ for all but finitely many $X \atomof \X$. The $*$-algebras $\ell^\infty(\X)$ and $c_0(\X)$ are canonically represented on the $\ell^2$-direct sum of the Hilbert spaces in $\At(\X)$, isometrically for the operator norm, and the norm
\begin{equation*}
\|a \|= \sup_{X \atomof \X} \|a(X)\|_\infty.
\end{equation*}
Represented in this way, $c_0(\X)$ is a concrete C*-algebra, and $\ell^\infty(\X)$ is a von Neumann algebra. In the context of noncommutative mathematics, the C*-algebra $c_0(\X)$ is the operator algebra associated to $\X$ considered as a discrete quantum topological space, and the von Neumann algebra $\ell^\infty(\X)$ is the operator algebra associated to $\X$ considered as an atomic quantum measure space. We might venture to say that $\ell(\X)$ is the operator algebra associated to $\X$ considered as a quantum set.

\begin{definition}
A von Neumann algebra $A$ is \emph{hereditarily atomic} just in case every von Neumann subalgebra of $A$ is atomic.
\end{definition}

Recall that a von Neumann algebra is said to be \emph{atomic}, or sometimes fully atomic, if every nonzero projection is above a minimal projection. Equivalently, a von Neumann algebra is atomic if and only if every projection is the sum of some family of pairwise orthogonal minimal projections.

\begin{proposition}\label{Q}
Let $A$ be a von Neumann algebra. The following are equivalent:
\begin{enumerate}
\item $A$ is hereditarily atomic
\item $A$ is isomorphic to $\ell^\infty(\X)$ for some quantum set $\X$
\item every self-adjoint operator $a$ in $A$ is diagonalizable
\end{enumerate}
\end{proposition}

We call a self-adjoint bounded operator $a$ \emph{diagonalizable} just in case there is a family of pairwise orthogonal projections $(p_\alpha \suchthat \alpha \in \RR)$ such that $a = \sum_\alpha \alpha p_\alpha$, with convergence in the ultraweak topology. Observe that if $a$ is diagonalizable, then the family $(p_\alpha)$ is unique, and each projection $p_\alpha$ is a spectral projection of $a$. As a consequence, $a$ is diagonalizable in $A$ if and only if it is diagonalizable in any given von Neumann subalgebra $B$ of $A$ that contains $a$.

\begin{proof}
$(1) \Yields (2)$. Let $A$ be a hereditarily atomic von Neumann algebra. The center of $A$ is atomic, so $A$ is an $\ell^\infty$-direct sum of factors. Every factor that is not finite type I has a von Neumann subalgebra isomorphic to $L^\infty([0,1],dt)$, which is not atomic, so $A$ must be a direct sum of finite type I factors. Choosing an irreducible representation for each factor, we obtain a quantum set $\X$ such that $\ell^\infty(\X) \iso A$.

$(2) \Yields (3)$. Let $A$ be isomorphic to $\ell^\infty(\X)$, and let $a \in A$ be self adjoint. Without loss of generality, we assume $A =\ell^\infty(\X)$. Each self-adjoint operator $a(X)$ can be diagonalized in $L(X)$ by the spectral theorem for self-adjoint matrices. Altogether, we have a diagonalization of $a$; a bounded net converges ultraweakly in an $\ell^\infty$-direct sum of von Neumann algebras if and only if it coverges ultraweakly in each summand.

$(3) \Yields (1)$. Assume that every self-adjoint operator in $A$ is diagonalizable. Let $B$ be a von Neumann subalgebra of $A$, and let $p$ be a nonzero projection in $B$. Choose a maximal abelian von Neumann subalgebra $C$ of the von Neumann algebra $pBp$. If $C$ has a von Neumann subalgebra isomorphic to $L^\infty([0,1],dt)$, then $C$ contains a nondiagonalizable self-adjoint operator, contradicting our assumption on $A$; indeed, the identity function is a nondiagonalizable element of $L^\infty([0,1],dt)$, and diagonalizability in a von Neumann subalgebra is equivalent to diagonalizability in the larger von Neumann algebra, as we observed just above the proof. Therefore, $C$ is atomic, and in particular, it contains a minimal projection $q$. The projection $q$ is also a minimal projection in $pBp$, because $C$ is maximal abelian. So, $q$ is a minimal projection in $B$ that is below $p$. We conclude that $B$ is atomic, and more generally, that $A$ is hereditarily atomic.
\end{proof}

\begin{definition}
We write $\Mstar0$ for the category of hereditarily atomic von Neumann algebras and normal $*$-homomorphisms. We write $\Mstar1$ for the category of hereditarily atomic von Neumann algebras and \emph{unital} normal $*$-homomorphisms.
\end{definition}

The letter ``M'' is intended to suggest matrix algebras.

\section{Three perspectives on a function}\label{section 6}

We now show that each partial function $F$ from a quantum set $\X$ to a quantum set $\Y$ induces a $*$-homomorphism $F^\star$ from $\ell(\Y)$ to $\ell(\X)$. The immediate significance of this construction is that it yields a contravariant equivalence of categories, from the category $\qSet$ of quantum sets and functions, to the category $\Mstar1$ of hereditarily atomic von Neumann algebras and unital normal $*$-homomorphisms (Theorem \ref{Q10}), placing our objects and morphisms within the usual framework of noncommutative mathematics, and enabling the use of operator-algebraic techniques.

If $\X$ and $\Y$ are quantum sets whose atoms are all one-dimensional, then a function $F$ from $\X$ to $\Y$ is effectively an ordinary function between ordinary sets, and $F^\star$ is effectively just precomposition by $F$. The unital $*$-homomorphism $F^\star$ can be viewed as precomposition by $F$ in the general case too (Theorem \ref{opalgebra}).

Let $\X$ and $\Y$ be arbitrary quantum sets. Let $F$ be an arbitrary partial function from $\X$ to $\Y$, and consider a pair of atoms $X \atomof \X$ and $Y \atomof \Y$. By the definition of partial function, $F \circ F^\dagger \leq I_\Y$, so
$F(X,Y) \cdot F(X,Y)^\dagger \leq \CC \cdot 1_Y$. We now recall a result classifying operator spaces of this kind \cite{SinhaGoswami07}*{4.2.7}.

\begin{lemma}\label{Q1}
Let $X$ and $Y$ be finite-dimensional Hilbert spaces. There are canonical bijective correspondences between
\begin{enumerate}
\item linear spaces $V$ of operators from $X$ to $Y$, satisfying $V \cdot V^\dagger \leq \CC \cdot 1_Y$,
\item coisometries $w\: X \To  Y \tensor H$, with $H$ a Hilbert space up to unitary equivalence, and
\item $*$-homomorphisms $\rho \: L(Y) \To L(X)$.
\end{enumerate}
\end{lemma}

Thus, we identify two coisometries $w\: X \To  Y \tensor H$ and $w'\: X \To Y \tensor H' $ iff there is a unitary operator $u\: H \To H'$ such that $w' = (1 \tensor u )w$.

\begin{proof}
We describe the three constructions. Let $V$ be any linear space of operators from $X$ to $Y$, satisfying $V \cdot V^\dagger \leq \CC \cdot 1_Y$. We have an inner product on $V$ defined by $v {v'}^\dagger = (v'| v) \cdot 1_Y$. Choose an orthonormal basis $(v_1, \ldots, v_n)$ for this inner product. Let $(e_1, \ldots, e_n)$ be the standard basis of $\CC^n$, and define $w\: X \To  Y\tensor\CC^n $ by $w(x) = \sum_{i=1}^n v_i(x)\tensor e_i$, for all $x \in X$. It is easy to check that $w^\dagger(y \tensor e_i ) = v_i^\dagger(y)$, for all $y \in Y$. A short computation shows that $w$ is a coisometry:
\begin{equation*}\< w^\dagger(y \tensor e_i) | w^\dagger(y' \tensor e_j) \> = \< v_i^\dagger y | v_j^\dagger y' \> = \< y | v_i v_j^\dagger y' \> =  \< y | y' \> \cdot \delta_{ij}= \< y \tensor e_i  | y' \tensor  e_j \>
\end{equation*}
Here, $\delta_{ij}$ is the Kronecker delta symbol. Adjusting that computation slightly, we find that if $w'$ is obtained using another choice of orthonormal basis, then $\< w^\dagger(y \tensor e_i) | {w'}^\dagger(y' \tensor e_j) \> =  \< y | y' \> u_{ij}$ for some unitary matrix $(u_{ij})$, so $w {w'}^\dagger = (1 \tensor u)$ for some unitary operator $u$ on $\CC^n$.

Let $w\: X \To Y \tensor H$ be any coisometry. Define $\rho\: L(Y) \To L(X)$ by $\rho(b) = w^\dagger (b \tensor 1) w$. It is easy to check that $\rho$ is a $*$-homomorphism, and that $\rho$ is the same if $H$ is replaced by a unitarily equivalent Hilbert space.

Let $\rho\: L(Y) \To L(X)$ be any $*$-homomorphism. Define $V$ to be the space of linear operators $v$ from $X$ to $Y$ such that $b v = v \rho(b)$ for all $b \in L(Y)$. It follows immediately that $V \cdot V^\dagger \subsetof L(Y)' = \CC \cdot 1_Y$.

Fix $V$, a linear space of operators from $X$ to $Y$ satisfying $V \cdot V^\dagger \leq \CC \cdot 1$. Performing the first two constructions, we obtain a $*$-homomorphism $\rho$ defined by $\rho(b) = \sum_{i=1}^n v_i^\dagger b v_i$, where $(v_1, \ldots, v_n)$ is some orthonormal basis for $V$. It is easy to check that each basis element $v_j$ satisfies $v_j \rho(b) = b v_j$, so $V$ is a subspace of the linear space $\tilde V$ constructed from $\rho$ via the third construction. Conversely, each operator $\tilde v$ in $\tilde V$ is in $V$ because
\begin{equation*}
\tilde v = 1 \cdot \tilde v = \tilde v \cdot \rho(1) = \sum_{i=0}^n \tilde v v_i^\dagger v_i = \sum_{i=0}^n (v_i|\tilde v)v_i.
\end{equation*}
Therefore $\tilde V = V$.

Thus, the three constructions compose to the identity. In particular, the third construction is surjective. The second construction is also surjective, by the representation theory of $L(Y)$. The first construction is surjective, because the construction $w \mapsto \mathrm{span}\{ (1 \tensor \hat e_i	) w\suchthat 1 \leq i \leq n  \}$, where $\hat e_i = \< e_i | \, \cdot\,\>$, is easily seen to be a right inverse. It follows immediately that all three constructions are bijective.
\end{proof}

Lemma \ref{Q1} suggests three different ways of looking at a function between quantum sets: as a relation, as a family of coisometries, and as a $*$-homomorphism. We use the term ``fission'' for the second of these three notions, because a function from a quantum set $\X$ to a quantum set $\Y$ splits each atom of $\X$ into finitely many atoms of $\Y$, and the corresponding family of coisometries makes this splitting explicit.

\begin{definition}\label{Q2}
Let $\X$ and $\Y$ be quantum sets. A \emph{partial fission} $f$ from $\X$ to $\Y$ is a family of finite-dimensional Hilbert spaces $(H_X^Y\suchthat X \atomof \X, Y \atomof \Y)$, together with a family of coisometries $(f_X^Y\: X \To Y \tensor H_X^Y\suchthat X \atomof \X,\, Y \atomof \Y)$ such that $(f_X^{Y_1})\cdot(f_X^{Y_2})^\dagger= 0$ whenever $Y_1 \neq Y_2$. We identify partial fissions $f$ and $f'$ from $\X$ to $\Y$ whenever there is a family of unitary operators $(u_X^Y\: H_X^Y \to{H'}\vphantom{f}_X^Y\suchthat X \atomof \X,\, Y \atomof \Y)$ such that ${f'}\vphantom{f}_X^Y = (1 \tensor u_X^Y) f_X^Y$ for all $X \atomof \X$ and $Y \atomof \Y$.
\end{definition}

The orthogonality condition implies that for each $X \atomof \X$, we have $\mathrm{dim} (H_X^Y) = 0$ for all but finitely many $Y \atomof \Y$. Intuitively, the partial fission splits each atom $X \atomof \X$ into atoms of $\Y$, with each atom $Y \atomof \Y$ occurring $\mathrm{dim}(H_X^Y)$ times. It is possible to formulate a definition that is even closer to this intuition, but it would obscure that, for example, the many orthogonal decompositions of a two-dimensional Hilbert space $X$ into subspaces isomorphic to $\CC$ are equivalent in the sense that they describe the same function from $\Q\{X\}$ to $\Q\{\CC\}$. In general, the support projections of the coisometries $(f_X^Y \suchthat Y \atomof \Y)$ need not sum to the identity on $X$; a part of the atom $X$ may be simply lost. For each $X \atomof \X$, we only have
\begin{equation*}
\dim(X) \geq \sum_{Y \atomof \Y} \dim (Y) \cdot \dim(H_X^Y).
\end{equation*}

\begin{theorem}\label{Q3}
Let $\X$ and $\Y$ be quantum sets. There are canonical bijective correspondences between
\begin{enumerate}
\item partial functions $F$ from $\X$ to $\Y$,
\item partial fissions $f$ from $\X$ to $\Y$, and
\item $*$-homomorphisms $\phi\: \ell(\Y) \To \ell(\X)$, continuous for the product topologies.
\end{enumerate}
Explicitly, the bijections are given by the following constructions:
\begin{itemize}
\item Let $F$ be a partial function from $\X$ to $\Y$. For each $X \atomof \X$, and each $Y \atomof \Y$, choose an orthonormal basis $B_X^Y$ of $F(X,Y)$ for the inner product defined by $(v | v') 1_Y = v' \cdot v^\dagger$. Equip $F(X,Y)^\dagger$ with the corresponding inner product, defined by $(h| h') 1_Y = h^\dagger \cdot h' $.
Then the family $(f_X^Y \: X \To Y \tensor F(X,Y)^\dagger \suchthat X \atomof \X,\, Y \atomof \Y)$ defined by
\begin{equation*}
f_X^Y(x) = \sum_{v \in B_X^Y}  v(x)\tensor v^\dagger
\end{equation*}
is a partial fission from $\X$ to $\Y$.
\item Let $f$ be a partial fission from $\X$ to $\Y$. Then the map $\phi\: \ell(\Y) \To \ell(\X)$ defined by
\begin{equation*}
\phi(b)(X) = \sum_{Y \atomof \Y} (f_X^Y)^\dagger ( b(Y) \tensor 1) (f_X^Y)
\end{equation*}
is a continuous $*$-homomorphism.
\item Let $\phi\: \ell(\Y) \To \ell(\X)$ be a continuous $*$-homomorphism. Then the binary relation $F$ from $\X$ to $\Y$ defined by
\begin{equation*}
F(X,Y) = \{v \in L(X, Y) \suchthat b(Y) \cdot v = v \cdot \phi(b)(X) \text{ for all }b \in \ell(\Y)\}
\end{equation*}
is a partial function from $\X$ to $\Y$.
\end{itemize}
Furthermore, the three constructions compose to the identity in each cyclical order.
\end{theorem}

\begin{proof}
Lemma \ref{Q1} yields canonical bijective correspondences between
\begin{enumerate}
\item families of subspaces $(F(X,Y) \leq L(X,Y) \suchthat X \atomof \X, Y \atomof \Y)$ satisfying $F(X,Y) \cdot F(X,Y)^\dagger \leq \CC \cdot 1_Y$,
\item families of coisometries $(f_{X}^{Y}\: X \To Y \tensor H_{X}^{Y} \suchthat X \atomof \X,\, Y \atomof \Y)$, up to unitary equivalence of the coefficient Hilbert spaces $H_X^Y$, and
\item families of $*$-homomorphisms $(\phi_X^Y\: L(Y) \To L(X) \suchthat X \atomof \X, Y \atomof \Y)$.
\end{enumerate}
We have three equivalent orthogonality conditions under these correspondences:
\begin{enumerate}
\item $F(X, Y_1) \cdot F(X, Y_2)^\dagger = 0$ whenever $Y_1 \neq Y_2$
\item $(f_X^{Y_1})\cdot(f_X^{Y_2})^\dagger= 0$ whenever $Y_1 \neq Y_2$
\item $\phi_X^{Y_1}(L(Y_1)) \cdot \phi_X^{Y_2}(L(Y_2)) = 0 $ whenever $Y_1 \neq Y_2$
\end{enumerate}
The implications $(1) \Rightarrow (2)$ and $(2)\Rightarrow (3)$ are immediate. Therefore, assume $(3)$. For all $v_1 \in F(X, Y_1)$ and $v_2 \in F(X,Y_2)$, we find that $v_1 \cdot v_2^\dagger = 1 \cdot v_1 \cdot v_2^\dagger \cdot 1 = v_1 \cdot \phi_X^{Y_1}(1) \cdot \phi_X^{Y_2}(1) \cdot v_2^\dagger = 0$, whenever $Y_1 \neq Y_2$. Thus, $(3) \Rightarrow (1)$.

The families of subspaces satisfying condition (1) are exactly the partial functions from $\X$ to $\Y$, essentially by definition. The families of coisometries satisfying condition (2) are exactly the partial fissions from $\X$ to $\Y$, by definition. The families of $*$-homomorphisms satisfying condition (3) correspond bijectively to the continuous $*$-homomorphisms $\phi:\ell(\Y) \To \ell(\X)$ via the equation
$\phi(b)(X) = \sum_{Y \atomof \Y} \phi_X^Y(b(Y))$
for $X \atomof \X$, by Lemma \ref{Q3.1}.
\end{proof}

\begin{lemma}\label{Q8}
Let $\X$ and $\Y$ be quantum sets, and let $F$ be a partial function from $\X$ to $\Y$. Let $f$ be the corresponding partial fission, and let $\phi$ be the corresponding $*$-homomorphism, in the sense of Theorem \ref{Q3}.
Then, the following are equivalent
\begin{enumerate}
\item $F^\dagger \circ F \geq I$
\item For all $X \atomof \X$, we have $\sum_{Y \atomof \Y} {f_X^{Y}}^\dagger f_X^Y = 1$.
\item $\phi(1) = 1$ 
\end{enumerate}
\end{lemma}

A partial fission satisfying condition (2) is thus reasonably called simply a \emph{fission}.

\begin{proof}
Condition (2) is easily seen to be equivalent to condition (3):
\begin{align*}
\phi(1) = 1
\;&\Leftrightarrow\;
\forall X \atomof \X\: \;\phi(1)(X) = 1_X
\\ &\Leftrightarrow\;
\forall X \atomof \X\: \; \sum_{Y \atomof \Y} f_X^{Y\,\dagger}(1 \tensor 1) f_X^Y =1
\\ &\Leftrightarrow\;
\forall X \atomof \X\: \; \sum_{Y \atomof \Y} f_X^{Y\,\dagger} f_X^Y =1
\end{align*}

To demonstrate the equivalence between condition (1) and condition (2), choose bases $(B_X^Y)$ for the spaces $(F(X,Y))$. Condition (1) is equivalent to $\sum_{Y \atomof \Y} F(X,Y)^\dagger \cdot F(X,Y) \ni 1_X$ for all $X \atomof \X$, and condition (2) is equivalent to the equation 
$\sum_{Y \atomof \Y} \sum_{v \in B_X^Y} v^\dagger v =1_X$ for all $X \atomof \X$. Thus, condition (2) clearly implies condition (1).

So, assume condition (1), and fix $X \atomof \X$. The definition of a function implies that $(F(X,Y)^\dagger \cdot F(X,Y) \suchthat Y \atomof \Y)$ is a family of pairwise-orthogonal $*$-subalgebras of $L(X)$. Thus, the identity operator $1_X$ can be written as a sum of orthogonal projections $\sum_{Y \atomof \Y} p_Y$, with $p_Y$ the identity of $F(X,Y)^\dagger \cdot F(X,Y)$. Fix $Y \atomof \Y$, and for convenience, enumerate the orthonormal basis $B_X^Y = \{v_1, v_2, \ldots, v_n\}$. Since $p_Y  \in F(X,Y)^\dagger \cdot F(X,Y)$, we can write $p_Y = \sum_{i,j=1}^n  \alpha_{ij} v_i^\dagger v_j$ for some doubly indexed family of complex numbers $(\alpha_{ij} \suchthat 1\leq i,j \leq n)$, and the orthonormality of the basis easily implies that $\alpha_{ij}$ is the Dirac delta symbol. Thus, $p_Y = \sum_{i=1}^n v_i^\dagger v_i = \sum_{v \in B_X^Y} v^\dagger v$. We vary $Y \atomof \Y$ to conclude that $1_X = \sum_{Y \atomof \Y} \sum_{v \in B_X^Y} v^\dagger v$, as desired.
\end{proof}

\section{Functoriality}\label{section 7}

We now define the composition of partial fissions, to prove that the bijective correspondences in Theorem \ref{Q3} are functorial.

\begin{definition}\label{Q4}
Let $\X$, $\Y$, and $\Z$ be quantum sets. Let $f$ be a partial fission from $\X$ to $\Y$, and let $g$ be a partial fission from $\Y$ to $\Z$. Explicitly, $f$ has components $f_X^Y \: X \To Y \tensor H_X^Y$, and $g$ has components $g_Y^Z\: Y \To Z \tensor K_Y^Z$. For all $X \atomof \X$ and $Z \atomof \Z$, the Hilbert space $L_X^Z = \bigoplus_{Y \atomof \Y} K_Y^Z \otimes H_X^Y$ is finite-dimensional (Definition \ref{Q2}). The \emph{composition} $g \circ f$ is the partial fission from $\X$ to $\Z$ whose components $(g \circ f)_X^Z \: X \To Z \tensor L_X^Z$ are defined by
\begin{equation*}
(g \circ f)_X^Z = \sum_{Y \atomof \Y} (g^Z_Y \tensor 1)\cdot f_X^Y.
\end{equation*}
Each individual term $(g^Z_Y \tensor 1)\cdot f_X^Y$ is a coisometry from $X$ to $Z \tensor K^Z_Y \tensor H_X^Y$, so the sum has finitely many nonzero terms and is an operator from $X$ to $\bigoplus_{Y \atomof \Y} Z \tensor K^Z_Y \tensor H_X^Y = Z \tensor L_X^Z$. 
For fixed $X \atomof \X$, it is easy to see that $[(g^Z_{Y} \tensor 1)\cdot f_X^{Y}] \cdot [ (g^{Z'}_{Y'} \tensor 1)\cdot f_X^{Y'}]^\dagger = 0$ unless $Y = Y'$ and $Z = Z'$. This implies that $(g \circ f)_X^Z$ is a coisometry for each $Z \atomof \Z$, and furthermore that $[(g \circ f)_X^Z]\cdot [(g \circ f)_X^Z]^\dagger = 0$ unless $Z = Z'$. We vary $X$ to conclude that $g \circ f$ is a partial fission.
\end{definition}

\begin{proposition}\label{Q5}
The bijective correspondences of Theorem \ref{Q3} are functorial. Explicitly: Let $\X$, $\Y$, and $\Z$ be quantum sets. Let $F$ be a partial function from $\X$ to $\Y$ with corresponding partial fission $f$ and homomorphism $\phi$. Similarly, let $G$ be a partial function from $\Y$ to $\Z$ with corresponding partial fission $g$ and homomorphism $\psi$. It follows that the partial function $G \circ F$ has corresponding partial fission $g \circ f$, and corresponding homomorphism $\phi \circ \psi$.
\end{proposition}

\begin{proof}
By definition of the corresponding homomorphisms $\phi$ and $\psi$, given in the statement on Theorem \ref{Q3}, we have that for all $c \in \ell^\infty(\Z)$ and all $X \atomof \X$,
\begin{align*} \phi (\psi(c))(X) & = \sum_{Y \atomof \Y} \sum_{Z \atomof \Z} f_X^{Y \, \dagger} (g_Y^{Z \, \dagger}(c(Z) \tensor 1)g_Y^Z \tensor 1) f_X^Y \\ &=  \sum_{Y \atomof \Y} \sum_{Z \atomof \Z} f_X^{Y \, \dagger} (g_Y^{Z \, \dagger}\tensor 1) (c(Z) \tensor 1 \tensor 1)(g_Y^Z \tensor 1) f_X^Y.
\end{align*}
The second tensor factor in $c(Z) \tensor 1 \tensor 1$ is identity on $K^Z_Y$, and the third is identity on $H_X^Y$. On the other hand, the homomorphism $\theta$ corresponding to $g \circ f$ is defined by
\begin{align*}
\theta(c)(X)
&=
\sum_{Z \atomof \Z} (g \circ f)_X^{Z\, \dagger}(c(z) \tensor 1) (g \circ f)_x^Z
\\ &=
\sum_{Z \atomof \Z} \left[ \left( \sum_{Y \atomof \Y}    (g^Z_Y \tensor 1) f_X^Y \right)^\dagger\cdot (c(Z) \tensor 1) \cdot
\left(\sum_{Y' \atomof \Y} (g^Z_{Y'} \tensor 1) f_X^{Y'}\right)\right]
\\ & =
\sum_{Z \atomof \Z} \sum_{Y \atomof \Y} \sum_{Y' \atomof \Y} f_X^{Y \, \dagger} (g_Y^{Z \, \dagger}\tensor 1) (c(Z) \tensor 1 )(g_{Y'}^Z \tensor 1) f_X^{Y'}.
\end{align*}
The second tensor factor in $c(Z) \tensor 1$ is identity on $L_X^Z = \bigoplus_{Y \atomof \Y} K_Y^Z \otimes H_X^Y$. Terms with $Y \neq Y'$ do not contribute because the ranges of $(g_{Y}^Z \tensor 1) f_X^{Y}$ and $(g_{Y'}^Z \tensor 1) f_X^{Y'}$ are $Z \tensor K_Y^Z \tensor H_X^Y$ and $Z \tensor K_{Y'}^Z \tensor H_X^{Y'}$ respectively, and each is invariant for $c(Z) \tensor 1$. On each summand $K_Y^Z \tensor H_X^Y$, the identity on $L_X^Z$ can also be written as the tensor product of the identity on $K_Y^Z$ and the the identity on $H_X^Y$, which appears in our expression for $\phi(\psi(c))$. Thus, we find that the two expressions are equal; in other words, $\phi \circ \psi = \theta$. We conclude that the correspondence between partial fissions and normal $*$-homomorphisms is functorial, contravariantly.

We now argue that the first correspondence in Theorem \ref{Q3} is also functorial. Fix $X \atomof \X$ and $Z \atomof \Z$. For each $Y \atomof \Y$, the composition map $G(Y,Z) \tensor F(X,Y) \To (G \circ F)(X,Z)$ is an isometry, as a direct consequence of the definition of the inner products on these three spaces. Furthermore, for distinct $Y, Y' \atomof \Y$, the ranges of the corresponding composition maps are orthogonal in $(G \circ F)(X,Z)$, because $F(X,Y) \cdot F(X,Y')^\dagger = 0$, as an immediate consequence of the definition of partial function. Together, the ranges of the composition maps $G(Y,Z) \tensor F(X,Y) \To (G \circ F)(X,Z)$, for $Y \atomof \Y$, span $(G \circ F)(X,Z)$, so they form an orthogonal decomposition of this Hilbert space. In other words, these composition maps together form a unitary operator $\overline u_X^Z$ from $\bigoplus_{Y \atomof \Y} G(Y,Z) \tensor F(X,Y)$ to $(G \circ F)(X,Z)$.

For each $Y \atomof \Y$, choose an orthonormal basis $B_X^Y$ for $F(X,Y)$, and an orthonormal basis $C_Y^Z$ for $G(Y,Z)$. For all $x \in X$, we have
\begin{align*}
(g \circ f)_X^Z(x)
& =
\sum_{Y \atomof \Y} (g_Y^Z \tensor 1) (f_X^Y(x))
\\ & =
\sum_{Y \atomof \Y} \sum_{v \in B_X^Y}
(g^Z_Y \tensor 1) ( v(x) \tensor v^\dagger)
\\ & = 
\sum_{Y \atomof \Y} \sum_{v \in B_X^Y} \sum_{w \in C_Y^Z} w(v(x)) \tensor w^\dagger \tensor v^\dagger
\end{align*}

The set $\{ w \tensor v\suchthat Y \atomof \Y,\, w \in C_Y^Z,\, v \in B_X^Y\}$ is an orthonormal basis of $\bigoplus_{Y \atomof \Y} G(Y,Z) \tensor F(X,Y)$, so the set $\{\overline u_X^Z( w \tensor v) \suchthat Y \atomof \Y,\,w \in C_Y^Z,\, v \in B_X^Y)\}$ is an orthonormal basis of $(G \circ F)(X, Z)$. By definition of $\overline u_X^Z$, we have that $\overline u_X^Z(w \tensor v) = w v$, so the partial fission $h$ corresponding to $G \circ F$ satisfies
\begin{equation*}h_X^Z(x) = \sum_{Y \atomof \Y} \sum_{v \in B_X^Y} \sum_{w \in C_Y^Z} w(v(x)) \tensor (wv)^\dagger.
\end{equation*}

Allowing $X$ and $Z$ to vary, we identify the partial fission $h$ and the partial fission $g \circ f$ according to Definition \ref{Q2}. Explicitly, we can use the family of unitaries defined by $u_X^Z = \dagger \circ \overline u _X^Z \circ (\dagger \tensor \dagger)$.
\end{proof}

\begin{definition}\label{Q11}
Let $\X$ and $\Y$ be quantum sets. Let $F\: \X \To \Y$ be a partial function. Write $F^\star\: \ell(\Y) \To \ell(\X)$ for the continuous $*$-homomorphism $\phi$ that corresponds to $F$ in the sense of Theorem \ref{Q3}.
\end{definition}

We will also write $F^\star$ for the restriction of $F^\star\: \ell(\Y) \To \ell(\X)$ to a normal $*$-homomorphism $\ell^\infty(\Y) \To \ell^\infty(\X)$, which exists by Lemma \ref{Q3.1}.

\begin{theorem}\label{Q10}
There is a contravariant monoidal equivalence from $\qPar$ to $\Mstar0$, which restricts to a contravariant monoidal equivalence from $\qSet$ to $\Mstar1$. It takes each quantum set $\X$ to the von Neumann algebra $\ell^\infty(\X)$, and it takes each function $F\: \X \To \Y$ to the normal $*$-homomorphism $F^\star\: \ell^\infty(\Y) \To \ell^\infty(\X)$.
This a monoidal functor in the strong sense, for the Cartesian product of quantum sets and the spatial tensor product of von Neumann algebras. This is an equivalence of categories in the weak sense that it is a full and faithful functor with the property that every hereditarily atomic von Neumann algebra is isomorphic to something in its image.
\end{theorem}

The monoidal structure on $\qPar$ is just the restriction of the monoidal structure on $\qRel$, defined in section \ref{section 3}: The monoidal product of two quantum sets $\X$ and $\Y$ is their generalized Cartesian product $\X \times \Y = \Q\{X \tensor Y \suchthat X \atomof \X,\, Y \atomof \Y\}$, and the monoidal unit is the quantum set $\mathbf 1 =\Q \{\CC\}$. The monoidal product of partial functions $F_1\: \X_1 \to \Y_1$ and $F_2\: \X_2 \To \Y_2$ is defined by $(F_1 \times F_2)(X_1 \tensor X_2, Y_1 \tensor Y_2) = F_1(X_1, Y_1) \tensor F_2(X_2, Y_2)$, for $X_1 \atomof \X_1$, $X_2 \atomof \X_2$, $Y_1 \atomof \Y_1$, and $Y_2 \atomof \Y_2$.

The monoidal product of two von Neumann algebras $A \subsetof L(H)$ and $B \subsetof L(K)$ is their spatial tensor product $A \tensor B = \{a \tensor b \in L(H \tensor K) \suchthat a \in A,\, b \in B \}''$, and the monoidal unit is the one-dimensional von Neumann algebra $L(\CC)$. The monoidal product of normal $*$-homomorphisms $\phi_1\: A _1 \To B_1$ and $\phi_2\: A _2 \To B_2$ is defined by $(\phi_1 \tensor \phi_2)(a_1 \tensor a_2) = \phi_1(a_1) \tensor \phi_2(a_2)$, for all $a_1 \in A_1$ and $a_2 \in A_2$. For general von Neumann algebras, the existence of such a normal $*$-homomorphism $\phi_1 \tensor \phi_2$ follows from the special case of injective $\phi_1$ and $\phi_2$, and the special case of surjective $\phi_1$ and $\phi_2$.

The spatial tensor product of two hereditarily atomic von Neumann algebras is equal to their categorical tensor product, because each is an $\ell^\infty$-direct sum of type I factors \cite{Guichardet66}. In other words, if $A_1$ and $A_2$ are hereditarily atomic von Neumann algebras, and $B$ is any von Neumann algebra, then a pair of normal $*$-homomorphisms $\phi_1\: A_1 \To B$ and $\phi_2\: A_2 \To B$ will factor through $A_1 \tensor A_2$ if and only if every operator in the image of $\phi_1$ commutes with every operator in the image of $\phi_2$.

\begin{proof}[Proof of Theorem \ref{Q10}]
Theorem \ref{Q3} and Proposition \ref{Q5} give us a contravariant equivalence from the category of quantum sets and partial functions, to a category of topological $*$-algebras and continuous $*$-homomorphisms. The contravariant equivalence maps each function $F\: \X \To \Y$ to the continuous $*$-homomorphism $\phi\: \ell(\Y) \to \ell(\X)$, which restricts to a normal $*$-homomorphism $\ell^\infty(\Y) \To \ell^\infty(\X)$ by Lemma \ref{Q3.1}. This defines a contravariant equivalence $\qPar \To \Mstar0$ by Proposition \ref{Q} and Lemma \ref{Q3.1}, which restricts to a contravariant equivalence $\qSet \To \Mstar1$ by Lemma \ref{Q8}. Both equivalences are monoidal because the von Neumann algebra isomorphism $\ell^\infty(\X) \tensor \ell^\infty(\Y) \To \ell^\infty(\X \times \Y)$ defined by $a \tensor b \mapsto (X \tensor Y \mapsto a(X) \tensor b(Y))$ is natural by Lemma \ref{Q7}.
\end{proof}

\begin{proposition}\label{Q12}
Let $\X$ and $\Y$ be quantum sets. Let $F\: \X \To \Y$ be a partial function. For each $X \atomof \X$ and each $Y \atomof \Y$, the vector space $L(X, Y)$ is canonically a Hilbert space with the inner product $\< w |  w'\> = \Tr(w^\dagger w')$. If  $(B_X^Y\suchthat X \atomof \X,\, Y \atomof \Y)$ is any family of orthonormal bases for the Hilbert spaces $(F(X,Y)\suchthat X \atomof \X,\, Y \atomof \Y)$ for these inner products, then for all $b \in \ell(\Y)$ and all $X \atomof \X$, we have \begin{equation*} F^\star(b)(X) =  \sum_{Y \atomof \Y}  \sum_{w \in B_X^Y} \mathrm{dim}(Y)\cdot w^\dagger b(Y)  w.
\end{equation*}
For fixed $b$ and $X$, only finitely many of the terms in the above sum are nonzero.
\end{proposition}

\begin{proof}
Fix $b \in \ell(\Y)$ and $X \atomof \X$. Choose orthonormal bases $(B_X^Y \suchthat Y \atomof \Y)$ for the Hilbert spaces $(F(X,Y)\suchthat Y \atomof Y)$ equipped with the inner products defined by $\< w |  w'\> = \Tr(w^\dagger w')$. Referring to the statement of Theorem \ref{Q3}, we find that for all $w, w' \in F(X,Y)$,
\begin{equation*}\< w |  w'\> = \Tr(w^\dagger w') = \Tr(w' w^\dagger) = \Tr((w| w') \cdot 1_Y) = (w| w') \cdot \dim(Y).\end{equation*}
It follows that for each $Y \atomof \Y$, the set $\hat B_X^Y = \{w \cdot \sqrt{\mathrm{dim}(Y)} \suchthat w \in B_X^Y\}$ is an orthonormal basis for the inner product $(\, \cdot\, |\, \cdot\,)$. We now calculate that for all $x, x' \in X$,
\begin{align*}\< x \suchthat \phi(b)(X) x' \>
& = 
\sum_{Y \atomof \Y} \< f_X^Y x | (b(Y) \tensor 1) f_X^Y x' \>
 \\ &=
\sum_{Y \atomof \Y} \sum_{v \in \hat B_X^Y} \sum_{v' \in \hat B_X^Y} \< v(x) \tensor v^\dagger | b(Y) v'(x') \tensor {v'}^\dagger \>
\\ &=
\sum_{Y \atomof \Y} \sum_{v \in \hat B_X^Y} \sum_{v' \in \hat B_X^Y} \< v (x)| b(Y) v' (x') \> \cdot  ( v^\dagger | {v'}^\dagger)
\\ &= \sum_{Y \atomof \Y} \sum_{v \in \hat B_X^Y} \< v (x)| b(Y) v (x') \>
=
\sum_{Y \atomof \Y} \sum_{v \in \hat B_X^Y} \< x| v^\dagger b(Y) v x' \>.
\end{align*}
Therefore,
\begin{equation*}\phi(b)(X) 
=
\sum_{Y \atomof \Y} \sum_{v \in \hat B_X^Y} v^\dagger b(Y) v 
=
\sum_{Y \atomof \Y} \sum_{w \in B_X^Y} \dim(Y)\cdot  w^\dagger b(Y) w. 
\end{equation*}
\end{proof}

\section{Completeness and cocompleteness}\label{section 8}

In this section, we establish the basic properties of the category $\qSet$ of quantum sets and functions, by leveraging its contravariant duality with the category $\Mstar1$ of hereditarily atomic von Neumann algebras and unital normal $*$-homomorphisms (Theorem \ref{Q10}). The arguments follow those of \cite{Kornell17}; their applicability relies on the trivial fact that any von Neumann subalgebra of a hereditarily atomic von Neumann algebra is itself hereditarily atomic. As a consequence, any quotient of a hereditarily atomic von Neumann algebra by an ultraweakly closed $*$-ideal is also hereditarily atomic.

\begin{proposition}\label{Y}
Let $F$ be a function from a quantum set $\X$ to a quantum set $\Y$. The following are equivalent:
\begin{enumerate}
\item $F$ is surjective
\item $F$ is epic in $\qSet$
\item $F^\star$ is monic in $\Mstar1$
\item $F^\star$ is injective as a map $\ell^\infty(\Y) \To \ell^\infty(\X)$
\end{enumerate}
\end{proposition}

\begin{proof}
$(1) \Yields (2)$: Assume that $F$ is surjective. Then, $F$ is both surjective and coinjective as a binary relation, so $F \circ F^\dagger = I_\Y$. For all functions $G_1$ and $G_2$ from $\Y$ to a quantum set $\Z$, if $G_1 \circ F = G_2 \circ F$, then $G_1 \circ F \circ F^\dagger = G_2 \circ F \circ F^\dagger$, so $G_1 = G_2$. Thus, $F$ is epic.

The equivalence $(2) \Equivalent (3)$ is immediate from the duality between $\qSet$ and $\Mstar1$.

$(3) \Yields (4)$: Assume $F^\star$ is not injective. Then, its kernel is equal to $(1- p)\cdot\ell^\infty(\Y)$ for some central projection $p \neq 1$. For $i \in \{1,2\}$, define $\phi_i\:\CC \oplus \CC \oplus \CC \To \ell^\infty(\Y)$ to be the unital normal $*$-homomorphism that takes the $i$-th minimal projection in $\CC \oplus \CC \oplus \CC$ to $1-p$ and takes the third minimal projection to $p$. Immediately, for $i \in \{1,2\}$, the homomorphism $F^\star \circ \phi_i\: \CC \oplus \CC \oplus \CC \To \ell^\infty(\X)$ takes the third minimal projection to the identity, and it is completely determined by this property. Thus, $F^\star$ is not monic.

$(4) \Yields (1)$: Assume $F$ is not surjective. This assumption is equivalent to the strict inequality $F \circ F^\dagger < I_\Y$, so there is some atom $Y_0 \atomof \Y$ such that $(F\circ F^\dagger)(Y_0,Y_0)$ = 0. The definition of composition then implies that $F(X,Y_0) = 0$ for all atoms $X \atomof \X$. Let $q \in \ell^\infty(\Y)$ be the projection defined by $q(Y_0) = 1_{Y_0}$, and $q(Y) = 0$ for $Y \neq Y_0$. The characterization of $F^\star$ given in Proposition \ref{Q12} immediately implies that $F^\star(q) = 0$, so $F^\star$ is not injective.
\end{proof}

\begin{definition}\label{Y1}
Let $\Y$ be a quantum set, and let $\X$ be a subset of $\Y$. We define the \emph{inclusion function} ${J_\X^\Y \: \X \hookrightarrow \Y}$ by $J_\X^\Y(X, Y) = \CC \cdot 1_X$ if $X = Y$, with $J_\X^\Y(X, Y)$ vanishing otherwise.
\end{definition}

\begin{lemma}\label{Y2}
Write $J = J_\X^\Y$. The unital normal $*$-homomorphism $J^\star\: \ell^\infty(\Y) \To \ell^\infty(\X)$ is defined by  $J^\star(b)(X) = b(X)$, for $X \atomof \X$. It is surjective, and its kernel is $(1-p)\cdot \ell^\infty(\Y)$, where $p$ is the central projection defined by $p(Y) = 1_Y$ for $Y \atomof \X$, with $p(Y) = 0$ otherwise.
\end{lemma}

\begin{proof}
For all $X \atomof \X$, the Hilbert space $F(X,X)$ is one-dimensional. The norm of the identity operator $1_X$ as a vector in this Hilbert space is $\sqrt{\mathrm{dim}(X)}$, so $\{1_X/\sqrt{\mathrm{dim}(X)}\}$ is an orthonormal basis for $F(X,X)$. For $X \neq Y$, the Hilbert space $F(X,Y)$ is zero-dimensional. Applying Proposition \ref{Q12}, we find that for all $b \in \ell^\infty(\Y)$ and all $X \atomof \X$, we have $J^\star(b)(X) = 1_X^\dagger b(X) 1_X = b(X)$ as claimed. It follows immediately that $J^\star$ is surjective. We also see that an operator $b$ is in the kernel of $J^\star$ if and only if $b(Y) = 0$ for all $Y \atomof \X$. In other words, $b$ is in the kernel of $J^\star$ if and only if $p \cdot b = 0$.
\end{proof}

\begin{proposition}\label{Z}
Let $F$ be a function from $\X$ to $\Y$. The following are equivalent:
\begin{enumerate}
\item $F$ is injective
\item $F$ is  monic in $\qSet$
\item $F^\star$ is epic in $\Mstar1$
\item $F^\star$ is surjective as a map $\ell^\infty(\Y) \To \ell^\infty(\X)$
\end{enumerate}
\end{proposition}

\begin{proof} We follow the pattern of the previous proof.

$(1) \Yields (2)$: Assume that $F$ is injective. Then, $F$ is both injective and cosurjective as a binary relation, so $F^\dagger \circ F = I_\X$. For all functions $G_1$ and $G_2$ from a quantum set $\Z$ to $\X$, if $F \circ G_1 = F\circ G_2$, then $F^\dagger \circ F \circ G_1 = F^\dagger \circ F \circ G_2$, so $G_1 = G_2$. Thus, $F$ is monic.

The equivalence $(2) \Equivalent (3)$ is immediate from the duality between $\qSet$ and $\Mstar1$.

$(3) \Yields (4)$: Assume $F^\star$ is not surjective. Then, $F^\star(\ell^\infty(\Y))$ is a proper ultraweakly closed $*$-subalgebra of $\ell^\infty(\X)$ that contains the identity of $\ell^\infty(\X)$, and it is possible to find an automorphism of $\ell^\infty(\X)$ that fixes the operators in $F^\star(\ell^\infty(\Y))$ (Lemma \ref{1}). This demonstrates that $F^\star$ is not epic.

$(4) \Yields (1)$: Assume that $F^\star$ is surjective. The kernel of $F^\star$ is equal to $(1-p) \cdot \ell^\infty(\Y)$ for some central projection $p$. The central projections of $\ell^\infty(\Y)$ are clearly in bijective correspondence with the subsets of $\Y$, with $p$ corresponding to the subset $\Z = \Q\{Y \atomof \Y \suchthat  p(Y) = 1_Y\}$. Let $J = J_\Z^\Y\: \Z \hookrightarrow Y$ be the corresponding inclusion map (Definition \ref{Y1}). By Lemma \ref{Y2}, the functions $F^\star\: \ell^\infty(\Y) \To \ell^\infty(\X)$, and $J^\star\: \ell^\infty(\Y) \To \ell^\infty(\Z)$ have the same kernel, so they factor through each other via an isomorphism of Neumann algebras. Dually, the functions $F\: \X \To \Y$ and $J\: \Z \To \Y$ must also factor through each other via an isomorphism in $\qSet$. Thus, it is sufficient to show that $J\: \Z \To \Y$ is injective (Proposition \ref{K}). This is straightforward.
\end{proof}

\begin{lemma}\label{1}
Let $A$ be an atomic von Neumann algebra, and let $B$ be a proper von Neumann subalgebra of $A$ that contains the unit of $A$. There is a nontrivial automorphism $\phi$ of $A$ that fixes each operator in $B$.
\end{lemma}

\begin{proof}
Being an atomic von Neumann algebra, $A$ is the $\ell^\infty$-direct sum of type I factors: $A = \bigoplus_i A_i$. The inclusion homomorphism $\rho\: B \hookrightarrow A$ decomposes into unital normal $*$-homomorphisms $\rho_i\: B \To A_i$, which we may regard as representations of $B$.

If any such representation $\rho_{i_0}$ is reducible, then the bicommutant theorem yields a nontrivial unitary operator $u_{i_0} \in A_{i_0}$ that commutes with every operator in $\rho_i(B)$. The unitary operator $u \in A$ that is $u_{i_0}$ in the direct summand $A_{i_0}$ and the identity in every other direct summand is then in the commutant of $B$, but not in the center of $A$, so conjugation by $u$ is the desired automorphism.

If all the representations $\rho_i\: B \To A_i$ are irreducible, then some pair of them $\rho_{i_1}$ and $\rho_{i_2}$ must be unitarily equivalent, as otherwise the minimal central projections of $A$ would all be in $B$, implying that $B = A$. The automorphism of $A$ that exchanges the summands $A_{i_1}$ and $A_{i_2}$ according to this unitary equivalence is then a nontrivial automorphism of $A$ that fixes the elements of $B$.
\end{proof}

\begin{proposition}\label{2}
The category $\Mstar1$ of hereditarily atomic von Neumann algebras and unital normal $*$-homomorphisms is complete, so the category $\qSet$ of quantum sets and functions is cocomplete.
\end{proposition}

\begin{proof}
See propositions 5.1 and 5.3 in \cite{Kornell17}, which show that the category $\mathbf{W}^*_1$ of von Neumann algebras and unital normal $*$-homomorphisms has all products and all equalizers. The $\ell^\infty$-direct sum of any family of hereditarily atomic von Neumann algebras is clearly itself hereditarily atomic, so the subcategory of hereditarily atomic von Neumann algebras is closed under products. The equalizer of two unital normal $*$-homomorphism $B \rightrightarrows A$ is an ultraweakly closed $*$-subalgebra of $B$, and is therefore hereditarily atomic when $B$ if hereditarily atomic, so the subcategory of hereditarily atomic von Neumann algebras is also closed under equalizers.
\end{proof}

\begin{proposition}\label{3}
The category $\Mstar1$ of hereditarily atomic von Neumann algebras and unital normal $*$-homomorphisms is cocomplete, so the category $\qSet$ of quantum sets and functions is complete.
\end{proposition}

\begin{proof}
See propositions 5.5 and 5.7 in \cite{Kornell17}. The subcategory of hereditarily atomic von Neumann algebras is \emph{not} closed under coproducts: Davis showed that the von Neumann algebra of bounded operators on a separable infinite-dimensional Hilbert space is generated by three projections \cite{Davis55}, so the coproduct of three copies of $\CC^2$ in the category of all von Neumann algebras is not hereditarily atomic. Thus, we will need to modify the proof of the cited proposition 5.5 in \cite{Kornell17}, rather than appeal to the proposition itself.

For any family $(A_{j})$ of von Neumann algebras, we define a joint representation of $(A_j)$ to be a family unital normal $*$-homomorphisms $(\rho_j\: A_j \To L(H))$, for some Hilbert space $H$. We construct the universal hereditarily atomic joint representation $ (\pi_j)$ of $(A_j)$ by taking a direct sum of all finite-dimensional joint representations. This representation is hereditarily atomic in the sense that the von Neumann algebra $A = (\Union_j \pi_j(A_j))''$ generated by the set of all represented operators is hereditarily atomic, being a subalgebra of an $\ell^\infty$-direct sum of finite-dimensional von Neumann algebras.

The universal hereditarily atomic joint representation of $(A_j)$ is indeed universal among its hereditarily atomic joint representations, in the sense that every hereditarily atomic joint representation $(\rho_j)$ of $(A_j)$ factors uniquely through $A$ via a unital normal $*$-homomorphism: Since, by assumption, the von Neumann algebra $(\Union_j \rho_j(A_j))''$ is hereditarily atomic, we may assume without loss of generality that this von Neumann algebra is an irreducibly represented finite type I factor. In this case, the joint representation $(\rho_j)$ is unitarily equivalent to one of the summands of the universal hereditarily atomic joint representation, giving us the desired homomorphism as conjugation by an isometry. This normal $*$-homomorphism is unique, because its values on the generators of $A$ are determined by the joint representation $(\rho_j)$.

If $(A_j)$ is an indexed family of hereditarily atomic von Neumann algebras, then the algebra $A$ is their coproduct. Indeed, any cocone on the family $(A_j)$ in $\Mstar1$ to some hereditarily atomic von Neumann algebra $B$ is a hereditarily atomic joint representation of $(A_j)$, and therefore factors uniquely through $A$, as we have shown. The inclusions of the coproduct are exactly the representations $\pi_j$.

The category of hereditarily atomic von Neumann algebras also has coequalizers, because it is closed under coequalizers as a subcategory of the category of von Neumann algebras and unital normal $*$-homomorphisms. Indeed, the coequalizer of two unital normal $*$-homomorphisms is a quotient of the codomain von Neumann algebra, which is hereditarily atomic if the codomain is itself hereditarily atomic. Thus, the category $\Mstar1$ is cocomplete.\end{proof}

In an effort to emphasize the duality between our two categories, we will use the symbols $\oplus$, $\circledast$, and $\otimes$ for the product, the coproduct, and the monoidal product on the category $\Mstar1$, and we will use the symbols $\uplus$, $\ast$, and $\times$ for the coproduct, the product, and the monoidal product on the category $\qSet$. We will refer to the former three operations as the direct sum, the free product, and the tensor product of hereditarily atomic von Neumann algebras; we will refer to the latter three operations as the disjoint union, the total product, and the Cartesian product of quantum sets.
Readers with a background in type theory may choose to perfect this notational correspondence by using $+$ for the disjoint union of quantum sets.

We comment briefly on the \emph{exact} definitions of these six operations. The direct sum $\bigoplus_j A_j$ of an indexed family $(A_j)$ of hereditarily atomic von Neumann algebras is obtained in the standard way, by first taking the $\ell^2$-direct sum of their Hilbert spaces. Similarly, the tensor product $A_0 \otimes A_1$ of hereditarily atomic von Neumann algebras $A_0$ and $A_1$ is obtained by first taking the tensor product of the Hilbert spaces; this is the ``spatial'' tensor product, though by a result of Guichardet \cite{Guichardet66}*{proposition 8.6}, it coincides with the ``categorical'' tensor product. The free product $\bigoast_j A_j$ is obtained as in the proof of Proposition \ref{3}, with the universal hereditarily atomic joint representation constructed as the direct sum of irreducible representations on Hilbert spaces of the form $\CC^n$, so such irreducible representations form a set. The disjoint union $\biguplus_j \X_j$ of an indexed family $(\X_j)$ of quantum sets is the union $\bigcup_j \X_j \times `\{ j\}$, and the Cartesian product $\X_0 \times \X_1$ is exactly as we defined it in section 2. The total product of an indexed family of quantum sets is only defined up to natural isomorphism.

\section{Quantum function sets}\label{section 9}

\begin{theorem}\label{5}
The category $\qSet$ of quantum sets and functions, equipped with the monoidal product $\times$, is a closed symmetric monoidal category.
\end{theorem}

\begin{proof}

See theorem 9.1 in \cite{Kornell17}. Our construction here follows the same basic pattern as the construction of the total product in the proof of proposition 8.5: we work in the category $\Mstar1$ of hereditarily atomic von Neumann algebras and unital normal $*$-homomorphisms, and we construct the desired universal object by taking a direct sum of finite-dimensional representations of the appropriate kind. Our task is to prove the following claim: for all hereditarily atomic von Neumann algebras $A$ and $B$, there is a hereditarily atomic von Neumann algebra $B^{\oast A}$ and unital normal $*$-homomorphism $\varepsilon\: B \To B^{\oast A} \otimes A$ that is universal among unital normal $*$-homomorphisms $\phi\: B \To C \otimes A$, for $C$ a hereditarily atomic von Neumann algebra, in the sense that there is a unique unital normal $*$-homomorphism $\pi\: B^{\oast A} \To C$ such that $(\pi \tensor \mathnormal 1_A) \circ \varepsilon = \phi$:
$$
\begin{tikzcd}
B \arrow{r}{\varepsilon} \arrow{rd}[swap]{\phi} &  B^{\oast A} \otimes A
\arrow[dotted]{d}{\pi  \otimes \mathnormal 1 }\\
& C \otimes A
\end{tikzcd}
$$
The shortest proof of this claim is essentially that of theorem 9.1 in \cite{Kornell17}. We take a slightly longer route, which I hope the reader will find more intuitive.

We first prove the above claim just in the special case that $A$ is a matrix algebra, i.e., $A = L(\CC^d)$. In this case, we construct $\varepsilon$ as the direct sum of all representations of $B$ on a Hilbert space of the form $\CC^n \tensor \CC^d$, for $n$ a nonnegative integer. Indexing the direct summands, we have a representation of $B$ on the Hilbert space $\bigoplus_i (\CC^{n_i} \otimes \CC^d) \iso  \left(\bigoplus_i \CC^{n_i} \right) \otimes \CC^d$. We now define the von Neumann algebra $B^{\oast L(\CC^d)}$ to be the smallest von Neumann algebra on the Hilbert space $\bigoplus_i \CC^{n_i} $ such that $\varepsilon( B) \subsetof B^{\oast L(\CC^d)} \otimes L(\CC^d)$. Certainly, $\varepsilon(B) \subsetof \bigoplus_i L(\CC^{n_i} \tensor \CC^d)  = \bigoplus_i (L(\CC^{n_i}) \otimes L(\CC^d)) \iso  \left(\bigoplus_i L(\CC^{n_i}) \right) \otimes L(\CC^d)$, so the von Neumann algebra $B^{\oast L(\CC^d)}$ is a subalgebra of the hereditarily atomic von Neumann algebra $\bigoplus_i  L(\CC^{n_i})$; it is therefore itself hereditarily atomic. It is not immediately apparent that $B^{\oast L(\CC^d)}$ is well-defined, in other words, that there is a smallest such von Neumann algebra. We may construct it by taking generators of the form $(\mathnormal 1 \tensor \Tr) ((1 \tensor a)\varepsilon (b))$, for $a \in L(\CC^d)$ and $b \in B$.

Let $\phi\: B \To C \tensor L(\CC^d) $ be any unital normal $*$-homomorphism, with $C$ a hereditarily atomic von Neumann algebra. Since $C$ is hereditarily atomic, we may assume without loss of generality that $C = \bigoplus_j L(\CC^{m_j})$, for some indexed family of positive integers $(m_j)$. We then have that $\phi$ is a map from $B$ to $\bigoplus_j  L(\CC^{m_j} \otimes \CC^d) $; in particular, we may view it as a direct sum of representations of $B$ on Hilbert spaces of the form $\CC^m \otimes \CC^d$. A unital normal $*$-homomorphism $\pi\: B^{\oast L(\CC^d)} \To C$ making the diagram commute can now be constructed by matching representations. It is unique because its values are determined on the generators of $B^{\oast L(\CC^d)}$:
\begin{align*} \pi((\mathnormal 1 \tensor \Tr ) &((1 \tensor a )\varepsilon (b)) ) 
=
(\pi \tensor \Tr) (( 1 \tensor a )\varepsilon (b))
\\ &= (\mathnormal 1 \tensor \Tr ) ((1 \tensor a) ( \pi \tensor \mathnormal 1) (\varepsilon (b))) =(1 \tensor \Tr ) ((1 \tensor a) \phi(b))
\end{align*}

We have proved the desired claim in the special case that $A$ is a matrix algebra. Translating this conclusion to the symmetric monoidal category $\qSet$ of quantum sets and functions equipped with the Cartesian product, we find that the functor $( \, -\, \times \X)$ has a right adjoint whenever $\X$ is atomic, i.e., whenever it has exactly one atom. In any symmetric monoidal category, this implies that arbitrary coproducts of such objects $\X$ enjoy the same property. Since, every quantum set is the coproduct of a family of atomic quantum sets, we have succeded in showing that $\qSet$ is a closed symmetric monoidal category.
\end{proof}

\begin{definition}\label{6}
Let $\X$ be a quantum set. We write $(\,-\,)^\X $ for the right adjoint of the functor $(\,-\,) \times \X$, where both functors are from the category $\qSet$ to itself. For any quantum set $\Y$, we call $\Y^\X$ the \emph{quantum function set} from $\X$ to $\Y$.
\end{definition}

Closed symmetric monoidal categories are common, and their basic properties are widely known. We review some of this basic theory in the context of our discussion of quantum sets. The expression $\Y^\X$ defines a functor, contravariant in $\X$, and covariant in $\Y$. Left adjoints preserve colimits and right adjoints preserve limits; thus, the functor $(\,-\,) \times \X$ preserves colimits, and the functor $(\,-\,)^\X$ preserves limits. In particular $(\Y_1 \ast \Y_2)^\X\iso \Y_1^\X \ast \Y_2 ^\X$; the same principle holds for total products of infinite families. The adjunction between the functors $(\,-\,) \times \X$ and $(\,-\,)^\X$ can be internalized: we have an isomorphism $(\Z^\Y)^\X \iso \Z^{\Y \times \X}$, natural in all variables. We also have a natural isomorphism $\Y^{\X_1 + \X_2} \iso \Y^{\X_1} \ast \Y^{\X_2}$; the same principle holds for disjoint unions of infinite families. In particular, if a quantum set $\X$ is decomposed as the union of its atomic subsets, $\X = \bigcup_{X \atomof \X} \Q\{X\}$, then
\begin{equation*} \Y^\X \iso \bigast_{X \atomof \X} \Y^{\Q\{X\}}.
\end{equation*}
The computation that establishes this natural isomorphism implicitly appears in the last step of our proof of theorem 9.1.

Each quantum set $\X$ has a maximum classical subset which consists of just its one-dimensional atoms, and which is naturally isomorphic to $`\mathrm{Fun}(\mathbf{1}; \X)$. In particular, the classical subset of $\Y^\X$ is naturally isomorphic to $`\mathrm{Fun}(\mathbf{1}; \Y^\X) \iso `\mathrm{Fun}(\X; \Y)$. Identifying $\Set$ with a subcategory of $\qSet$ via the functor $S \mapsto `S$, we might say that the classical subset of $\Y^\X$ consists of functions from $\X$ to $\Y$.

\begin{proposition}\label{7}
The category $\qSet$ of quantum sets and functions is not Cartesian closed.
\end{proposition}

Due to a clash of terminology, this proposition refers to the absence of right adjoints for functors of the form $(-) \ast \X$, not for functors of the form $(-) \times \X$. The product $\ast$ is Cartesian in the category-theoretic sense: it has the same universal property in the category $\qSet$ as the ordinary Cartesian product has in the category $\Set$. However, it is the product $\times$ that generalizes the ordinary Cartesian product in a manner appropriate to the noncommutative dictionary.
To begin with, the product $\times$ is a quantum generalization of the ordinary Cartesian product in the sense that $`S \times `T$ is naturally isomorphic to  $`(S \times T)$ for ordinary sets $S$ and $T$, and $\ast$ is not a quantum generalization of the ordinary Cartesian product in this sense. The product $\times$ is also effectively the established generalization of the ordinary Cartesian product in noncommutative geometry. For example, for any quantum set $\X$, a quantum group comultiplication on the von Neumann algebra $\ell^\infty(\X)$ is a unital normal $*$-homomorphism $\ell^\infty(\X) \To \ell^\infty(\X) \tensor \ell^\infty(\X)$, so dually a quantum group multiplication on $\X$ is a function $\X \times \X \To \X$, not a function $\X \ast \X \To \X$. We now also see that evaluation, the counit of the adjunction in Definition \ref{6}, is a function $\Y^\X \times \X \To \Y$.

To draw a conceptual distinction between the two products, we compare the Cartesian product $\X \times \X$ to the total product $\X \ast \X$, for a fixed quantum set $\X$. The Cartesian product $\X \times \X $ intuitively consists of \emph{pairs} of elements of $\X$: an element of $\X$ in one hand, and an element of $\X$ in the other. By contrast, the total product $\X \ast \X \iso \X^{`\{1,2\}}$ intuitively consists of \emph{functions} to $\X$, from the quantum set $`\{1,2\}$. Classically, for each ordinary set $S$, we have a canonical bijective correspondence between functions $\{1,2\} \To S$ and ordered pairs from $S$; we evaluate each function on input $1$ to obtain the first entry of the ordered pair, and then we evaluate it on input $2$ to obtain the second entry of the ordered pair. We have no such correspondence in the quantum setting because we cannot duplicate the function to evaluate it twice; there is no duplication morphism $\X^{`\{1,2\}} \To \X^{`\{1,2\}} \times \X^{`\{1,2\}}$.

\begin{proof}[Proof of Proposition \ref{7}]
In any Cartesian closed category, the category-theoretic product preserves colimits in each variable. The category-theoretic product in $\qSet$ is the total product, and it does not preserve colimits: the quantum set $(\mathbf{1} \uplus \mathbf{1}) \ast (\mathbf{1} \uplus \mathbf{1})$ has uncountably many atoms because there are uncountably many inequivalent irreducible joint representations of a pair of von Neumann algebras, each  isomorphic to $\CC^2$, on any two-dimensional Hilbert space, but the quantum set $[\mathbf{1} \ast (\mathbf{1} \uplus \mathbf{1})] \uplus [\mathbf{1} \ast (\mathbf{1} \uplus \mathbf{1})]$ has only four atoms, because the terminal object $\mathbf{1}$ is also the unit for the total product $\ast$.
\end{proof}

\section{Subobjects of a quantum set}\label{section 10}

Let $\X$ be a quantum set. Recall from category theory that the subobjects of $\X$ are defined via the category of all monomorphisms into $\X$, i.e., of all injective functions into $\X$. A morphism in this category from a monomorphism $\Z_1 \rightarrowtail \X$ to a monomorphism $\Z_2 \rightarrowtail \X$ is a function $\Z_1 \To \Z_2$ making the triangle commute. If such a function $\Z_1 \To \Z_2$ exists, it is unique, so the category of monomorphisms into $\X$ is a preorder. The subobjects of $\X$ are the equivalence classes of this preorder.

\begin{proposition}\label{8}
Let $\X$ be a quantum set. The map taking each subset of $\X$ to the equivalence class of its inclusion into $\X$ is an isomorphism of partial orders.
\end{proposition}

\begin{proof}
The statement of Proposition \ref{Z}, together with the proof of $(4) \Yields (1)$, shows that every monomorphism into $\X$ is isomorphic to the inclusion function of some subset $\Y \subsetof \X$ into $\X$. Distinct inclusion functions induce homomorphisms with distinct kernels (Lemma \ref{Y2}), so no two distinct inclusions can be isomorphic. Thus, we have a bijection between the subsets of $\X$ and the subobjects of $\X$ in $\qSet$.

If $\Z \subsetof \Y \subsetof \X$, then we may easily check that $J_\Y^\X \circ J_\Z^\Y = J_\Z^\X$, so our bijection is monotone (Definition \ref{Y1}). Conversely, assume that $\Y$ and $\Z$ are subsets of $\X$, and that there is a function $F\: \Z \To \Y$ such that $J_\Y^\X  \circ F = J_\Z^\X$. Let $Z \atomof \Z \subseteq \X$. By definition of composition, we have that
\begin{equation*}\bigvee_{Y \atomof \Y} J_\Y^\X(Y, Z) \cdot F(Z, Y)  = J_\Z^\X(Z, Z).
\end{equation*}
The space $J_\Z^\X(Z, Z)$ is nonempty, so for some $Y \atomof \Y$, we also have that $J_\Y^\X(Y, Z)$ is nonempty, which implies that $Z = Y \atomof \Y$. Therefore, $\Z \subsetof \Y$. We conclude that our bijection is an isomorphism of partial orders.
\end{proof}

Thus, modulo formalities, the subobjects of a quantum set $\X$ are exactly its subsets. A topos, by definition, must have a subobject classifier $\Omega$, an object admitting a bijection, natural in $\X$, between the subobjects of $\X$, and the morphisms from $\X$ to $\Omega$. Assume for the sake of contradiction that the category $\qSet$ has a subobject classifier $\Omega$, and consider its universal property in the dual category $\Mstar1$: the central projections of any hereditarily atomic von Neumann algebra $A$ must be in bijection with the unital normal $*$-homomorphisms from $\ell^\infty(\Omega)$ to $A$. In particular, there must be exactly two unital normal $*$-homomorphisms from $\ell^\infty(\Omega)$ to $L(\CC^n)$ for any positive integer $n$. It follows immediately that $\Omega$ does not have atoms of dimension larger than $1$, and that it must in fact have exactly two atoms; in other words, we conclude $\Omega \iso \mathbf{1} \uplus \mathbf{1}$. But there are uncountably many functions from $ \X = \Q\{\CC^2\}$ to $\mathbf{1} \uplus \mathbf{1}$, contradicting that $\Q\{\CC^2\}$ has only two subobjects. Thus, we have demonstrated the following:

\begin{proposition}\label{9}
The category $\qSet$ of quantum sets and functions does not have a subobject classifier.
\end{proposition}

However, there is a bijection between the subobjects of a quantum set $\X$ and the ``classical''  functions from $\X$ to $\mathbf 1 \uplus \mathbf 1$. The rest of this section is dedicating to defining and motivating this term.

We begin by examining the Cartesian product monoidal structure on $\qSet$. The unit of the Cartesian product is also the terminal object, so we have projection functions $P_1\: \X_1 \times \X_2 \To \X_1$ and $P_2\: \X_1 \times \X_2 \To \X_2$ defined by $P_1 = U \circ(I_{\X_1} \times\, !\,)$ and $P_2 = U \circ(\,!\,\times I_{\X_2})$, where $!$ denotes the unique map to the terminal object, and $U$ denotes the right or left unitor, as appropriate.

The most significant feature of the Cartesian product monoidal structure we have defined on $\qRel$ is that it coincides with the Cartesian product on ordinary sets. If we think of functions $F_1$ and $F_2$ from a fixed quantum set $\X$ to ordinary sets $`S_1$ and $`S_2$, respectively, as generalized observables, then the compatibility of the two observables should mean the existence of a function from $\X$ to the Cartesian product $`S_1 \times S_2$ that combines the two observables. This leads to the following definition.

\begin{definition}\label{10}
Let $F_1\: \X \To \Y_1$ and $F_2 \: \X \To \Y_2$ be functions. We say that $F_1$ and $F_2$ are \emph{compatible} just in case there is a function $F$ from $\X$ to $\Y_1 \times \Y_2$ such that $P_1 \circ F = F_1$ and $P_2 \circ F = F_2$, where $P_1$ and $P_2$ are the projection functions just defined.
\end{definition}

Viewing this definition in the category of hereditarily atomic von Neumann algebras, $F_1$ and $F_2$ are compatible just in case there is a unital normal $*$-homomorphism $\phi$ from the tensor product $\ell^\infty(\Y_1) \otimes \ell^\infty(\Y_2)$ to $\ell^\infty(\X)$ such that $\phi(b_1 \tensor b_2) = F_1^\star(b_1) \cdot F_2^\star(b_2)$ for all $b_1 \in \ell^\infty(\Y_1)$ and $b_2 \in \ell^\infty(\Y_2)$. In particular, the images of $F_1^\star$ and $F_2^\star$ must commute. This necessary condition is also sufficient, because the categorical tensor product coincides with the spatial tensor product for hereditarily atomic von Neumann algebras \cite{Guichardet66}*{proposition 8.6}. Thus, we have established the following:

\begin{lemma}\label{11}
Let $F_1$ and $F_2$ be functions, as in Definition \ref{10} above. The functions $F_1$ and $F_2$ are compatible if and only if every element in the image of $F_1^\star$ commutes with every element in the image of $F_2^\star$.
\end{lemma}

We use the term ``classical'' for functions that do not consume their arguments:

\begin{definition}\label{12}
A function out of a quantum set $\X$ is \emph{classical} just in case it is compatible with the identity function $I_\X$.  A quantum set $\X$ is said to be \emph{classical} just in case the identity function $I_\X$ is classical.
\end{definition}

\begin{proposition}
A quantum set $\X$ is classical if and only if each atom of $\X$ is one-dimensional.
\end{proposition}

\begin{proof}
By Lemma \ref{11}, the quantum set $\X$ is classical if and only if $\ell^\infty(\X)$ is commutative, or equivalently, if and only if the matrix algebra $L(X)$ is commutative for each $X \atomof \X$. 
\end{proof}

\begin{lemma}\label{12.5}
Let $F$ be a function from a quantum set $\X$ to a quantum set $\Y$. The following are equivalent:
\begin{enumerate}
\item $F$ is classical
\item $F$ is compatible with every function out of $\X$
\item $F$ factors through the canonical function $Q \: \X \twoheadrightarrow `\At(\X)$
$$\begin{tikzcd}
\X  \arrow[bend left=15]{rr}{F} \arrow[two heads]{r}[swap]{Q}& `\At(\X) \arrow[dotted]{r}  & \Y
\end{tikzcd}$$
\end{enumerate}
\end{lemma}

The canonical surjection $Q\: \X \To `\At(\X)$ is intuitively the quotient map that contracts each atom to be one-dimensional; it is defined by $Q(X, \CC_{X}) = L(X, \CC_{X})$ for $X \atomof \X$, with the other components vanishing. 

\begin{proof}
By Lemma \ref{11}, the function $F$ is classical if and only if the image of $F^\star$ is in the center of $\ell^\infty(\X)$. The latter condition is equivalent to $F$ factoring through the inclusion of the center into $\ell^\infty(\X)$, or equivalently, through $Q^\star$. Thus, $(1)$ is equivalent to $(3)$.

If $F$ is compatible with $I_\X$, then it is also compatible with $G \circ I_\X= G$, for any function $G$ out of $\X$, so $(1)$ implies $(2)$. The converse is trivial.
\end{proof}

\begin{proposition}\label{12.8}
Let $\X$ and $\Y$ be quantum sets. For every classical function $F$ from $\X$ to $\Y$, there is a unique ordinary function $f$ from $\At(\X)$ to $\mathrm{Fun}(\mathbf 1; \Y)$ such that $F = J \circ `f \circ Q$, where $Q$ is the canonical surjection from $\X$ to $`\At(\X)$, and $J$ is the canonical injection from $`\mathrm{Fun}(\mathbf 1; \Y)$ to $\Y$.
$$
\begin{tikzcd}
\X \arrow{r}{F} \arrow[two heads]{d}{Q}
&
\Y
\\
`\At(\X) \arrow[dotted]{r}{`f}
&
`\mathrm{Fun}(\mathbf 1; \Y) \arrow[tail]{u}{J}
\end{tikzcd}
$$
\end{proposition}

Up to canonical natural isomorphism, the functor $`\mathrm{Fun}( \mathbf 1 ; \, - \, )$ takes each quantum set $\Y$ to its maximum classical subset, and the injection $J$ is its inclusion function.

\begin{proof}
In light of Lemma \ref{12.5}, to establish the existence of $f$, it is sufficient to show that every function from a classical quantum set $`T$ to the quantum set $\Y$ factors through the inclusion function of the maximum classical subset of $\Y$. Reasoning in the category of hereditarily atomic von Neumann algebras, we simply observe that any unital normal $*$-homomorphism into a commutative von Neumann algebra must vanish on each noncommutative factor, and therefore must factor through the quotient by these noncommutative factors. The uniqueness of $f$ follows by Propositions \ref{Y} and \ref{Z}.
\end{proof}

\begin{proposition}\label{13}
Let $\X$ and $\Z$ be quantum sets.
For every injective function $J\: \Z \rightarrowtail \X$, there is a unique classical function $F\:\X \To \mathbf 1 \uplus \mathbf 1$ such that the following diagram is a pullback square:
$$
\begin{tikzcd}
\Z \arrow{r} \arrow[tail]{d}[swap]{J}
&
\mathbf 1 \arrow[tail]{d}{T}
\\
\X \arrow[dotted]{r}{!}[swap]{F}
&
\mathbf 1 \uplus \mathbf 1
\end{tikzcd}
$$
The function $\Z\To \mathbf 1$ at the top of the diagram is the unique function from $\Z$ to the terminal object $\mathbf 1$. The function $T\: \mathbf 1 \rightarrowtail \mathbf 1 \uplus \mathbf 1$ is the injection taking the singleton to the second summand of $\mathbf 1 \uplus \mathbf 1$. 
\end{proposition}

\begin{proof}
Equivalently, we are to show that for every surjective unital normal $*$-homomorphism $\pi$ from a hereditarily atomic von Neumann algebra $A$ to a hereditarily atomic von Neumann algebra $C$, there is a unique unital normal $*$-homomophism $\phi\: \CC^2 \To A$ with central image that makes the following diagram a pushforward square:
$$
\begin{tikzcd}
C
&
\CC \arrow{l}
\\
A  \arrow[two heads]{u}[swap]{\pi}
&
\CC^2 \arrow[dotted]{l}{\phi}[swap]{!}  \arrow{u}
\end{tikzcd}
$$

A unital normal $*$-homomorphism $\phi\: \CC^2 \To A$ with central image is simply a way of writing the unit of $A$ as a sum of two central projections $p_1$ and $p_2$. A cocone on the diagram
$$
\begin{tikzcd}
&
\CC 
\\
A 
&
\CC^2 \arrow{l}{\phi}  \arrow{u}
\end{tikzcd}
$$
is essentially just a unital normal $*$-homomorphism from $A$ to a hereditarily atomic von Neumann algebra $B$ that takes the second central projection $p_2$ in $A$ to the identity of $B$, so the colimit of this diagram is the quotient map $A \twoheadrightarrow p_2A$. 

Up to isomorphism of codomains, every surjective unital normal $*$-homomorphism $\pi\: A \To C$ is of this form, so there does exist a unital normal $*$-homomorphism $\phi\: \CC^2 \To A$ with central image making the first diagram of the proof into a pushforward square. This $\phi$ is unique, because it is completely determined by the projection $p_2$, and two quotient maps $A \twoheadrightarrow p_2A$ and $A \twoheadrightarrow p_2'A$ are distinct up to isomorphism of codomains whenever $p_2$ and $p_2'$ are distinct central projections.
\end{proof}

We have shown that the category $\qSet$ behaves in many ways like a topos, as discussed in section \ref{section 1.2}.

\section{Operators as functions on a quantum set}\label{section 11}

\begin{definition}
Let $\X$ be a quantum set. An \emph{observable} on $\X$ is a function from $\X$ to $`\RR$.
\end{definition}

We begin this section by demonstrating that the observables on $\X$ are in bijective correspondence with the self-adjoint operators in $\ell(\X)$. We later show that under this correspondence, the sum of two self-adjoint operators arises in the expected way from a function $`\RR \ast `\RR \To `\RR$.

Write $\mathrm{Herm}(A)$ for the vector space of self-adjoint elements of a von Neumann algebra $A$.

\begin{proposition}\label{FH}
The contravariant functors $\mathrm{Fun}(\,-\,; `\RR)$ and $\mathrm{Herm}(\ell(\,-\,))$, from the category of quantum sets and functions to the category of sets and functions, are naturally isomorphic. The natural isomorphism takes each function $F \: \X \To `\RR $ to the self-adjoint operator $F^\star(r)$, where $r$ is the element of $\ell(`\RR)$ defined by $r(\CC_{\alpha}) = \alpha$.
\end{proposition}

\begin{proof}
For each function $F\: \X \To \Y$ we consider the following diagram:
$$
\begin{tikzcd}
\mathrm{Fun}(\Y; `\RR) \arrow{r}{\iso} \arrow{d}{\circ F}
&
\Hom_1(\ell^\infty(`\RR),\ell^\infty(\Y)) \arrow{d}{F^\star \circ} \arrow{r}{\iso}
&
\mathrm{Herm}(\ell(\Y)) \arrow{d}{F^\star}
\\
\mathrm{Fun}(\X; `\RR) \arrow{r}{\iso}
&
\Hom_1(\ell^\infty(`\RR),\ell^\infty(\X)) \arrow{r}{\iso}
&
\mathrm{Herm}(\ell(\X))
\end{tikzcd}
$$
The commutative square on the left comes from the contravariant equivalence of $\qSet$ and $\Mstar1$. The isomorphism on the top right is defined to take each $\psi \in \Hom_1(\ell^\infty(`\RR),\ell^\infty(\Y))$ to the self-adjoint operator $\sum_{\alpha \in \RR} \alpha \cdot \psi(e_\alpha)$ in $\ell(\Y)$, where for each real number $\alpha$, we write $e_\alpha$ for the corresponding minimal projection in $\ell^\infty(`\RR)$. The isomorphism on the bottom right is defined likewise.

The square on the right commutes because for each unital normal $*$-homomorphism $\psi$ from $\ell^\infty(\RR)$ to $\ell^\infty(\Y)$,
\begin{equation*} F^\star \left( \sum_{\alpha \in \RR} \alpha \cdot \psi(e_\alpha) \right)= \sum_{\alpha \in \RR} \alpha \cdot F^\star (\psi(e_\alpha)) = \sum_{\alpha \in \RR} \alpha \cdot (F^\star \circ \psi) (e_\alpha).
\end{equation*}
The interchange of $F^\star$ with summation is justified because $F^\star$ is continuous as a function from $\ell(\Y)$ to $\ell(\X)$. Therefore, we have a natural isomorphism between the functors $\mathrm{Fun}(\,-\,; `\RR)$ and $\mathrm{Herm}(\ell(\,-\,))$.

Setting $\Y = `\RR$, and tracking the identity function $I \in \mathrm{Fun}(\Y; `\RR)$ through the diagram, we find that our natural isomorphism takes $F= I \circ F \in \mathrm{Fun}(\X; `\RR)$ to $F^\star(\sum_{\alpha \in \RR} \alpha e_\alpha) = F^\star(r)$.
\end{proof}

\begin{definition}
Write $Q_1$ and $Q_2$ for the two projection functions $`\RR \ast `\RR \To `\RR$. Define $\bm{+}\: `\RR \ast `\RR \To `\RR$ to be the unique function such that $\bm + ^\star (r) = Q_1^\star(r) + Q_2^\star(r)$, where $r$ is the element of $\ell(`\RR)$ defined by $r(\CC_{\alpha}) = \alpha$.
\end{definition}

\begin{definition}
Let $F_1$ and $F_2$ be observables on a quantum set $\X$. Their sum $F_1 + F_2$ is defined to be the observable $\bm + \circ \< F_1, F_2\>$, where $\< F_1, F_2\>$ is the unique function from $\X$ to $` \RR \ast `\RR$ defined by $Q_1 \circ \< F_1 , F_2 \> = F_1$ and $Q_2 \circ \< F_1, F_2 \> = F_2$.
\end{definition}

\begin{proposition}
Under the natural isomorphism of Proposition \ref{FH}, the sum of observables on $\X$ corresponds to the sum of self-adjoint operators in $\ell(\X)$.
\end{proposition}

\begin{proof}
This is immediate from the definition:
\begin{align*}(F_1 + F_2)^\star(r) & = \< F_1, F_2 \>^\star(\bm + ^\star (r)) = \< F_1, F_2 \>^\star (Q_1^\star (r) + Q_2^\star(r))  \\ & = \< F_1, F_2 \>^\star (Q_1^\star(r)) + \< F_1, F_2 \>^\star (Q_2^\star(r)) \\  &= F_1^\star(r) + F_2^\star(r) 
\end{align*} 
\end{proof}

Thus, we may think of the observables on a quantum set $\X$ as being functions into $`\RR$, and we may think of their additive structure as coming from the additive structure on $`\RR$, just as we do classically. The structure $(`\RR, +, `0)$ is like a quantum group, but its group product is defined on the total product $`\RR \ast `\RR$, as opposed to the Cartesian product $`\RR \times `\RR$. The group product of a quantum group is in effect defined on just the Cartesian product of that quantum group with itself.

The product of two self-adjoint operators is generally not self-adjoint, but it is possible to extend the ordinary product operation of the real numbers from the Cartesian product $`\RR \times `\RR$ to the total product $`\RR \ast `\RR$, by working with the Jordan product. However, we will instead turn to the $*$-algebra $\ell(\X)$, and work with the product there.

\begin{lemma}\label{lc}
The contravariant functors $\mathrm{Fun}(\,-\,; `\RR \ast `\RR)$ and $\ell(\,-\,)$, from the category of quantum sets and functions to the category of sets and functions, are naturally isomorphic. The natural isomorphism takes each function $F\: \X \To `\RR \ast `\RR$ to the self-adjoint operator $F^\star(s)$, where $s$ is the element of $\ell(`\RR \ast `\RR)$ defined by $s= Q_1^\star(r) + i Q_2^\star(r)$, and $r$ is the element of $\ell(`\RR)$ defined by $r(\CC_{\alpha}) = \alpha$.
\end{lemma}

\begin{proof}
We compose natural isomorphisms:
\begin{equation*}\mathrm{Fun}(\X; `\RR \ast `\RR) \iso \mathrm{Fun}(\X; `\RR) \times \mathrm{Fun}(\X; `\RR) \iso \mathrm{Herm}(\ell(\X)) \times \mathrm{Herm}(\ell(\X)) \iso \ell(\X)
\end{equation*}
The first isomorphism is the universal property of the total product; the second isomorphism is from Proposition \ref{FH}; and the third isomorphism is the bijection taking each pair of self-adjoint operators $(a_1, a_2)$ to the operator $a_1 + i a_2$. We track a function $F$ from a quantum set $\X$ to $`\RR \ast `\RR$ through this chain of bijections:
\begin{equation*}
F \mapsto (Q_1 \circ F, Q_2 \circ F) \mapsto ( (Q_1 \circ F)^\star(r), (Q_2 \circ F)^\star(r))   \mapsto (Q_1 \circ F)^\star(r) + i (Q_2 \circ F)^\star(r)
\end{equation*}
\begin{equation*} 
(Q_1 \circ F)^\star(r) + i (Q_2 \circ F)^\star(r) = F^\star(Q_1^\star(r)) + i F^\star(Q_2^\star(r)) = F^\star(Q_1^\star(r) + i Q_2^\star(r)) = F^\star(s) 
\end{equation*}\end{proof}

\begin{definition}\label{quantum c}
Define $\C = `\RR \ast `\RR$. Define $s \in \ell(\C)$ as in Lemma \ref{lc}. Define
\begin{enumerate}
\item a function $\bm + \: \C \ast \C \To \C$ by $\bm + ^\star (s) = Q_1^\star(s) + Q_2^\star(s)$,
\item a function $\bm \cdot \: \C \ast \C \to \C$ by $\bm \cdot ^\star(s) = Q_1^\star(s)Q_2^\star(s)$, and
\item a function $\overline{\phantom{s}}\: \C \To \C$ by $\overline{\phantom{s}}^\star(s) = s^\dagger$.
\end{enumerate}
Define the inclusion function $`\CC \hookrightarrow \C$ to be the canonical injective function $\< P_1, P_2 \>\: `\RR \times `\RR \rightarrowtail `\RR \ast `\RR$.
\end{definition}

\begin{definition}\label{sa}
Let $\X$ be a quantum set, and let $F_1$ and $F_2$ be functions from $\X$ to $\C$. Define
\begin{enumerate}
\item $F_1 + F_2 = \bm + \circ \< F_1, F_2 \>$,
\item $F_1 \cdot F_2 =  \bm \cdot \circ \< F_1, F_2 \>$, and
\item $\overline{F_1} = \overline{\phantom{s}} \circ F_1$.
\end{enumerate}
For each complex number $\alpha$, define $\mathrm{Cst}_\alpha$ to be the function from $\X$ to $\C$ given by the following composition:
$$
\begin{tikzcd}
\X \arrow{r}{!}
&
`\{\ast\} \arrow{r}{`\alpha}
&
`\CC \arrow[hook]{r}
& 
\C
\end{tikzcd}
$$

\end{definition}

\begin{theorem}\label{opalgebra}
The contravariant functors $\mathrm{Fun}(\,-\,; \C)$ and $\ell(\,-\,)$, from the category of quantum sets and functions, to the category of unital $*$-algebras over $\CC$ and unital $*$-homo\-morphisms, are naturally isomorphic. The natural isomorphism takes each function $F\: \X \To \C$ to the operator $F^\star(s)$.
\end{theorem}

\begin{proof}
It remains only to show that the natural isomorphism of Lemma \ref{lc} respects the structure of Definition \ref{sa}.

\begin{align*}(F_1 + F_2)^\star(s) & = \< F_1, F_2 \>^\star(\bm + ^\star (s)) = \< F_1, F_2 \>^\star (Q_1^\star (s) + Q_2^\star(s))  \\ & = \< F_1, F_2 \>^\star (Q_1^\star(s)) + \< F_1, F_2 \>^\star (Q_2^\star(s)) \\  &= F_1^\star(s) + F_2^\star(s)
\end{align*}

\begin{align*}(F_1 \cdot F_2)^\star(s) & = \< F_1, F_2 \>^\star(\bm \cdot ^\star (s)) = \< F_1, F_2 \>^\star (Q_1^\star (s) \cdot Q_2^\star(s))  \\ & = \< F_1, F_2 \>^\star (Q_1^\star(s)) \cdot \< F_1, F_2 \>^\star (Q_2^\star(s)) \\  &= F_1^\star(s) \cdot F_2^\star(s)
\end{align*}

\begin{align*}
\overline{F_1}^\star(s) = F_1^\star( \overline{\phantom{s}}^\star(s)) = F_1^\star(s^\dagger) = F_1^\star (s)^\dagger
\end{align*}

\begin{align*}
\mathrm{Cst}_\alpha^\star(s) & =  \,!^\star(`\alpha^\star(\< P_1,  P_2\>^\star(Q_1^\star(r) + iQ_2^\star(r)))  \\ &= \, !^\star(`\alpha^\star(P_1^\star(r) + iP_2^\star(r))
\\ & = \,!^\star(`\alpha^\star(P_1^\star(r)) + i `\alpha^\star(P_2^\star(r)))
\\ & = \,!^\star(\mathrm{Re}(\alpha) + i \mathrm{Im}(\alpha))
\\ & = \,!^\star(\alpha) = \alpha\cdot 1
\end{align*}
\end{proof}

\appendix

\renewcommand{\theequation}{\thesection\arabic{equation}}

\section{Displaced lemmas}

\setcounter{equation}{0}

\begin{lemma}\label{Q3.1}
Let $\X$ and $\Y$ be quantum sets. For each $X \atomof \X$, write $\pi_X\: \ell(\X) \To L(X)$ for the projection map, and for each $Y \atomof \Y$, write $\iota^Y\: L(Y) \To \ell(\Y)$ for the inclusion map.

For each continuous $*$-homomorphism $\phi\: \ell(\Y) \To \ell(\X)$, the family of $*$-homomorphisms
\begin{equation*}
(\phi_X^Y = \pi_X \circ \phi \circ \iota^Y \suchthat X \atomof \X, \,Y \atomof \Y)
\end{equation*}
satisfies $\phi_X^{Y_1}(L(Y_1)) \cdot \phi_X^{Y_2}(L(Y_2)) = 0$ for all distinct $Y_1, Y_2 \atomof \Y$. Conversely, for each family of $*$-homomorphisms $(\phi_X^Y\:L(Y) \To L(X) \suchthat X \atomof \X, \, Y \atomof \Y)$ satisfying $\phi_X^{Y_1}(L(Y_1)) \cdot \phi_X^{Y_2}(L(Y_2)) = 0$ for all distinct $Y_1, Y_2 \atomof \Y$, the equation
\begin{equation*}\phi(b)(X) = \sum_{Y \atomof \Y} \phi_X^Y(b(Y))
\end{equation*}defines a continuous $*$-homomorphism $\phi\:\ell(\Y) \To \ell(\X)$. The two constructions are inverses.

Furthermore, the same proposition is true for normal $*$-homomorphisms $\ell^\infty(\Y) \To \ell^\infty(\X)$, in place of continuous $*$-homomorphisms $\ell(\Y) \To \ell(\X)$. As a consequence, each continuous $*$-homomorphism $\ell(\Y) \To \ell(\X)$ restricts to a normal $*$-homomorphism $\ell^\infty(\Y) \To \ell^\infty(\X)$, and this defines a bijection between the two sets of morphisms.
\end{lemma}

As elsewhere, we use the product topologies on $\ell(\X)$ and $\ell(\Y)$.

\begin{proof}
Let $\phi\: \ell(\Y) \To \ell(\X)$ be any continuous $*$-homomorphism, and let $Y_1, Y_2 \atomof \Y$ be distinct. By definition of the canonical inclusion maps $\iota^{Y_1}$ and $\iota^{Y_2}$, we certainly have that $\iota^{Y_1}(L(Y_1)) \cdot \iota^{Y_2}(L(Y_2)) = 0$. Therefore, the family $(\phi_X^Y = \pi_X \circ \phi \circ \iota^Y \suchthat X \atomof \X, \,Y \atomof \Y)$ satisfies $\phi_X^{Y_1}(L(Y_1)) \cdot \phi_X^{Y_2}(L(Y_2)) = 0$, for all distinct $Y_1, Y_2 \atomof \Y$. Furthermore, for all $b \in \ell(\Y)$ and all $X \atomof \X$,
\begin{equation*}
\sum_{Y \atomof \Y} \phi_X^Y(b(Y)) = \sum_{Y \atomof \Y} \pi_X(\phi(\iota^Y(b(Y)))) = \pi_X\left( \phi \left(\sum_{Y \atomof \Y} \iota^Y(b(Y)) \right) \right) = \pi_X(\phi(b))  =\phi(b)(X).
\end{equation*}
We have used the assumed continuity of $\phi$, and the tautological continuity of $\pi_X$.

Now, let $(\phi_X^Y\:L(Y) \To L(X) \suchthat X \atomof \X, \, Y \atomof \Y)$ be any family of $*$-homomorphisms satisfying $\phi_X^{Y_1}(L(Y_1)) \cdot \phi_X^{Y_2}(L(Y_2)) = 0$, for all distinct $Y_1, Y_2 \atomof \Y$, and fix $X \atomof \X$. Since $X$ is finite-dimensional, $\phi_X^Y(L(Y)) = 0$ for all but finitely many $Y \atomof \Y$. It follows that for all $b \in \ell(\Y)$, the sum is $\sum_{Y \atomof \Y} \phi_X^Y(b(Y))$ is trivially convergent, and furthermore, this expression defines a continuous $*$-homomorphism $\phi_X\: \ell(\Y) \To L(X)$. Therefore, the $*$-homomorphism $\phi\: \ell(\Y) \To \ell(\X)$ defined by $\phi(b)(X) = \phi_X(b)$ is also continuous, by definition of the product topology. Furthermore, for all $X \atomof \X$, $Y \atomof \Y$, and $m \in L(Y)$,
\begin{equation*}
\pi_X(\phi(\iota^Y(m)) = \phi(\iota^Y(m))(X) = \phi_X(\iota^Y(m)) = \sum_{Y' \atomof \Y} \phi_X^{Y'}(\iota^Y(m)(Y')) = \phi_X^Y(m).
\end{equation*}
Therefore, the two constructions are inverses.

The argument works just the same with normal $*$-homomorphisms $\ell^\infty(\Y) \To \ell^\infty(\X)$, in place of continuous $*$-homomorphisms $\ell(\Y) \To \ell(\X)$, with a minor adjustment in the justifications. In the forward direction, we observe that the sum $\sum_{Y \atomof \Y} \iota^Y(b(Y))$ converges to $b$ ultraweakly, and that $\phi$ is ultraweakly continuous because it is normal. In the backward direction, we observe that $\ell^\infty(\Y)$ is the product of the family $(L(Y)\suchthat Y \atomof \Y)$ in the category of von Neumann algebras and normal $*$-homomorphisms.
\end{proof}

\begin{definition}\label{Q6}
For $i \in \{1,2\}$, let $\X_i$ and $\Y_i$ be quantum sets, and let $f_i\: \X_i \To \Y_i$ be a partial fission, with components $(f_i)_{X_i}^{Y_i} \: X_i \To Y_i \tensor (H_i)_{X_i}^{Y_i}$, for all $X_i \atomof \X_i$ and $Y_i \atomof \Y_i$. Up to the obvious permutation of tensor factors, the \emph{tensor product} partial fission $f_1 \tensor f_2$ is defined simply by forming the tensor product in each component. Explicitly,
\begin{equation*}(f_1 \tensor f_2)_{X_1 \tensor X_2}^{Y_1 \tensor Y_2} = (1 \tensor \sigma \tensor 1) \cdot ((f_1)_{X_1}^{Y_1} \tensor (f_2)_{X_2}^{Y_2})\: X_1 \tensor X_2 \To Y_1 \tensor Y_2 \tensor (H_1)_{X_1}^{Y_1} \tensor (H_2)_{X_2}^{Y_2},
\end{equation*}
where $\sigma$ is the braiding of the tensor product.
\end{definition}

\begin{lemma}\label{Q7}
For each $i \in \{1,2\}$, let $\X_i$ and $\Y_i$ be quantum sets, and let $F_i$ be a partial function from $\X_i$ to $\Y_i$, with corresponding partial fission $f_i$, and corresponding homomorphism $\phi_i$. It follows that the partial function $F_1 \times F_2$ has corresponding partial fission $f_1 \tensor f_2$, and corresponding homomorphism $\phi_1 \tensor \phi_2$.
\end{lemma}

In this context, the continuous homomorphism $\phi_1 \tensor \phi_2\: \ell(\Y_1\times \Y_2) \To \ell(\X_1 \times \X_2)$ is defined by $(\phi_1 \tensor \phi_2)_{X_1 \tensor X_2}^{Y_1 \tensor Y_2} = (\phi_1)_{X_1}^{Y_1} \tensor (\phi_2)_{X_2}^{Y_2} \:  L(Y_1 \tensor Y_2) \To L(X_1 \tensor X_2)$ (Lemma \ref{Q3.1}). For all $b_1 \in \ell(\Y_1)$ and $b_2 \in \ell(\Y_2)$, writing $b_1 \tensor b_2$ for the element of $\ell(\Y_1 \times \Y_2)$ defined by $(b_1 \tensor b_2)(Y_1 \tensor Y_2) = b_1(Y_1) \tensor b_2(Y_2)$, we find that
\begin{align*}
(\phi_1 \tensor \phi_2)(b_1 \tensor b_2)(X_1 \tensor X_2)
&=
\sum_{Y_1 \tensor Y_2 \atomof \Y_1 \times \Y_2} (\phi_1 \tensor \phi_2)_{X_1 \tensor X_2}^{Y_1 \tensor Y_2}((b_1 \tensor b_2)(Y_1 \tensor Y_2))
\\ &= 
\sum_{Y_1 \atomof \Y_1} \sum_{Y_2 \atomof \Y_2} (\phi_1)_{X_1}^{Y_1}(b_1(Y_1)) \tensor (\phi_2)_{X_2}^{Y_2}(b_2(Y_2)) 
 \\ &=
\phi_1(b_1)(X_1) \tensor \phi_2(b_2)(X_2).
\end{align*}

\begin{proof}[Proof of Lemma \ref{Q7}]
For all $b_1 \in \ell(\Y_1)$ and $b_2 \in \ell(\Y_2)$, and for all $X_1 \atomof \X_1$ and $X_2 \atomof \X_2$,
\begin{align*}
(\phi_1 & \tensor \phi_2) (b_1 \tensor b_2)  (X_1 \tensor X_2)
 =
\phi_1(b_1)(X_1) \tensor \phi_2(b_2)(X_2)
\\ &=
\left(\sum_{Y_1 \atomof \Y_1} (f_1)_{X_1}^{Y_1\, \dagger} (b_1(Y_1) \tensor 1) (f_1)_{X_1}^{Y_1}\right)
\tensor
\left(\sum_{Y_2 \atomof \Y_2} (f_2)_{X_2}^{Y_2\, \dagger} (b_2(Y_2) \tensor 1) (f_2)_{X_2}^{Y_2}\right)
\\ &=
\sum_{Y_1 \atomof \Y_1} \sum_{Y_2 \atomof \Y_2}  ((f_1)_{X_1}^{Y_1} \tensor (f_2)_{X_2}^{Y_2})^\dagger(b_1(Y_1) \tensor 1 \tensor b_2(Y_2) \tensor 1)    ((f_1)_{X_1}^{Y_1} \tensor (f_2)_{X_2}^{Y_2})
\\ &=
\sum_{Y_1 \atomof \Y_1} \sum_{Y_2 \atomof \Y_2}  ((f_1)_{X_1}^{Y_1} \tensor (f_2)_{X_2}^{Y_2})^\dagger (1 \tensor \sigma \tensor 1) ^\dagger \cdot  ((b_1 \tensor b_2)(Y_1 \tensor Y_2) \tensor 1 \tensor 1) \setlength{\jot}{-1ex} \raisebox{-4ex}{\!\!\!\!\!\!\!\!\!\!\!\!\!\!\!\!\!\!\!\!\!\!\!\!\!\!\!\!\!\!\!\!\!\!\!\!\!\!\!\!\!\!\!\!\!\!\!\!\!\!\!\!\!\!$\cdot\, (1 \tensor \sigma \tensor 1) ((f_1)_{X_1}^{Y_1} \tensor (f_2)_{X_2}^{Y_2}).$}
\end{align*}
The span of $\{b_1 \tensor b_2 \suchthat b_1 \in \ell(\Y_1),\, b_2 \in \ell(\Y_2)\}$ is clearly dense in $\ell(\Y_1 \times \Y_2)$, since this algebra is the closure of the span of its factors. Thus, $\phi_1 \tensor \phi_2$ is equal to the homomorphism corresponding to $f_1 \tensor f_2$.

Fix $X_1 \atomof \X_1$, $X_2 \atomof \X_2$, $Y_1 \atomof \Y_1$, and $Y_2 \atomof \Y_2$. Choose orthonormal bases $(B_1)_{X_1}^{Y_1}$ for $(B_2)_{X_2}^{Y_2}$ for $F_1(X_1, Y_1)$ and $F_2(X_2, Y_2)$ respectively. The set $\{v_1 \tensor v_2 \suchthat v_1 \in (B_1)_{X_1}^{Y_1},\, v_2 \in (B_2)_{X_2}^{Y_2} \}$ is an orthonormal basis for $F_1(X_1, Y_1) \tensor F_2(X_2, Y_2)$. For all vectors $x_1 \tensor x_2$ in an arbitrary atom $X_1 \tensor X_2 \atomof \X_1 \times \X_2$,
\begin{align*}
(f_1 \tensor f_2)_{X_1 \tensor X_2}^{Y_1 \tensor Y_2}(x_1 \tensor x_2)
& =
(1 \tensor \sigma \tensor 1) ((f_1)_{X_1}^{Y_1}(x_1) \tensor (f_2)_{X_2}^{Y_2}(x_2))
\\ &=
(1 \tensor \sigma \tensor 1) \sum_{v_1 \in (B_1)_{X_1}^{Y_1}} \sum_{v_2 \in (B_2)_{X_2}^{Y_2}} v_1(x_1) \tensor v_1^\dagger \tensor v_2(x_2) \tensor v_2^\dagger
\\ &=
\sum_{v_1 \in (B_1)_{X_1}^{Y_1}} \sum_{v_2 \in (B_2)_{X_2}^{Y_2}} v_1(x_1) \tensor  v_2(x_2) \tensor v_1^\dagger \tensor v_2^\dagger
\\ &=
\sum_{v_1 \in (B_1)_{X_1}^{Y_1}} \sum_{v_2 \in (B_2)_{X_2}^{Y_2}} (v_1 \tensor v_2)(x_1 \tensor x_2) \tensor (v_1 \tensor v_2)^\dagger.
\end{align*}
Vectors of the form $x_1 \tensor x_2$ span $X_1 \tensor X_2$, so we can vary $X_1$, $X_2$, $Y_1$, and $Y_2$ to conclude that $f_1 \tensor f_2$ is the partial fission corresponding to $F_1 \tensor F_2$.
\end{proof}

\section{Predicates on quantum sets}\label{unary}\label{section 12}

\setcounter{equation}{0}

An ordinary predicate $p$ on an ordinary set $S$ is essentially just a subset of $S$, but its superset $S$ is part of its data. We generalize this notion to the quantum setting.

\begin{definition}\label{qur}
A \emph{predicate} on a quantum set $\X$ is a function that assigns to each atom $X \atomof \X$ a subspace $P(X) \leq X$. 
\end{definition}

If $p$ is an ordinary predicate on an ordinary set $S$, we define the predicate $`p$ on the quantum set $`S$ in the expected way: $`p(\CC_{s})$ is equal to $\CC_{s}$ if $s \in p$, and it vanishes otherwise.

We write $\mathrm{Pred}(\X)$ for the set of predicates on a quantum set $\X$. It is a complete orthomodular lattice, with the partial order and the orthocomplemention defined atomwise (cf. Definition \ref{I}). When $P_1 \leq \neg P_2$, we say that that $P_1$ and $P_2$ are \emph{disjoint}. In this appendix, we extend $\mathrm{Pred}$ to a contravariant functor from the category $\qSet$ of quantum sets and functions to the category of ortholattices and monotone functions. We go on to exhibit a few naturally isomorphic functors.

\begin{definition}\label{image}
Let $R$ be a binary relation from a quantum set $\X$ to a quantum set $\Y$. Define the ordinary function $R_\star$ from $\mathrm{Pred}(\X)$ to $\mathrm{Pred}(\Y)$ by
\begin{equation*}R_\star(P)(Y) = \mathrm{span}\{rx\suchthat X \atomof \X, \, x \in P(X),\, r \in R(X,Y)\}.
\end{equation*}
\end{definition}

The function $R_\star$ is clearly monotone. If $R$ and $S$ are composable binary relations, then $(S \circ R)_\star = S_\star \circ R_\star$, for essentially the same reason that the composition of binary relations is associative. Thus, we have a covariant functor from the category of quantum sets and binary relations to the category of partially ordered sets and monotone functions. We call this functor the \emph{direct image functor}.

We define the \emph{inverse image functor} by composing the direct image functor with the adjoint functor, writing $R^\star = (R^\dagger)_\star$. This is a contravariant functor from the category of quantum sets and binary relations to the category of partially ordered sets and monotone functions. The functor $\mathrm{Pred}$ mentioned at the beginning of this section is just the restriction of this inverse image functor to $\qSet$. Overloading notation, we will also write $\mathrm{Pred}$ for the inverse image functor itself.

\begin{lemma}\label{PredRel}
The inverse image functor $\mathrm{Pred}(\,-\,)$ from the category $\qRel$ of quantum sets and binary relations to the category of ortholattices and monotone functions is naturally isomorphic to the functor $\mathrm{Rel}(\,-\,; \mathbf 1)$. The natural isomorphism takes each predicate $P$ on a quantum set $\X$ to the binary relation $R$ defined by $R(X, \CC) = \{ \< x | \, \cdot \, \> \suchthat x \in P(X)\}$.
\end{lemma}

\begin{proof}
On the level of Hilbert spaces, this natural isomorphism is just the canonical antiunitary from a Hilbert space $X$ to its dual $X^\ast = L( X, \CC)$; it is natural because $\< a^\dagger y | \, \cdot \,\> = \< y | a \cdot \>$ for all $y \in Y$, for each $a \in L(X, Y)$. The reasoning lifts to binary relations, as usual.
\end{proof}

Any binary relation $R$ from a quantum set $\X$ to $\mathbf 1$ is a partial function, since the condition $R \circ R^\dagger \leq {I_\mathbf 1}$ is satisfied automatically. Thus, we can equivalently say that $\mathrm{Pred}(\,-\,)$ is naturally isomorphic to $\mathrm{Par}(\,-\,; \mathbf 1)$.

Under the contravariant equivalence between $\qPar$ and $\Mstar0$, partial functions from $\X$ to $\mathbf 1$ correspond to normal $*$-homomorphisms from $\CC$ to $\ell^\infty(\X)$. A normal $*$-homomorphism $\phi$ from $\CC$ to any von Neumann algebra $A$ is uniquely determined by the projection $\phi(1)$, so we have a bijective correspondence between the predicates on $\X$ and the projections in $\ell^\infty(\X)$. To state this correspondence as a natural isomorphism of functors, we write $\mathrm{Proj}(\,-\,)$ for the functor that takes each von Neumann algebra $A$ to its ortholattice of projections and that simply restricts each normal $*$-homomorphism out of $A$ to this set.

\begin{lemma}\label{RelProj}
The functors $\mathrm{Par}(\,-\,; \mathbf 1)$ and $\mathrm{Proj}(\ell^\infty(\,-\,))$ are naturally isomorphic as contravariant functors from the category of quantum sets and partial functions to the category of ortholattices and monotone functions. For any quantum set $\X$, the natural isomorphism takes each partial function $R$ in $\mathrm{Par}(\X; \mathbf 1)$ to $R^\star(1)$, where $R^\star$ is the normal $*$-homomorphism corresponding to $R$.
\end{lemma}

\begin{proof}
It is immediate from Proposition \ref{Q12} that the assignment $R \mapsto R^\star(1)$ is monotone. Its inverse is also monotone, as we can see from the expression for $R$ in Theorem \ref{Q3} as the intertwining space of $R^\star$:
\begin{align*}R(X, \CC)   & = \{ v \in L(X, \CC) \suchthat \alpha v = v R^\star(\alpha)(X)\text{ for all $\alpha \in \CC$}\} \\ & =  \{ v \in L(X, \CC) \suchthat v = v R^\star(1)(X)\}
\end{align*}
Finally, the functoriality of $(\,-\,)^\star$ implies immediately that we have a natural transformation: for any partial function $F$ from a quantum set $X$ to a quantum set $\Y$, and any partial function $R$ from $\Y$ to $\mathbf 1$, we have $F^\star(R^\star(1)) = (R \circ F)^\star(1)$.
\end{proof}

As a functor to the category $\mathbf{Set}$ of sets and functions, $\mathrm{Par}(\,-\,;\mathbf 1)$ is naturally isomorphic to $\mathrm{Fun}(\,-\,; \mathbf 1 \uplus \mathbf 1)$; this is apparent in $\Mstar0$. In general, the morphism set $\mathrm{Fun}(\X; \Y)$, for arbitrary quantum sets $\X$ and $\Y$, has no canonical order structure, so the components of this natural isomorphism are bijections, rather than order isomorphisms. However, these bijections do become order isomorphisms if each morphism set $\mathrm{Fun}(\X; \mathbf 1 \uplus \mathbf 1)$ is given a partial order structure from $\mathbf 1 \uplus \mathbf 1$, in the manner of Definition \ref{sa}.

\begin{definition}
Write $\BB = \{0,1\}$. Define $t$ to be the element of $\ell(`\BB)$ satisfying $t(\CC_{\alpha}) = \alpha$ for $\alpha \in \BB$. Define
\begin{enumerate}
\item a function $\vee\: `\BB \ast `\BB \To `\BB$ by $\vee^\star(t) = Q_1^\star(t) \vee Q_2^\star(t)$,
\item a function $\wedge\: `\BB \ast `\BB \To `\BB$ by $\wedge^\star(t) = Q_1^\star(t) \wedge Q_2^\star(t)$, and
\item a function $\comp\: `\BB \To `\BB$ by $\comp^\star(t) = 1 - t$.
\end{enumerate}
\end{definition}

\begin{definition}
Let $\X$ be a quantum set, and let $F_1$ and $F_2$ be functions from $\X$ to $`\BB$. Define
\begin{enumerate}
\item $F_1 \vee F_2 = \vee \circ\< F_1 , F_2 \>$,
\item $F_1 \wedge F_2 = \wedge \circ \< F_1 , F_2 \>$,
 and
\item $\comp F_1 = \comp \circ F_1$.
\end{enumerate}
Define $F_1 \leq F_2$ just in case $F_1 \vee F_2 = F_2$.
\end{definition}

\begin{lemma}\label{FunProj}
The functors $\mathrm{Fun}(\,-\,; `\BB)$ and $\mathrm{Proj}(\ell^\infty(\,-\,))$ are naturally isomorphic as contravariant functors from the category of quantum sets and functions to the category of partially ordered sets and monotone functions. For any quantum set $\X$, the natural isomorphism takes each function $F$ in $\mathrm{Fun}(\X;` \BB)$ to $F^\star(t)$.
\end{lemma}

\begin{proof}
For any quantum set $\X$, the assignment $F \mapsto F^\star(t)$ is a bijection from $\mathrm{Fun}(\X; `\BB)$ to $\mathrm{Proj}(\ell^\infty(\X))$ thanks to the contravariant equivalence of $\qSet$ and $\Mstar0$ described in Theorem \ref{Q10}. This assignment is immediately seen to be natural in $\X$, just as in the proof of Proposition \ref{RelProj}. To show that this assignment is an order isomorphism, we demonstrate that it preserves $\vee$ and $\wedge$, just as in the proof of Theorem \ref{opalgebra}:

\begin{align*}(F_1 \vee F_2)^\star(t)  &= \< F_1, F_2 \>^\star ( \vee^\star(t)) = \< F_1, F_2 \>^\star ( Q_1^\star(t) \vee Q_2^\star(t)) \\ & = 
\< F_1, F_2 \>^\star ( Q_1^\star(t)) \vee \< F_1, F_2 \>^\star ( Q_2^\star(t))
\\ & = F_1^\star(t) \vee F_2^\star(t)
\end{align*}

\begin{align*}(F_1 \wedge F_2)^\star(t)  &= \< F_1, F_2 \>^\star ( \wedge^\star(t)) = \< F_1, F_2 \>^\star ( Q_1^\star(t) \wedge Q_2^\star(t)) \\ & = 
\< F_1, F_2 \>^\star ( Q_1^\star(t)) \wedge \< F_1, F_2 \>^\star ( Q_2^\star(t))
\\ & = F_1^\star(t) \wedge F_2^\star(t)
\end{align*}

Any unital normal $*$-homomorphism preserves the meets and joins of projections, because in any von Neumann algebra the meet of two projections $p$ and $q$ is equal to the ultraweak limit of the sequence $((pq)^n \suchthat n \in \NN)$.
\end{proof}

In fact, for each quantum set $\X$, the partial order $\mathrm{Fun}(\X ; ` \BB)$ is a complete orthomodular lattice with orthocomplementation $F \mapsto \comp F$, because it is order isomorphic to the complete orthomodular lattice $\mathrm{Proj}(\ell^\infty(\X))$, and the order isomorphism takes $\comp$ to the orthocomplementation:
\begin{equation*}(\comp F)^\star(t) = F^\star(\comp^\star(t)) = F^\star(1-t) = 1 - F^\star(t)
\end{equation*}

\begin{theorem}\label{4functors}
The functors $\mathrm{Pred}(\,-\,)$, $\mathrm{Rel}(\,-\,; \mathbf 1)$, $\mathrm{Proj}(\ell^\infty(\,-\,))$, and $\mathrm{Fun}(\, -\,; `\BB)$ are naturally isomorphic as contravariant functors from the category of quantum sets and functions to the category of ortholattices and ortholattice morphisms.
\end{theorem}

For us, an \emph{ortholattice morphism} from one ortholattice to another is a monotone function that preserves meets, joins, and orthocomplements, as well as the top and bottom elements.

\begin{proof}
By Lemmas \ref{PredRel}, \ref{RelProj}, and \ref{FunProj}, the four given functors are naturally isomorphic as functors into the category of ortholattices and monotone functions. The components of the three given natural isomorphisms are easily seen to be ortholattice isomorphisms: They preserve meets and joins, as well as the top and bottom elements simply by virtue of being order isomorphisms. Each component of the natural isomorphism in Lemma \ref{PredRel} preserves orthocomplementation because the map $x \mapsto \< x | \, \cdot\, \>$ is antiunitary. Each component of the natural isomorphism in Lemma \ref{RelProj} preserves orthocomplementation because the subspaces $\{ v \in L(X, \CC) \suchthat v = v p\}$ and $\{ v \in L(X, \CC) \suchthat v = v (1-p)\}$ are orthocomplements for every Hilbert space $X$, and every projection $p$ in $L(X)$. We have already observed that each component of the natural isomorphism in Lemma \ref{FunProj} preserves orthocomplementation.

Let $R$ be a function from a quantum set $\X$ to a quantum set $\Y$, and let $F$ be any element of $\mathrm{Fun}(\Y; `\BB)$. Immediately, $(\comp F) \circ R = \comp \circ F \circ R = \comp (F \circ R)$. Therefore, $\mathrm{Fun}(\, - \,; `\BB)$ is a functor into the category of ortholattices and ortholattice morphisms. Bootstrapping along our three natural isomorphisms, we conclude that the other three functors are also into the category of ortholattices and ortholattice morphisms.
\end{proof}

\begin{figure}
$$ 
\begin{tikzcd}
\mathrm{Pred}(\X) \arrow{r}{\iso}[swap]{\ref{PredRel}}
&
\mathrm{Rel}(\X; \mathbf 1) \arrow{dl}{\ref{RelProj}}[swap]{\iso}
\\
\mathrm{Proj}(\ell^\infty(\X))
&
\mathrm{Fun}(\X; `\BB) \arrow{l}{\ref{FunProj}}[swap]{\iso}
\end{tikzcd}
$$

$$
\begin{tikzcd}
P \arrow[mapsto]{r} \arrow[mapsto]{rd} \arrow[mapsto]{d}
&
R(X, \CC) = P(X)^{0 \perp}
\\
p(X) = \mathrm{proj}_{P(X)}
&
F(X, \CC_{0}) = P(X)^0 
\end{tikzcd}
\hspace{0.7in}
\begin{tikzcd}
P(X) = R(X, \CC)^{0\perp}
&
R \arrow[mapsto]{l} \arrow[mapsto]{ld} \arrow[mapsto]{d}
\\
p = R^\star(1)
&
{[R^\dagger, \neg R^{\dagger}]^\dagger}
\end{tikzcd}
$$

\vspace{0.2in}

$$
\begin{tikzcd}
P(X) = p(X)\cdot X
&
R(X, \CC) = \{v \suchthat vp(X) = v \}
\\
p \arrow[mapsto]{u} \arrow[mapsto]{ur} \arrow[mapsto]{r}
&
F(X, \CC_{0}) = \{v \suchthat vp(X) = 0  \}
\end{tikzcd}
\hspace{0.15in}
\begin{tikzcd}
P(X) = F(X, \CC_{0})^0
&
J_1^\dagger \circ F
\\
p = F^\star(t)
&
F \arrow[mapsto]{l} \arrow[mapsto]{lu} \arrow[mapsto]{u}
\end{tikzcd}
$$
\caption{The isomorphisms of Theorem \ref{4functors}.}
\end{figure}

For reference, we describe all twelve natural isomorphisms of Theorem \ref{4functors} in Figure 1. To simplify expressions, we suppress the canonical isomorphisms $`\BB \iso \mathbf 1 \uplus \mathbf 1$, $\CC_{1} \iso \CC$, and $\CC_{0} \iso \CC$. We use the notation $(\, \cdot \,)^0$ for the polar of a subspace: For each $X \atomof \X$, the Hilbert space $X^*= L(X, \CC)$ is the Hilbert space dual of $X$. For $H \leq X$ and $K \leq X^*$ we write:
\begin{align*}H^0 = \{ v \in X^*\suchthat v(x) = 0 \text{ for all } x \in H \} \hspace{0.3in} K^0 = \{x \in X \suchthat v(x)= 0 \text{ for all } v \in K\}
\end{align*}
\begin{align*}
H^\perp = \{ x \in X \suchthat \< x'| x \> = 0 \text{ for all } x' \in H  \}
\hspace{0.2in}
K^\perp = \{ v \in X^* \suchthat v v'^\dagger = 0 \text{ for all } v' \in K \}
\end{align*}
For any function $F$ from $\X$ to $`\BB$, and each $X \atomof \X$, we have $F(X, \CC_{1}) = F(X, \CC_{0})^\perp$, since under the identifications $\CC_{1} \iso \CC$ and $\CC_{0} \iso \CC$, inner products from the subspaces $F(X, \CC_{1})$ and $F(X, \CC_{0})$ of $L(X, \CC)$ vanish, while outer products contain the identity $1_{L(X, \CC)}$. The notation $[R_1, R_2]$ refers to the universal property of the disjoint union as the coproduct of $\qRel$, and $J_1$ is the inclusion of the first summand $`\{1\} = \Q\{\CC_{1}\}$ into $`\BB$.

\section{The corange of a partial function}\label{section 13}

\setcounter{equation}{0}

\begin{definition}
Let $G$ be a partial function from a quantum set $\X$ to a quantum set $\Y$. The \emph{corange} of $G$ is the predicate $G^\star(T_\Y)$ on $\X$, where $T_\Y$ is the maximum predicate on $\Y$.
\end{definition}

Intuitively, the corange of $G$ is its domain of definition, but I prefer to reserve the term ``domain'' for the source object $\X$ of $G$. Following the natural isomorphisms of Propositions \ref{PredRel} and \ref{RelProj}, the corange of $G$ corresponds to the projection $G^\star(1)$, as both $T_\Y \in \mathrm{Pred}(\Y)$ and $1 \in \mathrm{Proj}(\ell^\infty(\Y))$ are the top elements of these two partial orders. Thus, $G$ is a function if and only if its corange is $T_\X$ (Theorem \ref{Q10}).

\begin{proposition}
Let $G\: \X \To \Y$ be a partial function, and let $F\: \Y \To \Z$ be a function. Then, the corange of $F \circ G$ is equal to the corange of $G$.
\end{proposition}

\begin{proof}
$(F \circ G)^\star(1) = G^\star(F^\star(1)) = G^\star(1)$
\end{proof}

Thus, for each predicate $P$ on a quantum set $\X$, we have a category of partial functions out of $\X$ with corange $P$, whose morphisms are postcomposition by functions. Examining this category from $\Mstar0$, we see that it has an initial object, corresponding to the homomorphism $\phi_1\: \CC \To \ell^\infty(\X)$ mapping $1$ to $p$, and it has a terminal object, corresponding to the inclusion  $\phi_0\: {p \cdot \ell^\infty(\X) \cdot p} \hookrightarrow \ell^\infty(\X)$, where $p$ is the projection corresponding to $P$.

\begin{definition}
Let $P$ be a predicate on a quantum set $\X$. Define $R_P$ to be the binary relation from $\X$ to $\mathbf 1$ corresponding to $P$ under the natural isomorphism of Proposition \ref{PredRel}.
\end{definition}

\begin{proposition}

Let $G$ be a partial function from a quantum set $\X$ to a quantum set $\Y$ with corange $P$. Then, $R_P$ factors uniquely through $G$ via a function.
$$
\begin{tikzcd}
\X \arrow{rd}[swap]{G} \arrow{rr}{R_P}
&
&
\mathbf 1
\\
&
\Y \arrow[dotted]{ur}[swap]{!}
&
\end{tikzcd}
$$
\end{proposition}

\begin{proof}
In light of the contravariant equivalence between $\qPar$ and $\Mstar0$, it is enough to show that $R_P^\star(1) = G^\star(1)$. The projection $R_P^\star(1)$ corresponds to $P$ via the natural isomorphisms of Propositions \ref{PredRel} and \ref{RelProj}, by the latter proposition, and as we have already observed, $G^\star(1)$ corresponds to the corange of $G$. 
\end{proof}

\begin{definition}\label{part}
Let $P$ be a predicate on a quantum set $\X$. Assume that $P(X) \neq P(X')$ for distinct atoms $X, X' \atomof \X$. Define the quantum set $\P$ by $\P = \Q\{P(X) \suchthat X \atomof \X \}$. For each $X \atomof \X$, let $u_X \in L(P(X), X)$ be the inclusion isometry. Define the binary relation $K_P$ from $\X$ to $\P$ by
$K_P(X, P(X)) = \CC \cdot u^\dagger_X$
for $X \atomof \X$, with the other components of $K_P$ vanishing.
\end{definition}

Because each inclusion operator $u_X$ is an isometry, $K_P \circ K_P^\dagger = I_\P$, so $K_P$ is a surjective partial function from $\X$ to $\P$. Intuitively, $K_P^\dagger$ is the inclusion of $\P$ into $\X$, but in general $\P$ is not a subset of $\X$, and $K_P^\dagger$ is not a function. Our assumption on the predicate $P$ ensures that the atoms of $\P$ are in one-to-one correspondence with those atoms of $\X$ on which $P$ is nonzero. In general, distinct atoms of $\X$ may have equal subspaces.

\begin{proposition}\label{kp}
Same assumptions as of Definition \ref{part}. Let $G$ be a partial function from a quantum set $\X$ to a quantum set $\Y$ with corange $P$. Then, $G$ factors uniquely through $K_P$ via a function. 
$$
\begin{tikzcd}
\X \arrow{r}{K_P} \arrow{rrd}[swap]{G}
&
\P \arrow[dotted]{rd}{!}
&
\\
&
&
\Y 
\end{tikzcd}
$$
\end{proposition}

\begin{proof}
For all $b \in \ell^\infty(\P)$, and all $X \atomof \X$, we have that $K_P^\star(b)(X) = u_Xb(P(X))u_X^\dagger$ (Proposition \ref{Q12}).
Since each operator $u_X$ is an isometery, $K_P^\star$ is injective, and furthermore, the image of $K_P^\star$ is exactly $p \cdot \ell^\infty(\X) \cdot p$, where $p$ is the projection defined by $p(X) = u_X u_X^\dagger = \mathrm{proj}_{P(X)}$ for $X \atomof \X$. Thus, $K_P^\star$ factors through the inclusion $p \cdot \ell^\infty(\X) \cdot p \hookrightarrow \ell^\infty(\X)$ via an isomorphism of von Neumann algebras. The projection $p$ corresponds to the predicate $P$ under the isomorphism $\mathrm{Proj}(\ell^\infty(\X)) \iso \mathrm{Pred}(\X)$ described in appendix \ref{section 12}, so $G^\star(1) = p$. The proposition now follows by the contravariant equivalence of Theorem \ref{Q10}, from the universal property of the inclusion  $p \cdot \ell^\infty(\X) \cdot p \hookrightarrow \ell^\infty(\X)$ expressed in the following commutative diagram:
\begin{equation*}
\raisebox{-5ex}{\phantom{llllllllllllllllllllllllllllllllllM}}
\begin{tikzcd}
\ell^\infty(\X) 
&
p \cdot \ell^\infty(\X)\cdot p \arrow[hook']{l}
&
\\
&
&
\ell^\infty(\Y) \arrow{llu}{G^\star} \arrow[dotted]{lu}[swap]{!} 
\end{tikzcd}\raisebox{-5ex}{\phantom{lllllllllllllllllllllllllllllllM}\qedhere}
\end{equation*}
\end{proof}

\section{Material quantum sets}\label{section 14}

\setcounter{equation}{0}

The definition of quantum sets given in section \ref{section 2} is intuitively correct up to a weak equivalence of categories, but it is nonetheless a bit sloppy. A couple of the resulting blemishes are mentioned in that section. First, the one-dimensional atoms of a quantum set need not correspond to actual elements, though they intuitively should. Second, the equation $`S \times `T = `(S \times T)$ for ordinary sets $S$ and $T$ is only true modulo natural isomorphism, though intuitively it should be exactly true.

We also have an odd caveat in Definition \ref{part} and Proposition \ref{kp}. This caveat is necessary because distinct atoms of a quantum set may nevertheless have a nonzero subspace in common. This unpleasant phenomenon necessitates our definition of predicates as functions assigning subspaces to atoms, rather than as quantum sets in their own right. It could have been avoided by simply requiring that distinct atoms of a quantum set be disjoint in Definition \ref{A}. This requirement would have no effect on the arguments that follow it, apart from compelling some discussion of the set-theoretic details of various constructions. I judged the discussion of these technicalities to be harmful to the expository goals of the article.

Our distinction between quantum sets and predicates is appropriate to the structural, i.e., category-theoretic approach we have taken. In this section, we provide alternative definitions that express a more material conception of quantum sets. Roughly speaking, we identify quantum sets with predicates on a universe of uratoms. Formally, we change the definition of quantum sets in a way that preserves the soundness of the arguments in this article. We omit the proofs, which are tedious and straightforward.

The most significant change is replacing the dagger compact category $\FdHilb$ of finite-dimensional Hilbert spaces and linear operators with a weakly equivalent dagger compact category $\FdHilb'$, which we now define.

\begin{definition}
On objects, $\FdHilb'$ has the following structure:
\begin{enumerate}
\item An object of $\mathbf{FdHilb}'$ is a \emph{subspace} of $\ell^2(M)$, where $M$ is any finite set.
\item The tensor product of two such objects $X \leq \ell^2(M)$ and $Y \leq \ell^2(N)$ is defined by
\begin{equation*}X \otimes Y = \mathrm{span}\{ (m,n) \mapsto x(m)\cdot y(n) \suchthat x \in X,\, y \in Y\} \leq \ell^2(M \times N).
\end{equation*}
\item Each object is its own canonical dual, i.e., $\ell^2(M)^* = \ell^2(M)$ by definition.
\end{enumerate}
\end{definition}

As usual, we assume that $\ell^2(M)$ and $\ell^2(N)$ are disjoint as sets whenever $M$ and $N$ are distinct, as a consequence of our formalization of mathematics in set theory. Each object of the form $\ell^2(M)$ is essentially a finite-dimensional Hilbert space, equipped with an orthonormal basis. Consequently, our morphisms are essentially matrices.

\begin{definition}
On morphisms, $\FdHilb'$ has the following structure:
\begin{enumerate}
\item A morphism of $\mathbf{FdHilb}'$ from $X \leq \ell^2(M)$ to $Y \leq \ell^2(N)$ is an element of $Y \otimes X$.
\item If $a$ is a morphism from $X \leq \ell^2(M)$ to $Y \leq \ell^2(N)$, and $b$ is a morphism from $Y \leq \ell^2(N)$ to $Z \leq \ell^2(P)$, then their composition is given by
\begin{equation*}(b\circ a)(p, m) = \sum_{n \in N} b(p,n) \cdot a(n,m).\end{equation*}
\item If $a_i$ is a morphism from $X_i \leq \ell^2(M_i)$ to $Y_i \leq \ell^2(N_i)$, for $i\in \{1,2\}$, then the tensor product of $a_1$ and $a_2$ is defined by
\begin{equation*}(a \tensor b)((n_1, n_2), (m_1, m_2)) = a_1(n_1,m_1) \cdot a_2(n_2,m_2).\end{equation*}
\item If $a$ is a morphism from $X \leq \ell^2(M)$ to $Y \leq \ell^2(N)$, then we define its adjoint and transpose by
\begin{equation*} a^\dagger(m,n) = \overline{a(n,m)} \qquad \qquad \text{and} \qquad \qquad a^*(m,n) = a (n,m).\end{equation*}
\end{enumerate}
\end{definition}

\begin{proposition}
The category $\FdHilb'$ is dagger compact, and it is weakly equivalent to the dagger compact category $\FdHilb$. The equivalence is identity on objects. The equivalence takes each morphism $a$ from $X \leq \ell^2(M)$ to $Y \leq \ell^2(N)$ to the linear operator defined by \begin{equation*}f \mapsto \left(n \mapsto \sum_{m \in M} a(n, m) \cdot f(m)\right).
\end{equation*}
\end{proposition}

\begin{definition}[in place of Definition \ref{A}]\label{material}
A \emph{quantum set} $\X$ is completely determined by a set $\At(\X)$ of nonzero Hilbert spaces in $\FdHilb'$ that are pairwise disjoint.
\end{definition}

Note that $X\leq \ell^2(M)$ and $Y \leq \ell^2(N)$ are disjoint if and only if $M \neq N$, because $X$ consists of functions with domain $M$, and $Y$ consists of functions with domain $N$. Thus, quantum sets are now in bijective correspondence with predicates of small support on a ``quantum class'' $\U$, where $\At(\U) = \{\ell^2(M) \suchthat \text{$M$ is a finite set}\}$. 

\begin{proposition}
For each quantum set $\X$, and each Hilbert space $H$, define $\X(H)$ to be the unique atom of $\X$ that is a subspace of $H$, if such an atom exists, and define $\X(H)$ to be the zero subspace of $H$, if such an atom does not exist.
\end{proposition}

This defines a bijection between quantum sets and those predicates $P$ on $\U$ in the sense of Definition \ref{qur} that have small support in the sense that $P(H) = 0$ for all but set-many $H \in \At(\U)$. Thus, like ordinary sets, quantum sets don't quite form an ortholattice.

\begin{definition}
Let $\X$ and $\Y$ be quantum sets. Define $\Y \leq \X$ iff $\Y(H) \leq \X(H)$ for all $H \in \At(\U)$. Furthermore, define $\X \vee \Y$ and $X \wedge \Y$, and $X \setminus \Y$ if $\X \geq \Y$, by
\begin{enumerate}
\item $(\X \vee \Y)(H) = \X(H) \vee \Y(H)$,
\item $(\X \wedge \Y)(H) = \X(H) \wedge \Y(H)$, and
\item $(\X \setminus \Y)(H) = \X(H) \wedge \Y(H)^\perp$,
\end{enumerate}
for all $H \in \At(\U)$.
\end{definition}

\begin{definition}[in place of Definition \ref{qur}]\label{mpred}
For each quantum set $\X$, define $\mathrm{Pred}'(\X) = \{\P \leq \X\}$.
\end{definition}

The construction $\P \mapsto (X \mapsto \P(X))$ defines an isomorphism of ortholattices from the predicates on $\X$ as defined above, to the predicates on $\X$ as in Definition \ref{qur}. We may implement a similar identification for binary relations.

\begin{definition}[in place of Definition \ref{D}]\label{mrel}
For each quantum set $\X$ and each quantum set $\Y$, define $\mathrm{Rel}'(\X; \Y) = \{\R \leq \Y \times \X\}$.
\end{definition}

The construction $\R \mapsto ((X,Y) \mapsto \R(Y \otimes X))$ defines an isomorphism of ortholattices from the binary relations from $\X$ to $\Y$ as defined above, to the binary relations from $\X$ to $\Y$ as in Definition \ref{D}, provided that $L(X,Y)$ is identified with the space of morphisms from $X$ to $Y$ in the category $\FdHilb'$. Proceeding as in section \ref{section 3}, we obtain a category $\qRel'$, each of whose morphisms is also one of its objects. Echoing the formalization of mathematics in set theory, everything is a quantum set: each function, each predicate, and each binary relation. With the exception of our definition of union, we may proceed with these quantum sets and binary relations as before.

The trouble with Definition \ref{B}(4) of the union of two quantum sets is that it is no longer guaranteed to produce a quantum set. Indeed it is neither fully at home in the material approach taken in this appendix, nor in the structural approach taken in the rest of the article. Definition \ref{B}(4) produces a quantum set in the sense of this appendix if and only if both $\X$ and $\Y$ are subsets of $\X \vee \Y$, and in that case $\X \union \Y = \X \vee \Y$. It is therefore reasonable to use that notation in that case. However, apart from Definition \ref{B}(4), which serves only an expository purpose, the rest of the arguments remain sound.

\begin{theorem}
Every result in all preceding sections and appendices is correct using Definition \ref{material} in place of Definition \ref{A}, Definition \ref{mpred} in place of Definition \ref{qur}, Definition \ref{mrel} in place of Definition \ref{D}, and $\FdHilb'$ in place of $\FdHilb$.
\end{theorem}

\begin{proof}
Every argument remains valid just as it is.
\end{proof}

In addition to turning our morphisms into quantum sets, this definition also tightens the correspondence between ordinary sets and classical quantum sets.

\begin{proposition}\label{Z5}
The functor $`(\,\cdot \,) \: \Rel \To \qRel'$ is an isomorphism of dagger compact categories from $\Rel$ onto the full subcategory of $\qRel'$ consisting of classical quantum sets.
\end{proposition}

In other words, the construction $S \mapsto `S$ is a one-to-one correspondence between ordinary sets and quantum sets whose atoms are all one-dimensional, and it exactly preserves all dagger compact structure. In particular, for all ordinary sets $S$ and $T$, we have $`S \times `T = `(S \times T)$. For concreteness, we may take the monoidal unit of $\Set$ to be the set $1 = \{\emptyset\}$, so the monoidal unit of $\qSet$ is $\mathbf 1 = `1 = \Q\{\CC_\emptyset\} = \Q\{\ell^2(\{\emptyset\})\}$. Thus, where we speak of the Hilbert space $\CC$, we mean $\CC_\emptyset$.

\section*{Acknowledgements} The appearance of \cite{MustoReutterVerdon} spurred me to finally put pen to paper; I thank David Reutter and Dominic Verdon for useful discussion. I thank David Roberson for his detailed comments about the literature in quantum pseudotelepathy, which is new to me. I thank Neil Ross and Peter Selinger for organizing QPL 2018, an enriching experience that motivated the discussion of quantum pseudotelepathy in this introductory section. I thank an anonymous referee for advocating the inclusion of ``fissions'' in sections \ref{section 6} and \ref{section 7}. Finally, I am deeply grateful to Bert Lindenhovius and Michael Mislove for their confidence in this approach to quantization, and for many hours of enriching discussion.

\end{document}